\documentclass{amsart}

\usepackage[linktoc=page]{hyperref}

\usepackage[utf8]{inputenc}
\usepackage{cite}
\usepackage{amssymb}
\usepackage{amsthm}
\usepackage{amsfonts}
\usepackage{amsmath}
\usepackage{mathrsfs}
\usepackage[all]{xy}
\usepackage{graphicx}
\usepackage{mathrsfs}
\usepackage{extpfeil}
\usepackage{mathtools}
\usepackage{tikz-cd}
\usepackage{bm}
\usepackage{latexsym, amscd ,psfrag}
\usepackage{graphicx}
\usepackage{caption}
\usepackage{float}

\usepackage[mathscr]{euscript}
\usepackage{enumitem}
\usepackage{cleveref}

\makeatletter
\newsavebox{\@brx}
\newcommand{\llangle}[1][]{\savebox{\@brx}{\(\m@th{#1\langle}\)}%
  \mathopen{\copy\@brx\kern-0.5\wd\@brx\usebox{\@brx}}}
\newcommand{\rrangle}[1][]{\savebox{\@brx}{\(\m@th{#1\rangle}\)}%
  \mathclose{\copy\@brx\kern-0.5\wd\@brx\usebox{\@brx}}}
\makeatother

\newcommand{\ep}{\epsilon}
\newcommand{\si}{\sigma}
\newcommand{\Si}{\Sigma}
\newcommand{\Ga}{\Gamma}

\newcommand{\bC}{\mathbb{C}}
\newcommand{\bD}{\mathbb{D}}

\newcommand{\bL}{\mathbb{L}}
\newcommand{\bP}{\mathbb{P}}
\newcommand{\bQ}{\mathbb{Q}}
\newcommand{\bR}{\mathbb{R}}

\newcommand{\bZ}{\mathbb{Z}}

\newcommand{\cE}{\mathcal{E}}

\newcommand{\cL}{\mathcal{L}}

\newcommand{\cM}{\mathcal{M}}
\newcommand{\cO}{\mathcal{O}}

\newcommand{\cS}{\mathcal{S}}

\newcommand{\cT}{\mathcal{T}}

\newcommand{\ch}{\mathrm{ch}}
\newcommand{\Hom}{\mathrm{Hom}}

\newcommand{{\inv} }{\mathrm{inv}}
\newcommand{\ev}{\mathrm{ev}}
\newcommand{\Aut}{\mathrm{Aut}}

\newcommand{\Res}{\mathrm{Res}}

\newcommand{\val}{ {\mathrm{val}} }
\newcommand{\vir}{{\mathrm{vir}}}

\newcommand{\SYZ}{\mathrm{SYZ}}

\newcommand{\one}{\mathbf{1}}

\newcommand{\bu}{\mathbf{u}}

\newcommand{\bp}{\mathbf{p}}
\newcommand{\bt}{\mathbf{t}}

\newcommand{\su}{\mathsf{u}}
\newcommand{\sv}{\mathsf{v}}
\newcommand{\sw}{\mathsf{w}}

\newcommand{\nov}{\Lambda_{\mathrm{nov}}}

\newcommand\fh{\mathfrak{h}}

\newcommand{\CP}{\bP^1}
\newcommand{\RP}{\bR\bP^1}
\newcommand{\OGW}{\CP,\RP,T}

\newcommand{\textOGW}{\text{OGW}}
\newcommand{\forget}{\text{for}}
\usetikzlibrary{arrows.meta, bending}

\newtheorem{lma}{Lemma}[section]
\newtheorem{coro}[lma]{Corollary}
\newtheorem{defn}[lma]{Definition}
\newtheorem{prop}[lma]{Proposition}

\newtheorem{theorem}[lma]{Theorem}
\newtheorem{remark}[lma]{Remark}

\theoremstyle{definition}
\newtheorem{definition}[lma]{Definition}
\newtheorem{convention}[lma]{Convention}
\newtheorem{conjecture}[lma]{Conjecture}

\begin{document}

\title{All genus open mirror symmetry for the projective line}

\author{Jinghao Yu}
\address{Jinghao Yu, Department of Mathematical Sciences, Tsinghua University, Haidian District, Beijing 100084, China}
\email{yjh21@mails.tsinghua.edu.cn}

\author{Zhengyu Zong}
\address{Zhengyu Zong, Department of Mathematical Sciences,
	Tsinghua University, Haidian District, Beijing 100084, China}
\email{zyzong@mail.tsinghua.edu.cn}

\maketitle

\begin{abstract}
	We prove that the Chekhov-Eynard-Orantin recursion on the mirror curve of $\bP^1$ encodes all genus equivariant open Gromov-Witten invariants of $(\bP^1,\bR\bP^1)$. This result can be viewed as an all genus equivariant open mirror symmetry for $(\bP^1,\bR\bP^1)$.
\end{abstract}

\tableofcontents

\section{Introduction}

\subsection{Historical background and motivation}
\subsubsection{All genus closed mirror symmetry for the projective line}
Mirror symmetry is a duality from string theory originally discovered by physicists. It says two
dual string theories -- type IIA and type IIB -- on different Calabi-Yau 3-folds give rise to the same
physics. Mathematicians became interested in this relationship around
1990 when Candelas, de la Ossa, Green, and Parkes \cite{CdGP} obtained
a conjectural formula of the number of rational curves of arbitrary
degree in the quintic 3-fold by relating it to period integrals
of the quintic mirror.

By late 1990s mathematicians had
established the foundation of Gromov-Witten (GW) theory as a mathematical theory of A-model topological closed strings. In this context, the genus $g$ free energy of the topological A-model is
defined as a generating function of genus $g$ Gromov-Witten invariants. The mathematical aspect of genus zero closed mirror symmetry has been well-studied in many cases. Givental \cite{G96} and Lian-Liu-Yau \cite{LLY97} independently proved
the genus zero mirror formula for the quintic Calabi-Yau 3-fold $Q$; later they
extended their results to Calabi-Yau complete intersections in
projective toric manifolds \cite{G98, LLY99, LLY3}. The genus-zero mirror theorem for toric Deligne-Mumford stacks is proved in \cite{CCIT, CCK}.

The higher genus mirror symmetry is much more subtle. Bershadsky-Cecotti-Ooguri-Vafa (BCOV) conjectured the genus-one and genus-two mirror formulae for the quintic 3-fold \cite{BCOV}. Combining the techniques of BCOV, results of Yamaguchi-Yau \cite{YY04}, and boundary conditions,
Huang-Klemm-Quackenbush \cite{HKQ} proposed a mirror conjecture on $F_g^Q$ up to $g=51$. The BCOV genus-one  mirror formula was first proved
by A. Zinger in \cite{Zi09} using genus-one reduced Gromov-Witten theory, and later reproved in \cite{KimL, CFKim} via quasimap theory and
in \cite{CGLZ} via MSP theory. The BCOV genus-two mirror formula was proved by Guo-Janda-Ruan \cite{GJR} and Chang-Guo-Li \cite{CGL}. The Yamaguchi--Yau's finite generation and the holomorphic anomaly equation for quintic 3-folds were proved in \cite{CGL} and \cite{GJR2}.

The all genus closed mirror symmetry for the projective line is easier to study and there is a nice mathematical theory for the higher genus B-model. P. Dunin-Barkowski, N. Orantin, S. Shadrin and L. Spitz \cite{DOSS} relate the Chekhov-Eynard-Orantin topological recursion to the higher genus closed Gromov-Witten invariants of $\bP^1$, proving the Norbury-Scott conjecture \cite{NS}. In \cite{FLZ17}, the above result is generalized to the equivariant setting. The main result of \cite{FLZ17} relates the higher genus equivariant descendant Gromov-Witten potentials of $\bP^1$ to the oscillatory integrals of Chekhov-Eynard-Orantin invariants of the mirror curve. This result can be viewed as an all genus equivariant closed mirror symmetry for $\bP^1$. By taking the nonequivariant limit, the Norbury-Scott conjecture is recovered.

\subsubsection{Open Gromov-Witten theory for the projective line}
The open Gromov-Witten invariants are more difficult to study than the closed case. The all genus $S^1$-equivariant open Gromov-Witten theory of $(\bP^1,\bR\bP^1)$ is studied via coherent boundary conditions by Buryak-Netser Zernik-Pandharipande-Tessler in \cite{BNPT22}. A localization formula is conjectured and proved in the genus zero case. The localization formula is crucial for the computation of open Gromov-Witten theory of $(\bP^1,\bR\bP^1)$. It is natural to ask whether one can extend the result in \cite{FLZ17} to the open string sector. This will be introduced in Section \ref{sec:openMS}.

\subsubsection{All genus open mirror symmetry for the projective line}\label{sec:openMS}
In this paper, we prove the all genus open mirror symmetry for $(\bP^1,\bR\bP^1)$, extending the result in \cite{FLZ17} to the open string sector. On A-model side, we will study the open Gromov-Witten invariants of $(\bP^1,\bR\bP^1)$ from two different but related points of view. From the first point of view, we define the higher genus open Gromov-Witten invariants of $(\bP^1,\bR\bP^1)$ \textit{via integration over the fixed locus} of the $S^1$-action and consider the corresponding generating functions in Section \ref{sec:openGW}. On B-model side, instead of taking the oscillatory integrals, we take the expansion of the Chekhov-Eynard-Orantin invariants under certain local coordinate at the puncture of the mirror curve. The first main theorem (Theorem \ref{thm:main}) of this paper states that this expansion coincides with the generating function of higher genus open Gromov-Witten invariants of $(\bP^1,\bR\bP^1)$ under the mirror map. The second main theorem (Theorem \ref{thm:main-ii}) generalizes the above result to the case when descendant insertions appear.

From the second point of view, we define the higher genus open Gromov-Witten invariants of $(\bP^1,\bR\bP^1)$ as a graph sum in Section \ref{sec:comb-OGW} and Section \ref{sec:gen-geo}. The vertex factor in this graph sum formula is given by the open Gromov-Witten invariants via integration over the fixed locus defined in Section \ref{sec:openGW}. In other words, the open Gromov-Witten invariants via integration over the fixed locus are the building blocks in the second point of view of the open Gromov-Witten invariants of $(\bP^1,\bR\bP^1)$. In \cite{BNPT22}, this graph sum formula is conjectured to have a geometric definition via coherent boundary conditions. So we call the second point of view the open Gromov-Witten invariants \textit{via coherent boundary conditions}. The third main theorem (Theorem \ref{thm:main-iii}) of this paper is the all genus mirror symmetry for open Gromov-Witten invariants of $(\bP^1,\bR\bP^1)$ via coherent boundary conditions.

We would like to remark that similar open mirror symmetry phenomenon appears in the case of toric Calabi-Yau 3-folds/3-orbifolds, where the open Gromov-Witten invariants are defined \textit{via integration over the fixed locus}. Based on the work of Eynard-Orantin \cite{EO07} and Mari\~{n}o
\cite{Ma}, Bouchard-Klemm-Mari\~{n}o-Pasquetti  (BKMP) \cite{BKMP09, BKMP10} proposed a new formalism of the topological B-model in terms of the Chekhov-Eynard-Orantin invariants of the mirror curve. BKMP conjectured a precise correspondence, known as the BKMP Remodeling Conjecture, between the local expansion of the Chekhov-Eynard-Orantin invariants at the puncture of the mirror curve and the generating function of open Gromov-Witten invariants of toric Calabi-Yau 3-folds/3-orbifolds. The mathematical study of the Remodeling Conjecture can be found, for example, in \cite{BCMS,Ch09,EO15,FL,FLT,FLZ20a,FLZ20b,Zh09,Zh09b,Zh10,Zhu}. The difference in our case is that our target space $\bP^1$ is compact and non-Calabi-Yau.

We emphasize the following feature of our main results. Since $\bP^1$ and $\bR\bP^1$ are compact, the non-equivariant limit of the $S^1$-equivariant open Gromov-Witten invariants of $(\bP^1,\bR\bP^1)$ defined in Section \ref{sec:geo-open} is conjectured \cite[Conjecture 1]{BNPT22} to be equal to the non-equivariant open Gromov-Witten invariants of $(\bP^1,\bR\bP^1)$, defined geometrically via coherent boundary conditions. In general, the computation of higher genus open Gromov-Witten invariants is quite difficult. The main results of our paper provide an effective algorithm of computing higher genus open Gromov-Witten invariants of $(\bP^1,\bR\bP^1)$ recursively.

\subsection{Statement of the main result}
In this paper, we denote the complex projective line by $\bP^1$. Let $t\in T=S^1$ act on $\bP^1$ by
$$
	t\cdot [z_1,z_2] = [tz_1,  t^{-1}z_2].
$$
Let $\bC[\sv]=H_{T}^*(\mathrm{point};\bC)$ be
the $T$-equivariant cohomology of a point. Let $p_1=[1,0]$ and $p_2=[0,1]$ be the $T$ fixed points. The $T$-equivariant cohomology of $\bP^1$ is given by
$$
H^*_{T}(\bP^1;\bC)= \bC[H,\sv]/\langle (H+\sv/2)(H-\sv/2)\rangle
$$
where $\deg H=\deg \sv=2$. Let 
$$   
L := \{[e^{{\rm i}\varphi},e^{-{\rm i}\varphi}]\in\CP: \varphi\in\bR\}
$$
be the Lagrangian submanifold of $\CP$, which is preserved by the $T$-action. By taking a M\"obius transform, we can identify the pair $(\CP,L)$ with $(\CP,\RP)$.

In Section \ref{sec:gen.fun.}, we will define the generating function $F_{g,n}(\bt,Q;X_1,\dots,X_n)$ of genus $g$, $n$ boundary circles open Gromov-Witten invariants of $(\CP,\RP)$ via integration over the fixed locus of the $S^1$-action. 
Here $\bt=t^01 + t^1H$, $Q$ is the Novikov variable encoding the degree of the stable maps to $\bP^1$, and $X_1,\cdots,X_n$ are variables encoding the winding numbers (viewed as the open string coordinates).
We will also define the generating function $\llangle \cdots \rrangle_{g,m,n}^{\OGW}(\bt,X_1,\dots,X_n)$ of genus $g$, $m$ descendant insertion classes, $n$ boundary circles open Gromov-Witten invariants of $(\CP,\RP)$. 

In Section \ref{sec:curve}, we will study the mirror curve of $\bP^1$. Consider the $T$-equivariant superpotential $W:\bC^*\rightarrow \bC$ defined as
\[
W(Y) = t^0 + \frac{\sv}{2}\log q + Y + \frac{q}{Y} - \sv\log Y,
\]
where $q=e^{t^1}$. Let $X= e^{-x/\sv}$, $Y=e^y$. Consider the mirror curve $C_q$ defined as follows:
\[
C_q = \{(x,y)\in\bC^2: x = W(e^y)\}.
\]
In Section \ref{sec:TR}, we will study the Chekhov-Eynard-Orantin invariants $\omega_{g,n}$ of the mirror curve $C_q$. 
Then we use $\omega_{g,n}$ to define the B-model open primary potential $W_{g,n}(t^0,q;X_1,\dots,X_n)$. 
In Section \ref{sec:SYZmirror}, we study the 
oscillatory integrals of $\omega_{g,n}$ and define the B-model open descendant potential $W_{g,m,n}(t^0,q;\cE_1,\dots,\cE_m;X_1,\dots,X_n)$, where $\cE_1,\dots,\cE_m\in K_T(\CP)$.

The following three theorems are the main results of this paper:
\begin{theorem}[= Theorem \ref{thm:stableMS}]\label{thm:main}  For any $n>0$ and $g\geq 0$, we have
	\[
	F_{g,n}(\bt,1;X_1,\dots,X_n) = W_{g,n}(t^0,q;X_1,\dots,X_n).
	\]
\end{theorem}

\begin{theorem}[= Theorem \ref{thm:SYZ-open-mirror}]\label{thm:main-ii}
    Suppose $g,m,n\geq 0$, for any $\cE_1,\dots,\cE_m\in K_T(\CP)$, we have
    \begin{align*}
        &W_{g,m,n}(t^0,q;\cE_1,\dots,\cE_m;X_1,\dots,X_n)
        \\
        &= \llangle\frac{\kappa_{z_1}(\cE_1)}{z_1-\psi_1},\dots,\frac{\kappa_{z_m}(\cE_m)}{z_m-\psi_m}\rrangle_{g,m,n}^{\OGW}(\bt,X_1,\dots,X_n).
    \end{align*}
    Here the characteristic class $\kappa_{z}$ is defined in Section \ref{sec:SYZmirror} and $z_1,\cdots,z_m$ are formal variables. The dependence of $W_{g,m,n}$ on $z_1,\cdots,z_m$ is via oscillatory integrals defined in Section \ref{sec:SYZmirror}.
\end{theorem}

In Section \ref{sec:geo-counting} and \ref{sec:comb-OGW}, we will study the moduli specification $S$ and the combinatorial graph sum formula of open Gromov-Witten invariants $\textOGW(S,\vec{a},\vec{\gamma})$,
where the vertex factor is given by the open Gromov-Witten invariants via integration over the fixed locus defined in Section \ref{sec:openGW}. In \cite{BNPT22}, $\textOGW(S,\vec{a},\vec{\gamma})$ is conjectured to have a geometric definition via coherent boundary conditions.
We will define the decorated graph $\Ga_S$ and the generating function $\textOGW(\Ga_S,\vec{a},\vec{\gamma},\bt)$ in Section \ref{sec:gen-geo}. In Section \ref{sec:B-geo-count}, we define $\textOGW_{g,m,n}(\Ga_S, \frac{\kappa_{z_i}(\cE_i)}{z_i-\psi_i},\bt)$ as a function of $z_i$ composed of $\textOGW(\Ga_S,\vec{a},\vec{\gamma},\bt)$. 
The B-model potential $G_{g,m,n}(\Ga_S,\bt)$ is defined in Section \ref{sec:B-geo-count} via graph sum formula, where the building blocks are Chekhov-Eynard-Orantin invariants. 
The third main theorem is the mirror symmetry for $\textOGW_{g,m,n}(\Ga_S, \frac{\kappa_{z_i}(\cE_i)}{z_i-\psi_i},\bt)$.
\begin{theorem}[= Theorem \ref{thm:geo-MS}]\label{thm:main-iii}
    \rm Let $S$ be a pure connected moduli specification with one black vertex, $n$ white vertices and $m$ marked labels ${\sf I}$. Then we have
    $$
        \textOGW_{g,m,n}(\Ga_S, \frac{\kappa_{z_i}(\cE_i)}{z_i-\psi_i},\bt) = G_{g,m,n}(\Ga_S,\bt).
    $$
\end{theorem}

\subsection{Overview of the paper}
In Section \ref{sec:closedGW}, we study the equivariant closed Gromov-Witten theory of $\bP^1$ and give the graph sum formula for the descendant Gromov-Witten potential. In Section \ref{sec:openGW}, we define the equivariant open Gromov-Witten theory of $(\CP,\RP)$ via integration over the fixed locus of the $S^1$-action and define the A-model potential $F_{g,n}$ and $\llangle\frac{\kappa_{z_1}(\cE_1)}{z_1-\psi_1},\dots,\frac{\kappa_{z_m}(\cE_m)}{z_m-\psi_m}\rrangle_{g,m,n}^{\OGW}$. In Section \ref{sec:Bmodel}, we study the mirror curve of $\bP^1$ and the Chekhov-Eynard-Orantin topological recursion. Then we use the Chekhov-Eynard-Orantin invariants $\omega_{g,n}$ to define the B-model potential $W_{g,m,n}$. In Section \ref{sec:MS}, we prove the all genus open mirror symmetry for $F_{g,n}$ and $\llangle\frac{\kappa_{z_1}(\cE_1)}{z_1-\psi_1},\dots,\frac{\kappa_{z_m}(\cE_m)}{z_m-\psi_m}\rrangle_{g,m,n}^{\OGW}$. In Section 6, we consider the open Gromov-Witten invariants $\textOGW(S,\vec{a},\vec{\gamma})$ and prove the corresponding all genus open mirror symmetry.

\subsection*{Acknowledgements}
The authors would like to thank Bohan Fang, Chiu-Chu Melissa Liu, and Song Yu for useful discussions. The second author is partially supported by the Natural Science Foundation of Beijing, China (grant No. 1252008) and NSFC (grant No. 12571067).

\section{Equivariant closed Gromov-Witten theory of $\bP^1$}\label{sec:closedGW}
\subsection{Equivariant cohomology of $\bP^1$}
Let $t\in T=S^1$ act on $\bP^1$ by
\begin{equation*}\label{eq:S1-translation}
     t\cdot [z_1,z_2] = [tz_1,  t^{-1}z_2].
\end{equation*}
The $T$-equivariant cohomology of $\bP^1$ is given by
\[
    H^*_T(\bP^1;\bC) = \bC[H,\sv]/\langle (H+\sv/2)(H-\sv/2)\rangle.
\]
Let $p_1=[1,0]$ and $p_2=[0,1]$ be the $T$-fixed points. Then $H|_{p_1}=-\sv/2$, $H|_{p_2}=\sv/2$. 
The $T$-equivariant Poincaré dual of $p_1$ and $p_2$ are $H-\sv/2$ and $H+\sv/2$, respectively.

Let 
\[
    \phi_1 := -\frac{H-\sv/2}{\sv}, \phi_2:= \frac{H+\sv/2}{\sv} \in H^*_T(\bP^1;\bC)\otimes_{\bC[\sv]}\bC(\sv).
\]
We have
\[
    \phi_1 + \phi_2 = 1, \quad \phi_\alpha\cup\phi_\beta = \delta_{\alpha\beta}\phi_{\alpha},
\]
\[
    (\phi_\alpha, \phi_\beta)_{\bP^1,T} := \int_{\bP^1}\phi_\alpha\cup\phi_\beta = \delta_{\alpha\beta}\int_{\bP^1}\phi_\alpha = \frac{\delta_{\alpha\beta}}{\Delta^\alpha},
    \quad \alpha,\beta\in\{1,2\},
\]
where $$\Delta^1=-\sv, \quad \Delta^2 = \sv.$$

For convenience, we introduce the notations $R_T := \bC[\sv]$ and ${S}_T := \bC(\sv)$, and we 
let $\bar{S}_T$ be the minimal field extension of $S_T$ containing $\{\sqrt{\Delta^\alpha}:\alpha = 1,2\}$. Then
$\{\phi_\alpha:\alpha=1,2\}$ is a canonical basis of the semisimple Frobenius algebra
\[
    (H^*_T(\bP^1;\bC)\otimes_{R_T}\bar{S}_T,\cup,(\cdot,\cdot)_{\bP^1,T})
\]
over $\bar{S}_T$.
Let 
$$\hat{\phi}_\alpha := \sqrt{\Delta^\alpha}\phi_\alpha.$$
Then $\{\hat{\phi}_\alpha:\alpha = 1,2\}$ is the normalized canonical basis of the semisimple Frobenius algebra $H^*_T(\bP^1;\bC)\otimes_{R_T}\bar{S}_T$:
\[
    \hat{\phi}_\alpha\cup\hat{\phi}_\beta =\delta_{\alpha\beta}\sqrt{\Delta^\alpha}\hat{\phi}_\alpha,\quad
    (\hat{\phi}_\alpha,\hat{\phi}_\beta)_{\CP,T} = \delta_{\alpha\beta}. 
\]

\subsection{Equivariant Gromov-Witten invariants of $\bP^1$}
Let $E(\bP^1)$ denote the effective curve classes in $H_2(\bP^1;\bZ)$.
Given nonnegative integers $g, n$ and an effective curve class $\beta\in E(\bP^1)$,
let $\overline{\cM}_{g,n}(\bP^1,\beta)$ be the moduli stack of genus $g$, $n$-pointed, degree $\beta$ stable maps to $\bP^1$.
Let $\ev_i: \overline{\cM}_{g,n}(\bP^1,\beta)\rightarrow \bP^1$ be the evaluation map at the $i$-th marked point. The $T$-action on $\bP^1$ induces an $T$-action 
on $\overline{\cM}_{g,n}(\bP^1,\beta)$ and the evaluation map $\ev_i$ is $T$-equivariant.

For $i=1,\dots,n$, let $\bL_i$ be the $i$-th tautological line bundle over $\overline{\cM}_{g,n}(\bP^1,\beta)$ formed by the cotangent line at the $i$-th marked point.
Define the $i$-th descendant class $\psi_i$ as 
$$
    \psi_i := c_1(\bL_i)\in H^2(\overline{\cM}_{g,n}(\bP^1,\beta);\bQ).
$$

Given $\gamma_1,\dots,\gamma_n\in H^*_T(\bP^1;\bC)$ and nonnegative integers $a_1,\dots, a_n$, we define genus-$g$, degree-$\beta$, $T$-equivariant descendant 
Gromov-Witten invariants of $\bP^1$:
\[
    \langle \gamma_1\psi^{a_1},\dots,\gamma_n\psi^{a_n}\rangle_{g,n,\beta}^{\bP^1,T} := \int_{[\overline{\cM}_{g,n}(\bP^1,\beta)]^{\vir,T}}
    \prod_{i=1}^n \psi_i^{a_i}\ev_i^*(\gamma_i) \in \bC[\sv],
\]
where $[\overline{\cM}_{g,n}(\bP^1,\beta)]^{\vir,T}$ is the $T$-equivariant virtual fundamental class of $\overline{\cM}_{g,n}(\bP^1,\beta)$ \cite{GP99}.

Define the Novikov ring 
\[
    \nov := \bC[\widehat{E(\CP)}] = \Bigg\{\sum_{\beta\in E(\CP)}c_\beta Q^\beta: c_\beta\in\bC\Bigg\},
\]
where $Q$ is the Novikov variable. 
Let $\bt = t^0 1+ t^1H$, we define the following double correlator:
\[
    \llangle\gamma_1\psi^{a_1},\dots,\gamma_n\psi^{a_n}\rrangle_{g,n}^{\bP^1,T} := 
    \sum_{\beta\in E(\bP^1)}\sum_{m=0}^\infty \frac{Q^\beta}{m!}\langle \gamma_1\psi^{a_1},\dots,\gamma_n\psi^{a_n}
    ,\bt^m\rangle_{g,n+m,\beta}^{\bP^1,T}.
\]

For $j=1,\dots,n$, introduce formal variables
\[
    {\bf{u}}_j = {\bf{u}}_j(z) = \sum_{a\geq 0}(u_j)_az^a
\]
where $(u_j)_a\in H^*_T(\bP^1)\otimes_{R_T}\bar{S}_T$. Define
\[
\llangle {\bf u}_1,\dots, {\bf u}_n \rrangle_{g,n}^{\bP^1,T} =  \llangle {\bf u}_1(\psi),\dots, {\bf u}_n(\psi)\rrangle_{g,n}^{\bP^1,T}
        = \sum_{a_1,\dots,a_n\geq 0}\llangle (u_1)_{a_1}\psi^{a_1},\dots, (u_n)_{a_n}\psi^{a_n}\rrangle_{g,n}^{\bP^1,T}.
\]

Let $z_1,\dots,z_n$ be formal variables and $\gamma_1,\dots,\gamma_n\in H^*_T(\bP^1)\otimes_{R_T}\bar{S}_T$. Define
\[
    \llangle \frac{\gamma_1}{z_1-\psi},\dots,\frac{\gamma_n}{z_n-\psi}\rrangle_{g,n}^{\bP^1,T} = 
    \sum_{a_1,\dots,a_n\in\bZ_{\geq 0}}\llangle\gamma_1\psi^{a_1},\dots,\gamma_n\psi^{a_n}\rrangle_{g,n}^{\bP^1,T}\prod_{i=1}^{n}z_i^{-a_i-1}.
\]
We use the conventions that
\[
    \langle\frac{\gamma}{z-\psi}\rangle_{0,1,0}^{\bP^1,T} := z\int_{\bP^1}\gamma,
\]
\[
    \langle\frac{\gamma_1}{z-\psi},\gamma_2\rangle_{0,2,0}^{\bP^1,T} := \int_{\bP^1}\gamma_1\cup\gamma_2,
\]
\[
    \langle\frac{\gamma_1}{z_1-\psi_1},\frac{\gamma_2}{z_2-\psi_2}\rangle_{0,2,0}^{\bP^1,T} := \frac{1}{z_1+z_2}\int_{\bP^1}\gamma_1\cup\gamma_2.
\]

\subsection{Equivariant quantum cohomology and Frobenius structure}
The $T$-equivariant quantum cohomology ring $QH_T^*(\bP^1)$ of $\bP^1$ is defined by its genus-zero primary Gromov-Witten invariants. As a $\bC[\sv]$-module, 
$QH_T^*(\bP^1) = H^*_T(\bP^1)$. The ring structure is given by the quantum product $\star$:
\[
    (\gamma_1\star\gamma_2,\gamma_3)_{\CP,T} = \llangle \gamma_1,\gamma_2,\gamma_3\rrangle_{0,3}^{\bP^1,T}.
\]
In other words,
\[
    \gamma_1\star \gamma_2 = \gamma_1\cup\gamma_2 + q\left(\int_{\bP^1}\gamma_1\right)\left(\int_{\bP^1}\gamma_2\right),
\]
where $\cup$ is the cup product in $H^*_T(\bP^1)$ and $q = Qe^{t^1}$. In particular,
\[
    H\star H = \sv^2/4  + q.
\]
The $T$-equivariant quantum cohomology ring of $\bP^1$ is given by
\[
    QH^*_T(\bP^1;\bC) = \bC[H,\sv,q]/\langle (H+\sv/2)\star (H-\sv/2) - q\rangle,
\]
where $\deg H=\deg \sv = 2$ and $\deg q = 4$.

Let 
\[
    \phi_1(q) = \frac{1}{2} - \frac{H}{\sv\sqrt{1+4q/\sv^2}},
\quad
    \phi_2(q) = \frac{1}{2} + \frac{H}{\sv\sqrt{1+4q/\sv^2}}.
\]
Then 
\[
    \phi_\alpha(q)\star \phi_\beta(q) = \delta_{\alpha\beta}\phi_\alpha(q), \quad
    (\phi_\alpha(q), \phi_\beta(q))_{\bP^1,T} = \frac{\delta_{\alpha\beta}}{\Delta^\alpha(q)},
\]
where
\[
    \Delta^1(q) = -\sv\sqrt{1 + 4q/\sv^2}, \quad \Delta^2(q) = -\Delta^1(q).
\]
Let $S := \bar{S}_T[[q]]$. Then
$\{\phi_1(q), \phi_2(q)\}$ is a canonical basis of the semisimple Frobenius algebra 
$$(H^*_T(\bP^1;\bC)\otimes_{R_T} S, \star,(\cdot,\cdot)_{\CP,T}).$$ 
The normalized canonical basis is defined by
\[
    \{ \hat{\phi}_\alpha(q) := \sqrt{\Delta^\alpha(q)}\phi_\alpha(q) : \alpha = 1,2\}.
\]
They satisfy
\[
    \hat{\phi}_\alpha(q) \star \hat{\phi}_\beta(q) = \delta_{\alpha\beta}\sqrt{\Delta^\alpha(q)}\hat{\phi}_\alpha(q), \quad (\hat{\phi}_\alpha(q),\hat{\phi}_\beta(q))_{\bP^1,T} = \delta_{\alpha\beta}.
\]

\subsection{Canonical coordinates and the transition matrix}
Let $\{t^0,t^1\}$ be the flat coordinates with respect to the basis $\{T_0=1,T_1=H\}$, and let $\{u^1,u^2\}$ be the canonical coordinates with respect to the basis $\{\phi_1(q),\phi_2(q)\}$.
Then 
\begin{equation*}
    \begin{aligned}
        \frac{\partial}{\partial u^1} &= \frac{1}{2}\frac{\partial}{\partial t^0} + \frac{1}{\Delta^1(q)}\frac{\partial}{\partial t^1},
        \\
        \frac{\partial}{\partial u^2} &= \frac{1}{2}\frac{\partial}{\partial t^0} + \frac{1}{\Delta^2(q)}\frac{\partial}{\partial t^1}.
    \end{aligned}
\end{equation*}
The canonical coordinates $u^1$ and $u^2$ are characterized by the above equations up to a constant in $\bC(\sv)$. We choose the canonical coordinates $(u^1,u^2)$ 
such that
\[
    \lim _{q\rightarrow 0} (u^1 - t^0 + \sv t^1/2) = 0, \quad \lim_{q\rightarrow 0}(u^2-t^0-\sv t^1/2) = 0.
\]
We use Greek alphabet $\alpha\in\{1,2\}$ to represent canonical coordinates and the Latin alphabet $i\in\{0,1\}$ to represent
flat coordinates. Let $\Psi = (\Psi_i^{\ \alpha})$ be the transition matrix between the flat and quantum normalized canonical bases:
\[
    T_i = \sum_{\alpha\in\{1,2\}}\Psi_i^{\ \alpha}\hat\phi_{\alpha}(q).
\]
It can be rewritten as
\[
    \frac{\partial}{\partial t^i} = \sum_{\alpha\in\{1,2\}}\Psi_i^{\ \alpha}\sqrt{\Delta^\alpha(q)}\frac{\partial}{\partial u^\alpha},
\]
which is equivalent to
\[
    \frac{du^\alpha}{\sqrt{\Delta^\alpha(q)}} = \sum_{i\in\{0,1\}}dt^i\Psi_i^{\ \alpha}.
\]
Then
\[
    \Psi_0^{\ \alpha} = \frac{1}{\sqrt{\Delta^\alpha(q)}},\quad
    \Psi_1^{\ \alpha} = \frac{\sqrt{\Delta^\alpha(q)}}{2}.
\]
Let $\Psi^{-1} = (\Psi^{-1})_\alpha^{\ i}$, so that
\[
    \sum_{i\in\{0,1\}}(\Psi^{-1})_\alpha^{\ i}\Psi_i^{\ \beta} = \delta_\alpha^{\ \beta}.
\]
Then 
\[
    (\Psi^{-1})_\alpha^{\ 0} = \frac{\sqrt{\Delta^\alpha(q)}}{2},\quad
    (\Psi^{-1})_\alpha^{\ 1} = \frac{1}{\sqrt{\Delta^\alpha(q)}}.
\]

\subsection{The $\cS$-operator} For any $a,b\in H^*_T(\bP^1;\bar{S}_T)$, define
the $\cS$-operator by
\[
    (a,\cS(b))_{\bP^1,T} = (a,b)_{\bP^1,T}+\llangle a,\frac{b}{z-\psi}\rrangle_{0,2}^{\bP^1,T}.
\]

We consider several different (flat) bases for $H^*_T(\bP^1;\bar{S}_T)$:
\begin{enumerate}
    \item The natural basis: $T_0=1$ and $T_1=H$.
    \item The basis dual to the natural basis: $T^0 = H$ and $T^1=1$.
    \item The classicial canonical basis $\{\phi_1,\phi_2\}$.
    \item The basis dual to the classicial canonical basis with respect to 
    the $T$-equivariant Poincaré pairing: $\phi^1 = \Delta^1\phi_1$, $\phi^2 = \Delta^2\phi_2$.
    \item The normalized classical basis $\hat\phi_1 = \sqrt{\Delta^1}\phi_1$ and $\hat\phi_2 = \sqrt{\Delta^2}\phi_2$, 
    which is self-dual: $\hat\phi^1 = \hat\phi_1$ and $\hat\phi^2 = \hat\phi_2$.
\end{enumerate}
We also consider several different non-flat bases for $H^*_T(\CP;\bar{S}_T)\otimes_{\bar{S}_T} S$:
\begin{enumerate}
    \item The quantum canonical basis $\{\phi_1(q),\phi_2(q)\}$.
    \item The basis dual to the quantum canonical basis with respect to 
    the $T$-equivariant Poincaré pairing: $\phi^1(q) = \Delta^1(q)\phi_1(q)$, $\phi^2(q) = \Delta^2(q)\phi_2(q)$.
    \item The normalized quantum canonical basis $\{\hat\phi_1(q),\hat\phi_2(q)\}$, which is self-dual.
\end{enumerate}
We introduce some additional notations. For $\alpha,\beta\in\{1,2\}$, define
\[
    S^\alpha_{\ \beta} = (\phi^\alpha,\cS(\phi_\beta))_{\bP^1,T},\quad S^{\hat{\underline{\alpha}}}_{\ \beta} = (\hat{\phi}_\alpha(q),\cS(\phi_\beta))_{\bP^1,T}, 
    \quad S^{\hat{\underline{\alpha}}}_{\ \hat{\underline{\beta}}} = (\hat{\phi}_\alpha(q),\cS(\hat{\phi}_\beta(q)))_{\bP^1,T}.
\]
For $\alpha,\beta\in\{1,2\}$ and $i\in\{0,1\}$, define
\[
    S_i^{\ \hat{\alpha}} = (T_i, \cS(\hat{\phi}^\alpha))_{\bP^1,T}.
\]
For $a,b\in H^*_T(\bP^1;\bar{S}_T)$, define
\begin{equation*}
    \begin{aligned}
        S_z(a,b) &= (a,\cS(b))_{\bP^1,T}
        \\
        V_{z_1,z_2}(a,b) &= \frac{(a,b)_{\bP^1,T}}{z_1+z_2} + \llangle\frac{a}{z_1-\psi_1},\frac{b}{z_2-\psi_2}\rrangle_{0,2}^{\bP^1,T}.
    \end{aligned}
\end{equation*}
We have the following identity:
\begin{equation}\label{eqn:V-to-S}
    V_{z_1,z_2}(a,b) = \frac{1}{z_1+z_2}\sum_{\alpha\in\{1,2\}}S_{z_1}(\hat{\phi}_\alpha(q),a)S_{z_2}(\hat{\phi}_\alpha(q),b).
\end{equation}

\begin{remark}\rm \label{rmk:Q-is-1}
    Let $t'=t^1H$ and $t''=t^01$. By the divisor equation, we have
    \begin{equation*}
        \begin{aligned}
             & \ (a,b)_{\CP,T} + \llangle a,\frac{b}{z-\psi}\rrangle_{0,2}^{\CP,T} 
             \\
             = 
        & \ (a,be^{t'/z})_{\CP,T} + \sum_{m=0}^{\infty}\sum_{\beta\in E(\CP)\atop {(\beta,m)\neq (0,0)}}
        \frac{Q^\beta e^{\int_\beta t'}}{m!}\langle a,\frac{be^{t'/z}}{z-\psi},(t'')^m\rangle^{\CP,T}_{0,2+m,\beta}.
        \end{aligned}
    \end{equation*}
    The factor $e^{\int_\beta t'}$ plays the role of $Q^\beta$, so the operator $\cS(b)\big|_{Q=1}$ is well-defined.
\end{remark}

\subsection{Equivariant $J$-function}\label{sec:J-function}
The $T$-equivariant $J$-function $J(z)$ is characterized by
\[
    (J(z), a)_{\CP,T} = (1,\cS(a)\big|_{Q=1})_{\bP^1,T},
\]
for any $a\in H^*_T(\bP^1;\bar{S}_T)$. Equivalently, 
\[
    J(z) = 1 + \sum_{\alpha\in\{1,2\}}\llangle 1,\frac{\hat{\phi}_\alpha}{z-\psi}\rrangle_{0,2}^{\bP^1,T}\Big|_{Q=1}\hat{\phi}_\alpha.
\]
By the genus zero mirror theorem \cite{G96,LLY97}, 
\[
    J(z) = e^{(t^0+t^1H)/z}\Bigg(
        1 + \sum_{d=1}^{\infty}\frac{q^d}{\prod_{m=1}^{d}(H+\sv/2+mz)\prod_{m=1}^{d}(H-\sv/2+mz)}
    \Bigg),
\]
where $q=e^{t^1}$.

Let $J(z) = J^1\phi_1+J^2\phi_2$. Then for $\alpha=1,2$, we have 
\begin{equation}\label{eq:J-function}
    \begin{aligned}
        J^\alpha &= e^{(t^0+ t^1\Delta^\alpha/2)/z}\sum_{d=0}^{\infty}
        \frac{q^d}{d!z^d}\frac{1}{\prod_{m=1}^{d}(\Delta^\alpha+mz)}
        \\
        &= e^{(t^0+ t^1\Delta^\alpha/2)/z}\sum_{m=0}^{\infty}\Big(\frac{\sqrt{q}}{z}\Big)^{2m}
        \frac{\Ga(\Delta^\alpha/z+1)}{m!\Ga(\Delta^\alpha/z+m+1)}
        \\
        &= e^{t^0/z}z^{\Delta^\alpha/z}\Ga(\Delta^\alpha/z+1)I_{\Delta^\alpha/z}\Big(\frac{2\sqrt{q}}{z}\Big),
    \end{aligned}
\end{equation}
where the function $I_\alpha(x)$ is the modified Bessel function of first kind in Appendix \ref{sec:Bessel}.

\subsection{The $R$-matrix}
Let $\hat{H}$ be the formal scheme 
\[
    \hat{H}:=\operatorname{Spec}(S[\![t^0,t^1]\!]),
\]
and let $\mathcal{O}_{\hat{H}}$ be the structure sheaf on $\hat{H}$ and $\mathcal{T}_{\hat{H}}$ be the tangent sheaf on $\hat{H}$.
Given an open set $U$ in $\hat{H}$, we have 
\[
    \mathcal{T}_{\hat{H}}(U) \cong \bigoplus_{i\in\{0,1\}}\mathcal{O}_{\hat{H}}(U)\frac{\partial}{\partial t^i}.
\]
The quantum product $\star$ and the $T$-equivariant Poincar\'e pairing define the structure of a formal Frobenius manifold on $\hat{H}$.

Consider the quantum connection, which is a family of connections parameterized by $z\in\bC\cup \{\infty\}$, on the tangent bundle $T_{\hat{H}}$ on the formal Frobenius manifold $\hat{H}$:
\begin{equation*}
        \nabla^z_\alpha = \frac{\partial}{\partial t^\alpha} - \frac{1}{z}\hat{\phi}_\alpha\star,
\end{equation*}
The equation 
\[
    \nabla^z\mu=0
\]
for a section $\mu\in\Ga(\hat{H},\cT_{\hat{H}})$ is called the $T$-equivariant big quantum differential equation.

Let $U$ denote the diagonal matrix $\text{diag}(u^1, u^2)$. The results in \cite{Giv98b} imply the following statement.
\begin{theorem}
    There exists a unique matrix power series $R(z) = \one + R_1z + R_2z^2 + \dots$ satisfying the following properties:
    \begin{enumerate}
        \item The entries of each $R_k$ lie in $\bar{S}_T[[q,t^0]]$.
        \item $\tilde{S} = \Psi R(z)e^{U/z}$ is a fundamental solution to the $T$-equivariant big quantum differential equation.
        \item $R$ satisfies the unitary condition $R^T(-z)R(z)=\one$.
        \item For $\alpha,\beta\in\{1,2\}$, we have
        \[
            \lim_{q,t^0 \rightarrow 0} R_{\alpha}^{\ \beta}(z) = \delta_{\alpha\beta}\exp\left(
                -\sum_{n=1}^{\infty} \frac{B_{2n}}{2n(2n-1)}\left(\frac{z}{\Delta^\beta}\right)^{2n-1}
            \right).
        \]  
    \end{enumerate}
\end{theorem}

\subsection{The graph sum formula for descendant Gromov-Witten potential}\label{sec:A-model-graph-sum}
In this subsection, we introduce the graph sum formula for the generating functions of descendant
Gromov-Witten invariants. We introduce a labeled graph $\vec{\Ga}=(\Ga,g,\beta,k)$ as follows (see Figure \ref{fig:givental-graph}). Here $\Ga$ is a connected graph equipped with the following data: 
\begin{enumerate}
    \item $V(\Gamma)$ is the set of vertices in $\Gamma$.
    \item $H(\Gamma)$ is the set of half-edges in $\Gamma$. $H(\Ga)$ is equipped with a vertex assignment $v: H(\Ga)\rightarrow V(\Ga)$ and an involution $\iota$.
    \item $E(\Gamma)$ is the set of edges in $\Gamma$, which is defined by the 2-cycles of $\iota$ in $H(\Ga)$.
    \item $L(\Gamma)$ is the set of leaves in $\Gamma$, which is defined by the fixed points of $\iota$.
    It admits a disjoint splitting $L(\Ga)=L^o(\Ga)\sqcup L^1(\Ga)$, 
    where $L^o(\Ga)=\{l_1,\dots,l_n\}$ is the set of ordinary leaves in $\Ga$, with $n=|L^o(\Ga)|$, and
    $L^1(\Ga)$ is the set of dilaton leaves in $\Ga$. 
\end{enumerate}
With the above notation, we introduce the following labels:
\begin{enumerate}
    \item (genus) $g: V(\Ga)\rightarrow\bZ_{\geq 0}$.
    \item (marking) $\beta: V(\Gamma)\rightarrow\{1,2\}$. This induces $\beta: L(\Ga)=L^o(\Gamma)\cup L^1(\Gamma)
    \rightarrow \{1,2\}$, as follows: if $l\in L(\Gamma)$ is a leaf attached to a vertex $v\in V(\Gamma)$, define
    $\beta(l) = \beta(v)$.
    \item (height) $k: H(\Gamma)\rightarrow \bZ_{\geq 0}$. Moreover, we require $k(l)\geq 2$ for every dilaton leaf $l\in L^1(\Ga)$.
\end{enumerate}
An automorphism of the labeled graph $\vec{\Ga}=(\Ga,g,\beta,k)$ is an automorphism of the underlying graph $\Ga$
preserving all incidence relations and all decorations $(g,\beta,k)$. Moreover, it fixes each ordinary leaf $l_i\in L^o(\Ga)$. The dilaton leaves are 
unlabeled and may be permuted by automorphisms. We denote the automorphism group by $\Aut(\vec{\Ga})$.

Given an edge $e$, let $h_1(e)$ and $h_2(e)$ be the two half-edges associated to $e$. The order of the two half-edges
does not affect the graph sum formula in this paper. Given a vertex $v\in V(\Gamma)$, let $H(v)$ denote the set of half-edges
emanating from $v$. The valency of the vertex $v$ is equal to the cardinality of the set $H(v)$: $\val(v)= |H(v)|$. 
A labeled graph $\vec{\Ga} = (\Ga,g,\beta,k)$ is \emph{stable} if 
\[
    2g(v)-2 +\val(v)>0
\] 
for all $v\in V(\Ga)$.

\begin{figure}
    \center

\tikzset{every picture/.style={line width=0.55pt}} 

\begin{tikzpicture}[x=0.55pt,y=0.55pt,yscale=-1,xscale=1]

\draw [color={rgb, 255:red, 208; green, 2; blue, 27 }  ,draw opacity=1 ]   (223.86,93.38) -- (320,93.38) ;
\draw [color={rgb, 255:red, 208; green, 2; blue, 27 }  ,draw opacity=1 ]   (219.5,103.5) -- (238.2,127.8) ;
\draw [color={rgb, 255:red, 208; green, 2; blue, 27 }  ,draw opacity=1 ]   (277,155.5) -- (349.2,127.6) ;
\draw    (439.43,83) -- (484.93,59.68) ;
\draw    (256.43,170.61) -- (230.43,188.68) ;
\draw    (194.93,53.18) -- (205.05,82.07) ;
\draw    (440.93,103) -- (484.93,120.68) ;
\draw    (275.43,170.61) .. controls (277.78,170.48) and (279.02,171.6) .. (279.15,173.95) .. controls (279.28,176.3) and (280.52,177.42) .. (282.87,177.29) .. controls (285.22,177.16) and (286.46,178.28) .. (286.58,180.63) .. controls (286.71,182.98) and (287.95,184.1) .. (290.3,183.98) .. controls (292.65,183.85) and (293.89,184.97) .. (294.02,187.32) .. controls (294.15,189.67) and (295.39,190.79) .. (297.74,190.66) -- (299.43,192.18) -- (299.43,192.18) ;
\draw    (169.43,117.68) .. controls (170,115.39) and (171.42,114.52) .. (173.71,115.09) .. controls (176,115.66) and (177.42,114.8) .. (177.99,112.51) .. controls (178.56,110.22) and (179.98,109.36) .. (182.27,109.92) .. controls (184.56,110.48) and (185.98,109.62) .. (186.54,107.33) .. controls (187.11,105.04) and (188.53,104.18) .. (190.82,104.75) .. controls (193.11,105.31) and (194.53,104.45) .. (195.1,102.16) .. controls (195.67,99.87) and (197.09,99.01) .. (199.38,99.57) -- (199.5,99.5) -- (199.5,99.5) ;
\draw    (165.83,78.68) .. controls (168.2,77.55) and (169.8,78.01) .. (170.93,80.08) .. controls (172.06,82.15) and (173.66,82.62) .. (175.73,81.49) .. controls (177.8,80.36) and (179.4,80.82) .. (180.53,82.89) .. controls (181.66,84.96) and (183.25,85.42) .. (185.32,84.29) .. controls (187.39,83.16) and (188.99,83.63) .. (190.12,85.7) .. controls (191.25,87.77) and (192.85,88.23) .. (194.92,87.1) -- (199.7,88.5) -- (199.7,88.5) ;
\draw [color={rgb, 255:red, 74; green, 144; blue, 226 }  ,draw opacity=1 ]   (320,93.38) -- (419.75,93.38) ;
\draw [color={rgb, 255:red, 74; green, 144; blue, 226 }  ,draw opacity=1 ]   (238.2,127.8) -- (256.9,152.2) ;
\draw [color={rgb, 255:red, 74; green, 144; blue, 226 }  ,draw opacity=1 ]   (349.2,127.6) -- (421.4,99.7) ;
\draw   (198.25,93.38) .. controls (198.25,86.3) and (203.98,80.57) .. (211.05,80.57) .. controls (218.12,80.57) and (223.86,86.3) .. (223.86,93.38) .. controls (223.86,100.45) and (218.12,106.18) .. (211.05,106.18) .. controls (203.98,106.18) and (198.25,100.45) .. (198.25,93.38) -- cycle ;
\draw   (252.75,161.88) .. controls (252.75,154.8) and (258.48,149.07) .. (265.55,149.07) .. controls (272.62,149.07) and (278.36,154.8) .. (278.36,161.88) .. controls (278.36,168.95) and (272.62,174.68) .. (265.55,174.68) .. controls (258.48,174.68) and (252.75,168.95) .. (252.75,161.88) -- cycle ;
\draw   (419.75,93.38) .. controls (419.75,86.3) and (425.48,80.57) .. (432.55,80.57) .. controls (439.62,80.57) and (445.36,86.3) .. (445.36,93.38) .. controls (445.36,100.45) and (439.62,106.18) .. (432.55,106.18) .. controls (425.48,106.18) and (419.75,100.45) .. (419.75,93.38) -- cycle ;

\draw (202,86.97) node [anchor=north west][inner sep=0.55pt]    {$v_{1}$};
\draw (255.86,155.97) node [anchor=north west][inner sep=0.55pt]    {$v_{3}$};
\draw (423.16,86.97) node [anchor=north west][inner sep=0.55pt]    {$v_{2}$};
\draw (201.86,57.9) node [anchor=north west][inner sep=0.55pt]    {$l_{1}$};
\draw (453.86,50.4) node [anchor=north west][inner sep=0.55pt]    {$l_{2}$};
\draw (453.86,91.3) node [anchor=north west][inner sep=0.55pt]    {$l_{3}$};
\draw (233.36,158.4) node [anchor=north west][inner sep=0.55pt]    {$l_{4}$};
\draw (250.86,71.4) node [anchor=north west][inner sep=0.55pt]  [color={rgb, 255:red, 208; green, 2; blue, 27 }  ,opacity=1 ]  {$h_{1}( e)$};
\draw (353.36,71.4) node [anchor=north west][inner sep=0.55pt]  [color={rgb, 255:red, 74; green, 144; blue, 226 }  ,opacity=1 ]  {$h_{2}( e)$};
\draw (185.86,38.4) node [anchor=north west][inner sep=0.55pt]    {\small$1$};
\draw (152.86,66.9) node [anchor=north west][inner sep=0.55pt]    {\small$2$};
\draw (154.36,111.4) node [anchor=north west][inner sep=0.55pt]    {\small$3$};
\draw (302.86,189.9) node [anchor=north west][inner sep=0.55pt]    {\small$2$};
\draw (220.36,189.4) node [anchor=north west][inner sep=0.55pt]    {\small$0$};
\draw (487.86,48.9) node [anchor=north west][inner sep=0.55pt]    {\small$0$};
\draw (487.86,112.4) node [anchor=north west][inner sep=0.55pt]    {\small$3$};
\draw (260.86,96.4) node [anchor=north west][inner sep=0.55pt]   [color={rgb, 255:red, 208; green, 2; blue, 27 }  ,opacity=1 ] {\small$4$};
\draw (360.86,96.4) node [anchor=north west][inner sep=0.55pt]   [color={rgb, 255:red, 74; green, 144; blue, 226 }  ,opacity=1 ] {\small$6$};
\draw (213.86,112.4) node [anchor=north west][inner sep=0.55pt]   [color={rgb, 255:red, 208; green, 2; blue, 27 }  ,opacity=1 ] {\small$5$};
\draw (231.86,133.4) node [anchor=north west][inner sep=0.55pt]  [color={rgb, 255:red, 74; green, 144; blue, 226 }  ,opacity=1 ]  {\small$1$};
\draw (375.86,119.4) node [anchor=north west][inner sep=0.55pt]  [color={rgb, 255:red, 74; green, 144; blue, 226 }  ,opacity=1 ]  {\small$4$};
\draw (320.86,140.4) node [anchor=north west][inner sep=0.55pt]   [color={rgb, 255:red, 208; green, 2; blue, 27 }  ,opacity=1 ] {\small$3$};
\end{tikzpicture}
\caption{A labeled graph $\vec{\Ga}=(\Ga,g,\beta,k)$}\label{fig:givental-graph}
\begin{minipage}{\textwidth}
    \raggedright
    An example of a labeled graph, with vertex functions $g$ and $\beta$ omitted. Vertices are represented by circles and labeled by $v_i$. Straight external legs attached to vertices denote ordinary leaves, marked by $l_j$. Wavy external legs attached to vertices represent dilaton leaves. The half-edges $h_1(e)$ and $h_2(e)$ are colored red and blue, respectively. 
    Each half-edge is assigned an integer specifying the height function $k\colon H(\Gamma)\to\mathbb{Z}_{\geq 0}$.
\end{minipage}
\end{figure}

Let ${\bf\Gamma}(\bP^1)$ denote the set of all stable labeled graphs $\vec{\Gamma}=(\Gamma, g,\beta,k)$.
The genus of a stable labeled graph $\vec{\Gamma}$ is defined to be
\[
    g(\vec{\Gamma}) := \sum_{v\in V(\Ga)} g(v) + |E(\Gamma)| - |V(\Ga)| + 1
    = \sum_{v\in V(\Ga)}(g(v)-1)+ (\sum_{e\in E(\Ga)}1) + 1.
\]
Define 
\[
    {\bf \Ga}_{g,n}(\bP^1) = \{\vec{\Ga}=(\Ga,g,\beta,k)\in {\bf\Ga}(\bP^1) : 
    g(\vec{\Ga})=g, |L^o(\Ga)| = n\}.
\]
We assign weights to leaves, edges, and vertices of a labeled graph $\vec{\Gamma}\in {\bf \Ga}(\bP^1)$ as follows.
\begin{enumerate}
    \item \emph{Ordinary leaves}. To each ordinary leaf $l_j \in L^o(\Gamma)$ with $\beta(l_j)=\beta\in\{1,2\}$
        and $k(l_j)=k\in\bZ_{\geq 0}$, we assign the following descendant weight:
        \[
            (\cL^{\bf u})^\beta_k(l_j) = [z^k](\sum_{\alpha,\gamma \in \{1,2\}}\left(
                \frac{{\bf u}^\alpha_j(z)}{\sqrt{\Delta^\alpha(q)}}S^{\hat{\underline{\gamma}}}_{\ \hat{\underline{\alpha}}}(z)
            \right)_+ R(-z)_\gamma^{\ \beta}).
        \]
    where $(\cdot)_+$ means taking the nonnegative powers of $z$.
    \item \emph{Dilaton leaves}. To each dilaton leaf $l\in L^1(\Ga)$ with $\beta(l)=\beta \in \{1,2\}$ and 
        $2\leq k(l)=k\in\bZ_{\geq 0}$, we assign
        \[
            (\cL^1)^\beta_k(l) = [z^{k-1}](-\sum_{\alpha\in \{1,2\}}\frac{1}{\sqrt{\Delta^\alpha(q)}}R_\alpha^{\ \beta}(-z)).
        \] 
    \item \emph{Edges}. To an edge connecting a vertex marked by $\alpha\in\{1,2\}$ and a vertex marked by $\beta\in\{1,2\}$,
    and with heights $k$ and $l$ at the corresponding half-edges, we assign
        \[
            \cE^{\alpha,\beta}_{k,l} = [z^kw^l]\Big(
                \frac{1}{z+w}(\delta_{\alpha,\beta}-\sum_{\gamma\in \{1,2\}}R_\gamma^{\ \alpha}(-z)R_\gamma^{\ \beta}(-w))
            \Big).
        \]
    \item \emph{Vertices}. To a vertex $v$ with genus $g(v)=g\in \bZ_{\geq 0}$ and with marking $\beta(v)=\beta$,
        with $m=\val(v)$ half-edges attached to it with heights $k_1,\dots,k_m\in\bZ_{\geq 0}$, we assign
        \[
            \Big(\sqrt{\Delta^\beta(q)}\Big)^{2g(v)-2+\val(v)}\langle\tau_{k_1}\dots\tau_{k_{m}}\rangle_g,
        \]
        where $\langle\tau_{k_1}\dots\tau_{k_{m}}\rangle_g = \int_{\overline{\cM}_{g,m}}\psi_1^{k_1}\dots\psi_{m}^{k_{m}}$.
\end{enumerate}
We define the weight of a labeled graph $\vec{\Ga}\in {\bf\Ga}_{g,n}(\bP^1)$ to be 
\begin{equation*}
    \begin{aligned}
        \omega^{\bf u}_A(\vec{\Ga}) = &\prod_{v\in V(\Ga)}\Big(\sqrt{\Delta^{\beta(v)}(q)}\Big)^{2g(v)-2+\val(v)}\langle\prod_{h\in H(v)}\tau_{k(h)}\rangle_{g(v)}
        \prod_{e\in E(\Ga)}\cE^{\beta(v_1(e)),\beta(v_2(e))}_{k(h_1(e)),k(h_2(e))}
        \\
        & \cdot \prod_{l\in L^1(\Ga)}(\cL^1)^{\beta(l)}_{k(l)}(l)\prod_{j=1}^n(\cL^{\bf u})^{\beta(l_j)}_{k(l_j)}(l_j).
    \end{aligned}
\end{equation*}
With the above definition of the weight of a labeled graph, we have the following theorem which expresses the $T$-equivariant descendant
Gromov-Witten potential of $\bP^1$ in terms of graph sum.
\begin{theorem}[\cite{Giv98b}]\label{thm:descendant-graph-sum}
    Suppose that $2g-2+n > 0$. Then
    \[
        \llangle{\bf u}_1, \dots, \bu_n\rrangle_{g,n}^{\bP^1,T} = \sum_{\vec{\Ga}\in{\bf \Ga}_{g,n}(\bP^1)}
        \frac{\omega^\bu_A(\vec{\Ga})}{|\Aut(\vec{\Ga})|}.
    \]
\end{theorem}

\section{Equivariant open Gromov-Witten theory of $(\bP^1,\bR\bP^1)$ via integration over the fixed locus}\label{sec:openGW}
\subsection{Moduli spaces of stable maps to ($\bP^1,\bR\bP^1$)}
Let 
$$   
    L := \{[e^{{\rm i}\varphi},e^{-{\rm i}\varphi}]\in\CP: \varphi\in\bR\}
$$
be the Lagrangian submanifold of $\CP$ which is preserved by the $T$-action. By taking a M\"obius transform, we can identify the pair $(\CP,L)$ with $(\CP,\RP)$. Let $\bD_1$ and $\bD_2$ be the two hemispheres centered at $p_1$ and $p_2$ respectively.
Then we have
\[
    H_2(\bP^1,\bR\bP^1) = \bZ[\bD_1]\oplus \bZ[\bD_2].
\]
Sometimes, we identify the relative homology group $H_2(\bP^1,\bR\bP^1)$ with $\bZ^2$ so that the generators $(1,0)$ and $(0,1)$ represent $\bD_1$ and $\bD_2$ 
respectively. 
The relative homology group $H_2(\bP^1,\bR\bP^1)$ could also be viewed as the extension of the homology groups $H_2(\bP^1)$ and $H_1(\bR\bP^1)$.
Consider the short exact sequence
\begin{equation}
    \begin{tikzcd}\label{eqn:relative-homology}
        0 \arrow[r] & H_2(\bP^1)= \bZ[\bP^1] \arrow[r, "\delta"] & H_2(\bP^1,\bR\bP^1) \arrow[r, "\partial"] & H_1(\bR\bP^1)= \bZ[\RP] \arrow[r] & 0.
    \end{tikzcd}
\end{equation}
For $d[\bP^1]\in H_2(\bP^1)$, the map $\delta:H_2(\bP^1)\rightarrow H_2(\bP^1,\bR\bP^1)$ is defined by $\delta(d[\bP^1])=(d,d)$.
The connecting map $\partial:H_2(\bP^1,\bR\bP^1)\rightarrow H_1(\bR\bP^1)$ is given by $\partial(a,b) = (b-a)[\RP]$. 

Let $(\Si,\partial\Si)$ be a prestable bordered Riemann surface, and let $\partial\Si=R_1\cup\dots\cup R_h$.
The (small) genus $g(\Si)$ of $\Si$ is the genus of the closed Riemann surface obtained by capping each boundary component of $\Si$ with a disk.
A genus-$g$ bordered Riemann surface with $h$ boundary components is called a Riemann surface of topological type $(g,h)$.

We are concerned with the maps $u:(\Si,\partial\Si)\rightarrow
(\bP^1,\bR\bP^1)$.
The \emph{degree} of the map $u$ is the element 
\[
    \beta' = (d_-,d_+)= u_*[\Si]\in H_2(\bP^1,\RP).
\]
Let
\[
\mu_i[\RP]= (u\mid_{R_i})_*[R_i]\in H_1(\RP)=\bZ[\RP],\quad i=1,\cdots,h.
\]
The number $\mu_i$ is called the $i$-th winding number. 
Let $\beta'= u_*[\Si]\in H_2(\bP^1,\RP)$, $\vec{\mu} = (\mu_1,\dots,\mu_{h})\in\bZ_{\neq 0}^h$ and
let $J_\pm = \{j\in \{1,\dots,h\}:\pm \mu_j > 0\}$. 
Then there exists $d\in\bZ_{\geq 0}$
such that 
$$\beta' = d[\bP^1] - \sum_{j\in J_-} \mu_j[\bD_1] + \sum_{j\in J_+}\mu_j[\bD_2].$$
By the short exact sequence \eqref{eqn:relative-homology}, we have
\[
        d_+-d_- = \sum_{j=1}^{h}\mu_j,
\]
\[
        d = d_+ - \sum_{j\in J_+}\mu_j = d_- +\sum_{j\in J_-}\mu_j.
\]

Let $(\Si,\partial\Si,x_1,\dots,x_n)$ be a prestable bordered Riemann surface with $n$ interior marked points,
and consider stable maps:
\[
    u: (\Si,\partial\Si,x_1,\dots,x_n)\rightarrow (\bP^1,\bR\bP^1).
\]
Here, a stable map is a map whose automorphism group is finite.

Let $\overline{\cM}_{(g,h),n}(\CP,\RP | \beta',\vec{\mu})$ be the moduli space of degree $\beta'$ stable maps to $(\bP^1,\bR\bP^1)$ from type $(g,h)$ bordered Riemann surfaces with $n$ interior marked points, such that
the winding numbers are given by $\mu_i\in\bZ$.

\subsection{Equivariant open Gromov-Witten invariants}\label{sec:virOGW}
Let $\gamma_1,\dots,\gamma_n\in H^*_T(\bP^1;\bC)$, $\beta'\in H_2(\bP^1,\RP)$.
The $T$-action on $\CP$ induces a $T$-action on the moduli space $\overline{\cM}_{(g,h), n}(\bP^1,\bR\bP^1|\beta',\vec{\mu})$.
Let $F:=\overline{\cM}_{(g,h), n}(\bP^1,\bR\bP^1|\beta',\vec{\mu})^T$ be the $T$-fixed locus. The evaluation map $\ev_i: \overline{\cM}_{(g,h), n}(\bP^1,\bR\bP^1|\beta',\vec{\mu})\to \bP^1$ is $T$-equivariant.

In order to describe the fixed locus $F$, we introduce the following decorated graphs.

\begin{defn}[Decorated graphs (a)]
    Let $n\in\bZ_{\geq 0}$ and $\beta' = d[\bP^1] - \sum_{j\in J_-} \mu_j[\bD_1] + \sum_{j\in J_+}\mu_j[\bD_2]\in H_2(\CP,\RP)$. A genus $g$, $n$-pointed,
    degree $\beta'$ decorated graph for $(\CP,\RP)$ is a tuple $\widetilde{\Ga}=(\Ga,\vec{f},\vec{d}, \vec{g},\vec{s})$ consisting of 
    the following data (see Figure \ref{fig:decorated-graph}). 
    \begin{enumerate}
        \item $\Gamma$ is a compact connected graph. Let $V(\Ga)\sqcup V_\circ(\Ga)$ denote
        the set of vertices in $\Ga$, denoted by $\bullet$ and $\circ$, respectively. The $\circ$ vertices are univalent.
        Let $E(\Ga)$ denote the set of edges, where an edge $e$ is a line connecting two $\bullet$ vertices.
        Let $L(\Ga) = \{l_j: j=1,\dots,h\}$ denote the set of leaves, where a leaf $l_j$ is a line connecting one $\bullet$ vertex and one $\circ$ vertex.
        We require the leaves to be ordered. 
        Let $F(\Ga)$ be the set of flags:
        \[
            \{(e,v)\in E(\Ga)\times V(\Ga):v\in e\} \cup \{(l_j,v)\in L(\Ga)\times V(\Ga):v\in l_j\}.
        \]
        For each $v\in V(\Ga)$, let $F_v$ (\emph{resp.} $E_v$, $L_v$) denote the flags (\emph{resp.} edges, leaves) attached to $v$, and let $\val(v)=|F_v|$ denote the number of flags incident to $v$.
        \item The \emph{label map} $\vec{f}: V(\Ga)\rightarrow \{1,2\}$ labels each $\bullet$ with a number.
            If $v_1,v_2\in V(\Ga)$ are connected by an edge, we require
            $\vec{f}(v_1)\neq \vec{f}(v_2)$.
        \item The \emph{degree map} $\vec{d}:E(\Ga)\cup L(\Ga)\rightarrow \bZ_{>0}$ sends an edge $e$ (\emph{resp.} a leaf $l_j$) to a positive integer $\vec{d}(e)=d_e$ (\emph{resp.} $\vec{d}(l_j)=d_{l_j}$).
        \item The \emph{genus map} $\vec{g}:V(\Ga)\rightarrow\bZ_{\geq 0}$ sends a vertex $v\in V(\Ga)$ to a nonnegative integer $g_v$.
        \item The \emph{marking map} $\vec{s}:\{1,2,\dots,n\}\rightarrow V(\Ga)$. For each $v\in V(\Ga)$, define $S_v := \vec{s}^{-1}(v)$, and $n_v=|S_v|$.
    \end{enumerate}
    The data is required to satisfy the following conditions:
    \begin{itemize}
        \item [(i)] (genus) $g = \sum_{v\in V(\Ga)}g_v + |E(\Ga)| - |V(\Ga)| + 1$.
        \item [(ii)] (degree) $d = \sum_{e\in E(\Ga)}d_e$.
        \item [(iii)] (winding numbers) Let $l_j\in L(\Ga)$ be the $j$-th leaf, and let $v_j\in V(\Ga)$ be its incident $\bullet$ vertex, then 
                $\mu_j = (-1)^{\vec{f}(v_j)}\vec{d}(l_j)$.
    \end{itemize}
    Let $G_{g,n}(\CP,\RP|\beta',\vec{\mu})$ be the set of all decorated graphs $\widetilde{\Ga}=(\Ga,\vec{f},\vec{d},\vec{g},\vec{s})$ satisfying the above constraints. 
\end{defn}

Given a decorated graph $\widetilde{\Ga}\in G_{g,n}(\CP,\RP|\beta',\vec{\mu})$. Let $V^S(\Ga)$ be the set of stable vertices:
\[
    V^S(\Ga) := \{v\in V(\Ga):2g_v-2+\val(v)+n_v >0\}.
\]

\noindent
Let $\Aut(\widetilde{\Ga})$ denote the group of automorphisms of $\widetilde{\Ga}$, and let $A^0_{\widetilde{\Ga}}$ be the group of covering automorphisms
\[
    A^0_{\widetilde{\Ga}} = \prod_{e\in E(\Ga)}\bZ_{\vec{d}(e)}\times\prod_{l_j\in L(\Ga)}\bZ_{\vec{d}(l_j)}.
\]
\noindent
The automorphism group $A_{\widetilde{\Ga}}$ fits in the short exact sequence: 
\[
    1 \rightarrow A^0_{\widetilde{\Ga}}
    \rightarrow A_{\widetilde{\Ga}}\rightarrow \Aut(\widetilde{\Ga})\rightarrow 1.
\]
\noindent
We set 
\[
    \overline{\cM}_{\widetilde{\Ga}} := \prod_{v\in V^S(\Ga)}\overline{\cM}_{g_v,F_v\cup S_v},
\quad
    F_{\widetilde{\Ga}} := [\overline{\cM}_{\widetilde{\Ga}}/A_{\widetilde{\Ga}}].
\]

Every decorated graph $\widetilde{\Ga}\in G_{g,n}(\CP,\RP|\beta',\vec{\mu})$ represents a topological type of the $T$-invariant stable open map $u:(\Si,\partial\Si,x_1,\dots,x_n)\rightarrow (\CP,\RP)$ in the following way:
\begin{enumerate}
    \item Every $v\in V(\Ga)$ represents a stable curve or a single point $\Si_v$ in the domain curve $\Si$. 
    If $\Si_v$ is a stable curve, it is of genus-$g_v$ with marked points $S_v$.
    The map $u$ contracts $\Si_v$ to the fixed-point $p_{\vec{f}(v)}$.
    \item Every leaf $l_j\in L(\Ga)$ represents a disk $D_j$ in the domain curve $\Si$. 
    The map $u|_{D_j}$ is a standard degree $\vec{d}(l_j)=|\mu_j|$ cover of disk branched at $p_{\vec{f}(v_j)}$, where $v_j$ is the $\bullet$ vertex attached at $l_j$.
    \item Every edge $e\in E(\Ga)$ represents a sphere $\CP\cong C\subset \Si$,
    so that $u|_C$ is a degree $\vec{d}(e)=d_e$ cover of $\CP$ branched at $p_1,p_2$. 
\end{enumerate}

Vice versa, given a $T$-invariant stable open map $u$, one can read off its decorated graph $\widetilde{\Gamma}$. Therefore, we can use the decorated graphs to 
describe the connected components of the 
$T$-fixed locus of the moduli space $\overline{\cM}_{(g,h), n}(\bP^1,\bR\bP^1|\beta',\vec{\mu})$:
\begin{equation}\label{eq:fixed-locus}
    \overline{\cM}_{(g,h), n}(\bP^1,\bR\bP^1|\beta',\vec{\mu})^T = \bigcup_{\widetilde{\Ga} \in G_{g,n}(\CP,\RP|\beta',\vec{\mu})}
    F_{\widetilde{\Ga}}.
\end{equation}

\begin{figure}[H]
\tikzset{every picture/.style={line width=0.75pt}} 

\begin{tikzpicture}[x=0.75pt,y=0.75pt,yscale=-1,xscale=1]

\draw    (153.19,60.5) -- (200.76,77) ;
\draw    (153.69,99) -- (200.76,78.18) ;
\draw    (153.19,158.5) -- (199.69,138) ;
\draw    (301.26,137.68) -- (339.26,158.18) ;
\draw    (199.76,77) .. controls (208.69,63.23) and (292.83,63.23) .. (301.76,77) ;
\draw    (199.76,77) .. controls (218.29,90.03) and (283.23,90.03) .. (301.76,77) ;
\draw    (199.69,136.5) .. controls (214.69,134.63) and (277.09,111.43) .. (301.76,77.68) ;
\draw    (199.69,138) .. controls (221.49,143.03) and (258.29,143.83) .. (299.26,138.68) ;

\draw   (429.4,58.01) .. controls (433.86,52.81) and (444.42,51.74) .. (453,55.62) .. controls (461.57,59.49) and (464.91,66.85) .. (460.45,72.05) .. controls (456,77.24) and (445.44,78.31) .. (436.86,74.44) .. controls (428.29,70.56) and (424.95,63.2) .. (429.4,58.01) -- cycle ;
\draw    (442.15,57.37) .. controls (436.67,62.88) and (443,68.28) .. (453.24,65.59) ;
\draw    (440.52,60.37) .. controls (445.63,60.42) and (448.48,62.76) .. (449.4,66.25) ;
\draw   (519.85,78.96) .. controls (511.15,74.11) and (508.81,63.7) .. (514.64,55.73) .. controls (520.46,47.76) and (532.23,45.23) .. (540.93,50.09) .. controls (549.63,54.95) and (551.96,65.35) .. (546.14,73.32) .. controls (540.32,81.3) and (528.55,83.82) .. (519.85,78.96) -- cycle ;
\draw    (524.79,64.23) .. controls (519.99,69.47) and (525.63,74.49) .. (534.67,71.87) ;
\draw    (523.37,67.07) .. controls (527.89,67.07) and (530.43,69.25) .. (531.28,72.53) ;
\draw    (528.36,52.77) .. controls (523.56,58.01) and (529.21,63.03) .. (538.24,60.41) ;
\draw    (526.95,55.61) .. controls (531.46,55.61) and (534,57.79) .. (534.85,61.07) ;
\draw   (421.44,115.03) .. controls (426.19,108.03) and (439.22,106.77) .. (450.55,112.21) .. controls (461.88,117.65) and (467.22,127.74) .. (462.47,134.74) .. controls (457.72,141.75) and (444.69,143.01) .. (433.36,137.57) .. controls (422.03,132.13) and (416.7,122.04) .. (421.44,115.03) -- cycle ;
\draw    (440.56,113.96) .. controls (436.79,118.25) and (441.23,122.36) .. (448.33,120.21) ;
\draw    (439.46,116.28) .. controls (443,116.28) and (445,118.07) .. (445.67,120.75) ;
\draw    (430.27,121.63) .. controls (426.5,125.91) and (430.93,130.02) .. (438.03,127.88) ;
\draw    (429.16,123.95) .. controls (432.71,123.95) and (434.7,125.74) .. (435.37,128.42) ;
\draw    (446.12,125.59) .. controls (442.35,129.88) and (446.79,133.99) .. (453.89,131.84) ;
\draw    (445.01,127.91) .. controls (448.56,127.91) and (450.56,129.7) .. (451.22,132.38) ;
\draw   (518.3,114.81) .. controls (521.48,108.73) and (529.83,106.99) .. (536.94,110.92) .. controls (544.06,114.84) and (547.26,122.95) .. (544.08,129.03) .. controls (540.9,135.1) and (532.55,136.85) .. (525.43,132.92) .. controls (518.31,129) and (515.12,120.89) .. (518.3,114.81) -- cycle ;
\draw    (527.56,113.15) .. controls (523.63,119.69) and (529.11,125.36) .. (537.06,121.71) ;
\draw    (526.51,116.62) .. controls (530.58,116.36) and (533.04,118.82) .. (534.05,122.69) ;
\draw   (453.16,113.03) .. controls (451.87,110.51) and (465.46,100.96) .. (483.52,91.69) .. controls (502.57,82.43) and (518.25,76.95) .. (519.55,79.47) .. controls (520.84,81.99) and (507.25,91.54) .. (488.19,101.81) .. controls (470.14,110.07) and (454.46,115.55) .. (453.16,113.03) -- cycle ;
\draw   (461.45,71.76) .. controls (461.45,68.74) and (472.06,66.29) .. (486.37,66.29) .. controls (500.68,66.29) and (512.29,68.74) .. (512.29,71.76) .. controls (512.29,74.78) and (500.68,77.23) .. (486.37,77.23) .. controls (472.06,77.23) and (461.45,74.78) .. (461.45,71.76) -- cycle ;
\draw   (455.43,55.45) .. controls (455.43,52.78) and (468.64,50.62) .. (484.93,50.62) .. controls (501.22,50.62) and (514.43,52.78) .. (514.43,55.45) .. controls (514.43,58.12) and (501.22,60.29) .. (484.93,60.29) .. controls (468.64,60.29) and (455.43,58.12) .. (455.43,55.45) -- cycle ;
\draw   (464.56,126.48) .. controls (464.35,123.2) and (476.01,119.8) .. (490.61,118.86) .. controls (505.21,117.93) and (517.21,119.82) .. (517.42,123.09) .. controls (517.63,126.36) and (505.97,129.77) .. (491.37,130.71) .. controls (476.77,131.64) and (464.76,129.75) .. (464.56,126.48) -- cycle ;
\draw    (452.22,76.49) .. controls (458.4,76.13) and (466.75,76.25) .. (470.95,83.34) ;
\draw    (452.22,76.09) .. controls (452.09,82.73) and (461.11,91.9) .. (464.04,91.16) ;
\draw   (470.84,83.21) .. controls (471.32,83.69) and (470.18,85.86) .. (468.3,88.06) .. controls (466.42,90.25) and (464.51,91.64) .. (464.03,91.16) .. controls (463.56,90.68) and (464.7,88.5) .. (466.58,86.31) .. controls (468.46,84.11) and (470.37,82.72) .. (470.84,83.21) -- cycle ;
\draw    (443.03,53.16) .. controls (444.02,47.06) and (445.43,38.83) .. (453.34,35.76) ;
\draw    (443.03,53.16) .. controls (449.57,54.31) and (460.02,46.81) .. (459.74,43.81) ;
\draw   (452.93,35.86) .. controls (453.48,35.46) and (455.45,36.92) .. (457.33,39.11) .. controls (459.21,41.31) and (460.29,43.41) .. (459.74,43.81) .. controls (459.19,44.21) and (457.22,42.75) .. (455.34,40.55) .. controls (453.46,38.36) and (452.38,36.25) .. (452.93,35.86) -- cycle ;
\draw    (455.82,139.69) .. controls (462,139.73) and (470.35,139.85) .. (474.37,146.64) ;
\draw    (455.82,139.69) .. controls (455.69,146.33) and (464.71,155.5) .. (467.64,154.76) ;
\draw   (474.44,146.81) .. controls (474.92,147.29) and (473.78,149.46) .. (471.9,151.66) .. controls (470.02,153.85) and (468.11,155.24) .. (467.63,154.76) .. controls (467.16,154.28) and (468.3,152.1) .. (470.18,149.91) .. controls (472.06,147.71) and (473.97,146.32) .. (474.44,146.81) -- cycle ;
\draw    (523.17,131.21) .. controls (521.19,137.07) and (518.47,144.96) .. (510.61,146.71) ;
\draw    (523.17,131.21) .. controls (516.9,129.01) and (505.37,134.72) .. (505.16,137.72) ;
\draw   (510.59,146.68) .. controls (509.98,146.98) and (508.27,145.22) .. (506.77,142.75) .. controls (505.27,140.28) and (504.55,138.03) .. (505.16,137.72) .. controls (505.77,137.42) and (507.47,139.18) .. (508.97,141.65) .. controls (510.47,144.12) and (511.19,146.37) .. (510.59,146.68) -- cycle ;
\draw    (485.8,60.29) .. controls (482.19,59.12) and (481.94,52.37) .. (485.8,50.62) ;
\draw    (486.37,77.23) .. controls (483.09,76.03) and (482.16,67.68) .. (486.37,66.29) ;
\draw    (487.09,102.03) .. controls (483.16,101.28) and (481.49,92.43) .. (486.93,90.62) ;
\draw    (489.08,130.33) .. controls (484.96,130.48) and (483.93,121.17) .. (488.2,118.95) ;
\draw    (485.8,50.62) .. controls (490.94,51.22) and (489.94,60.82) .. (485.8,60.29) ;
\draw    (486.37,66.29) .. controls (491.01,68.29) and (490.01,76.87) .. (486.37,77.23) ;
\draw    (486.93,90.62) .. controls (490.11,90.8) and (491.71,100) .. (486.93,102.03) ;
\draw    (488.2,118.95) .. controls (492.11,120.37) and (492.15,129.12) .. (489.08,130.33) ;
\draw   (444.89,196.15) .. controls (444.89,192.87) and (464.74,190.22) .. (489.23,190.22) .. controls (513.71,190.22) and (533.56,192.87) .. (533.56,196.15) .. controls (533.56,199.43) and (513.71,202.08) .. (489.23,202.08) .. controls (464.74,202.08) and (444.89,199.43) .. (444.89,196.15) -- cycle ;
\draw    (489.49,201.63) .. controls (484.89,200.13) and (484.69,191.83) .. (489.23,190.22) ;
\draw    (489.23,190.22) .. controls (493,191.03) and (493.8,199.93) .. (489.49,201.63) ;
\draw   (443.89,196.15) .. controls (443.89,195.87) and (444.12,195.65) .. (444.39,195.65) .. controls (444.67,195.65) and (444.89,195.87) .. (444.89,196.15) .. controls (444.89,196.43) and (444.67,196.65) .. (444.39,196.65) .. controls (444.12,196.65) and (443.89,196.43) .. (443.89,196.15) -- cycle ;
\draw   (533.56,196.15) .. controls (533.56,195.87) and (533.79,195.65) .. (534.06,195.65) .. controls (534.34,195.65) and (534.56,195.87) .. (534.56,196.15) .. controls (534.56,196.43) and (534.34,196.65) .. (534.06,196.65) .. controls (533.79,196.65) and (533.56,196.43) .. (533.56,196.15) -- cycle ;

\draw (146.33,55.4) node [anchor=north west][inner sep=0.75pt]  {$\circ $};
\draw (146.33,95.4) node [anchor=north west][inner sep=0.75pt]  {$\circ $};
\draw (146.33,155.4) node [anchor=north west][inner sep=0.75pt]  {$\circ $};
\draw (196.33,72.9) node [anchor=north west][inner sep=0.75pt]    {$\bullet $};
\draw (196.33,132.9) node [anchor=north west][inner sep=0.75pt]    {$\bullet $};
\draw (296.33,73.57) node [anchor=north west][inner sep=0.75pt]    {$\bullet $};
\draw (296.33,133.9) node [anchor=north west][inner sep=0.75pt]    {$\bullet $};
\draw (336.33,155.4) node [anchor=north west][inner sep=0.75pt]  {$\circ $};
\draw (176.33,216.06) node [anchor=north west][inner sep=0.75pt]    {$\widetilde{\Gamma} \in G_{g,n}(\CP,\RP|\beta',\vec{\mu})$};
\draw (483.67,162.07) node [anchor=north west][inner sep=0.75pt]    {$\downarrow $};
\draw (426.34,190.06) node [anchor=north west][inner sep=0.75pt]    {$p_{1}$};
\draw (538.74,190.06) node [anchor=north west][inner sep=0.75pt]    {$p_{2}$};
\draw (458.94,216.06) node [anchor=north west][inner sep=0.75pt]    {$\left(\mathbb{P}^{1} ,\mathbb{RP}^{1}\right)$};
\draw (170.74,150.06) node [anchor=north west][inner sep=0.75pt]    {$l_j$};
\end{tikzpicture}
\caption{The left figure is an example of the graph $\widetilde{\Gamma}$ $\in$ $G_{g,n}(\bP^1,\bR\bP^1|\beta',\vec{\mu})$. 
The right figure is an example of stable map corresponding to $\widetilde{\Gamma}$.}
\label{fig:decorated-graph}
\end{figure}

Given $\widetilde{\Ga}\in G_{g,n}(\CP,\RP|\beta',\vec{\mu})$, we introduce the following notations:
\begin{itemize}
    \item (weight) We define
    \[
        {\bf w}(p_1) = -\sv, \quad {\bf w}(p_2) = \sv,
    \]
    For a flag $f\in F_v$, we define
    \begin{equation*}
        {\bf w}_f := \left\{
        \begin{aligned}
            &\frac{{\bf w}(p_{\vec{f}(v)})}{d_e}, \quad f=(e,v),
            \\
            &\frac{{\bf w}(p_{\vec{f}(v)})}{d_{l_j}}, \quad f=(l_j,v).    
        \end{aligned}
        \right.
    \end{equation*}
    \item (vertex contribution) Let $g\in \bZ_{\geq 0}$,
        \[
            {\bf h}(p_i,g) = \frac{\Lambda^\vee_g({\bf w}(p_i))}{{\bf w}(p_i)},
        \] 
        where $\Lambda^\vee_g(\su) := \sum_{i=0}^g (-1)^i\lambda_i\su^{g-i}$ and $\lambda_i$ is the $i$-th Hodge class on $\overline{\cM}_{g,n}$.
    \item (edge contribution) For $d\in \bZ_{>0}$, we define
        \[
            {\bf h}(e,d) = \frac{(-1)^dd^{2d}}{(d!)^2\sv^{2d}}.
        \]
    \item (disk factor) For $\mu\in\bZ_{>0}$, we define
    \begin{equation*}\label{eqn:disk-factor}
        D^1(\mu) = (-1)^{\mu + 1}\frac{\mu^{\mu-2}}{\mu!\sv^{\mu-2}},\quad
        D^2(\mu) = \frac{\mu^{\mu-2}}{\mu!\sv^{\mu-2}}.
    \end{equation*}
\end{itemize}

Following the fixed-locus formula proposed in \cite{BNPT22}, we define the $T$-equivariant open Gromov-Witten invariants of $(\bP^1,\bR\bP^1)$ as the following graph sum.
\begin{defn}\label{prop:localization} Let $\beta' = d[\bP^1] - \sum_{j\in J_-} \mu_j[\bD_1] + \sum_{j\in J_+}\mu_j[\bD_2]\in H_2(\CP,\RP)$. Then for $\gamma_1,\dots,\gamma_n\in H^*_T(\bP^1)$ and $g, a_1,\dots,a_n\in \bZ_{\geq 0}$, we define
    \begin{equation*}
        \begin{aligned}
            &\langle \gamma_1\psi^{a_1},\dots,\gamma_n\psi^{a_n}\rangle^{\OGW}_{g,\beta',\vec{\mu}} 
            \\
            &:= \sum_{\widetilde{\Ga}\in G_{g,n}(\CP,\RP|\beta',\vec{\mu})}\frac{1}{|\Aut(\widetilde{\Ga})|}\prod_{e\in E(\Ga)}\frac{{\bf h}(e,d_e)}{d_e}
            \prod_{v\in V(\Ga)}\Big({\bf w}(p_{\vec{f}(v)})^{|E_v|}\prod_{i\in S_v}i^*_{p_{\vec{f}(v)}}\gamma_{i}\Big)
            \\
            &\cdot\prod_{j\in J_-}D^1(-\mu_j)\prod_{j\in J_+}D^2(\mu_j)
            \prod_{v\in V(\Ga)}\int_{\overline{\cM}_{g_v,F_v\cup S_v}}\frac{{\bf h}(p_{\vec{f}(v)},g_v)\prod_{i\in S_v}\psi_i^{a_i}}{\prod_{l_j\in L_v}{\bf w}_{(l_j,v)}\prod_{f\in F_v}({\bf w}_{f}-\psi_{f})}.
        \end{aligned}
    \end{equation*}
    We use the following convention for the unstable integrals:
\begin{equation*}
    \begin{aligned}
        \int_{\overline{\cM}_{0,1}}\frac{1}{{\bf w}-\psi} = {\bf w},
        \quad
        \int_{\overline{\cM}_{0,2}}\frac{\psi_2^a}{{\bf w}-\psi_1} = (-{\bf w})^a&, \quad a\in \bZ_{\geq 0},
        \\
        \int_{\overline{\cM}_{0,2}}\frac{1}{({\bf w}_1-\psi_1)({\bf w}_2-\psi_2)}= \frac{1}{{\bf w}_1+{\bf w}_2}&.
    \end{aligned}
\end{equation*}
\end{defn}

\begin{remark} \rm
    There is a definition of equivariant open Gromov-Witten invariants for 3-dimensional Calabi-Yau smooth toric Deligne-Mumford stacks in \cite[Section 3.6]{FLT}.
In \cite{FLT}, the virtual normal bundle $N^\vir$ is defined via the tangent-obstruction complex. Each connected component of the fixed locus $F$ is a compact orbifold without boundary and 
the virtual fundamental class $[F]^\vir = [F]$. The equivariant open Gromov-Witten invariants are defined by integrating the inverse of the Euler class of $N^\vir$ (denoted by $e_T(N^\vir)^{-1}$) over $[F]^\vir$ (see \cite[Equation (13)]{FLT}).
In the case of $(\CP,\RP)$, the equivariant open Gromov-Witten invariants \emph{via integration over the fixed locus} can be defined similarly. After the computation of $e_T(N^\vir)^{-1}$, we obtain the right hand side of the expression in Definition \ref{prop:localization}.
\end{remark}

\begin{remark}
    \rm We set $\mu_i\neq 0$ for geometric reasons. If $\mu_i=0$ for some $i$, then the fixed locus $\overline{\cM}_{(g,h), n}(\bP^1,\bR\bP^1|\beta',\vec{\mu})^T$ is empty.
  The $T$-action on $\overline{\cM}_{(g,h), n}(\bP^1,\bR\bP^1|\beta',\vec{\mu})$ is given by $(t\cdot u)(p) := t\cdot u(p)$ for $p\in\Si$,
  where $t\in T$ and $u:(\Si,\partial\Si)\rightarrow (\CP,\RP)$ is a map. The map $u$ lies in the fixed locus if and only if
  there exists a domain automorphism $\varphi_t: (\Si,\partial\Si)\rightarrow (\Si,\partial \Si)$ such that $u\circ \varphi_t = t\cdot u$. 
  It forces that $u\in \overline{\cM}_{(g,h), n}(\bP^1,\bR\bP^1|\beta',\vec{\mu})^T$ only if $u|_{R_i}(z)=z^{\mu_i}$ for $\mu_i\neq 0$.
\end{remark}

\subsection{Open/descendant correspondence}
Given $u:(\Si,\partial\Si, {\bf x})\rightarrow (\CP,\RP)$ that represents a point $\xi\in \overline{\cM}_{(g,h),n}(\CP,\RP|\beta',\vec{\mu})^T$,
we have 
\[
    \Si = C \cup \bigcup_{j=1}^h D_j,
\]
where $C$ is a closed nodal curve of genus $g$ with marked points $x_1,\dots,x_n$, and $D_j$'s are holomorphic disks.
$C$ and $D_j$ intersect at $y_j$. Let $\hat{u}= u|_C$ and $u_j=u|_{D_j}$.
Then $\hat{u}:(C,{\bf x}, {\bf y})\rightarrow \CP$ represents a point $\hat{\xi}\in\overline{\cM}_{g,n+h}(\bP^1,\beta)^{T}$, 
where the $(n+j)$-th marked point $x_{n+j}$ is $y_j$.

The following definition is introduced in \cite[Definition 52]{Liu13}.

\begin{defn}[Decorated graphs (b)]
    Define $G_{g,n+h}(\CP,\beta)$ to be the set of all decorated graphs $\hat{\Ga}=(\Ga,\vec{f},\vec{d},\vec{g},\vec{s})$ defined as follows. 
    Let $n\in\bZ_{\geq 0}$, $h\in \bZ_{>0}$ and $\beta=d[\bP^1]\in E(\CP)$. A genus $g$, $(n+h)$-pointed,
    degree $\beta$ decorated graph for $\CP$ is a tuple $\hat{\Ga}=(\Ga,\vec{f},\vec{d}, \vec{g},\vec{s})$ consisting of 
    the following data.
    \begin{enumerate}
        \item $\Gamma$ is a compact, connected 1-dimensional CW complex. Let $V(\Ga)$ denote
        the set of vertices in $\Ga$, denoted by $\bullet$. 
        Let $E(\Ga)$ denote the set of edges, where an edge $e$ is a line connecting two $\bullet$ vertices.
        Let $F(\Ga)$ be the set of flags:
        \[
            \{(e,v)\in E(\Ga)\times V(\Ga):v\in e\}.
        \]
        For each $v\in V(\Ga)$, let $F_v$ (\emph{resp.} $E_v$) denote the flags (\emph{resp.} edges) attached to $v$, and let $\val(v)=|F_v|=|E_v|$ denote the number of flags (\emph{resp.} edges) incident to $v$.
        \item The \emph{label map} $\vec{f}: V(\Ga)\rightarrow \{1,2\}$ labels each $\bullet$ with a number.
            If $v_1,v_2\in V(\Ga)$ are connected by an edge, we require
            $\vec{f}(v_1)\neq \vec{f}(v_2)$.
        \item The \emph{degree map} $\vec{d}:E(\Ga)\rightarrow \bZ_{>0}$ sends an edge $e$ to a positive integer $\vec{d}(e)=d_e$.
        \item The \emph{genus map} $\vec{g}:V(\Ga)\rightarrow\bZ_{\geq 0}$ sends a vertex $v\in V(\Ga)$ to a nonnegative integer $g_v$.
        \item The \emph{marking map} $\vec{s}:\{1,2,\dots,n+h\}\rightarrow V(\Ga)$. For each $v\in V(\Ga)$, define $S_v := \vec{s}^{-1}(v)$, and $n_v=|S_v|$.
    \end{enumerate}
    The data is required to satisfy the following conditions:
    \begin{itemize}
        \item [(i)] (genus) $g = \sum_{v\in V(\Ga)}g_v + |E(\Ga)| - |V(\Ga)| + 1$.
        \item [(ii)] (degree) $d = \sum_{e\in E(\Ga)}d_e$.
    \end{itemize}
\end{defn}

Let $F_{\hat{\Ga}}$ be the connected component corresponding to $\hat{\Ga}$. We have
\[
    \overline{\cM}_{g,n+h}(\bP^1,\beta)^{T} = \bigcup_{\hat{\Ga}\in G_{g,n+h}(\bP^1,\beta)}F_{\hat{\Ga}}.
\]
By \cite[Theorem 73]{Liu13}, we get
\begin{prop}\label{prop:localization-liu} Let $\beta=d[\CP]\in E(\CP)$. Then for $\gamma_1,\dots,\gamma_{n+h}\in H^*_T(\CP)$ and $g, a_1,\dots, a_{n+h}\in\bZ_{\geq 0}$, we have
    \begin{equation*}
    \begin{aligned}
        &\langle \gamma_1\psi^{a_1},\dots,\gamma_{n+h}\psi^{a_{n+h}}\rangle^{\CP,T}_{g,n+h,\beta} 
            \\
            &= \sum_{\hat{\Ga}\in G_{g,n+h}(\CP,\beta)}\frac{1}{|\Aut(\hat{\Ga})|}\prod_{e\in E(\Ga)}\frac{{\bf h}(e,d_e)}{d_e}
            \prod_{v\in V(\Ga)}\Big({\bf w}(p_{\vec{f}(v)})^{|E_v|}\prod_{i\in S_v}i^*_{p_{\vec{f}(v)}}\gamma_{i}\Big)
            \\
            &\cdot\prod_{v\in V(\Ga)}\int_{\overline{\cM}_{g_v,E_v\cup S_v}}\frac{{\bf h}(p_{\vec{f}(v)},g_v)\prod_{i\in S_v}\psi_i^{a_i}}{\prod_{e\in E_v}({\bf w}_{(e,v)}-\psi_{(e,v)})}.
    \end{aligned}
\end{equation*}
\end{prop}

Let $\widetilde{\Ga}$ and $\hat{\Ga}$ be the decorated graphs corresponding to the map $u$ and $\hat{u}$ respectively.
There is a map 
\begin{equation}\label{eq:graph-map}
    G_{g,n}(\CP,\RP|\beta',\vec{\mu})\rightarrow G_{g,n+h}(\CP,\beta): \widetilde{\Ga}\mapsto\hat{\Ga}.
\end{equation}    
by cutting all leaves $l_j$ of $\widetilde{\Ga}$, and then labeling the attached $\bullet$ vertex of $l_j$ with $n+j$. 
Because the leaves $l_j$ are labeled, we have $\Aut(\widetilde{\Ga})=\Aut(\hat{\Ga})$.

By Definition \ref{prop:localization}, Equation \eqref{eq:graph-map} and Proposition \ref{prop:localization-liu}, we give a open/descendant correspondence theorem:

\begin{theorem}\label{thm:open-descendant} Let $\beta' = d[\bP^1] - \sum_{j\in J_-} \mu_j[\bD_1] + \sum_{j\in J_+}\mu_j[\bD_2]\in H_2(\CP,\RP)$. Then for $\gamma_1,\dots,\gamma_n\in H^*_T(\CP)$ and $g, a_1,\dots, a_n\in\bZ_{\geq 0}$, we have
    \begin{equation*}
    \begin{aligned}
   & \langle \gamma_1\psi^{a_1},\dots,\gamma_n\psi^{a_n}\rangle^{\OGW}_{g,\beta',\vec{\mu}}
    = \prod_{j\in J_-}D^1(-\mu_j)\prod_{j\in J_+}D^2(\mu_j)\ \ \ \ \ \ \ \ \ \ \ \ \
    \\
    & \cdot \int_{[\overline{\cM}_{g,n+h}(\bP^1,\beta)]^{\vir,T}}\frac{\prod_{i=1}^{n}\psi_i^{a_i}\ev_i^*(\gamma_i)\prod_{j\in J_-}\ev_{n+j}^* \phi_1\prod_{j\in J_+}\ev_{n+j}^* \phi_2}
        {\prod_{j=1}^h\frac{\sv}{\mu_j}(\frac{\sv}{\mu_j}-\psi_{n+j})}.
    \end{aligned}
\end{equation*}
\end{theorem}

\subsection{Generating function of open Gromov-Witten invariants}\label{sec:gen.fun.}
Consider the generating function of genus $g$, $h$ boundary circles open Gromov-Witten invariants of $(\CP,\RP)$: let $\bt=t^01 + t^1H$,
\[
    F_{g,h}(\bt,Q;X_1,\dots,X_h) = \sum_{\beta\in E(\bP^1)\atop {l\geq 0}}\sum_{\vec{\mu}=(\mu_1,\dots,\mu_h)\atop {\mu_i\in\bZ_{\neq 0}}}
    \frac{Q^\beta}{l!}\langle\bt^l\rangle^{\OGW}_{g,\beta',\vec{\mu}}\prod_{i=1}^{h}X_{i}^{\mu_i}.
\]
Let
\[
    \tilde{\xi}^1_k(X) = \sum_{d<0}(\frac{\sv}{d})^{-(k+2)}D^1(-d)X^d,    
\]
\[
    \tilde{\xi}^2_k(X) = \sum_{d>0}(\frac{\sv}{d})^{-(k+2)}D^2(d)X^d.
\]
For $\alpha \in \{1,2\}$,
\[
    \Big(\frac{1}{\sv}X\frac{d}{dX}\Big)\tilde{\xi}^\alpha_k(X) =
    \tilde{\xi}^\alpha_{k+1}(X),
\]
\[
    \tilde{\xi}^\alpha(z,X):= \sum_{k\in\bZ_{\geq -2}}z^k\tilde{\xi}^\alpha_k(X).
\]

With the above notations and the Theorem \ref{thm:open-descendant}, we have
\begin{prop}\label{prop:open-generating-function}
    \rm
    \begin{itemize}
        \item [(1)](disk invariants)
\begin{equation*}
    \begin{aligned}
                & \ \ \ \ \ F_{0,1}(\bt,Q;X) 
        \\
                &= 
                \sum_{d<0}S_z(1,\phi_1)\Big|_{z=\frac{\sv}{d}} D^1(-d)X^d + \sum_{d>0} S_z(1,\phi_2)\Big|_{z=\frac{\sv}{d}} D^2(d)X^d
        \\
                &= \sum_{\alpha\in\{1,2\}}\Big(\frac{1}{\Delta^\alpha}\tilde{\xi}^\alpha_{-2}(X) + (\frac{t^0}{\Delta^\alpha}+\frac{t^1}{2})\tilde{\xi}^\alpha_{-1}(X)
        +\sum_{k\geq 0} \llangle \phi_\alpha\psi^k\rrangle_{0,1}^{\bP^1,T}\tilde{\xi}^\alpha_k(X)\Big)
        \\
                &= [z^{-2}]\sum_{\alpha \in \{1,2\}} S_z(1,\phi_\alpha)\tilde{\xi}^\alpha(z,X).
    \end{aligned}
\end{equation*}

        \item[(2)] (annulus invariants)
\begin{equation*}
    \begin{aligned}
                & \ \ \ \ \ F_{0,2}(\bt,Q;X_1,X_2)-F_{0,2}(0;X_1,X_2)
        \\
        &= \sum_{k_1,k_2\geq 0}\sum_{\alpha_1,\alpha_2\in\{1,2\}}\llangle \phi_{\alpha_1}\psi^{k_1},\phi_{\alpha_2}\psi^{k_2}\rrangle_{0,2}^{\bP^1,T}
        \tilde{\xi}_{k_1}^{\alpha_1}(X_1)\tilde{\xi}_{k_2}^{\alpha_2}(X_2)
        \\
        &= [z_1^{-1}z_2^{-1}]\sum_{\alpha_1,\alpha_2\in\{1,2\}}\llangle \frac{\phi_{\alpha_1}}{z_1-\psi_1},
                \frac{\phi_{\alpha_2}}{z_2-\psi_2}\rrangle_{0,2}^{\bP^1,T}\tilde{\xi}^{\alpha_1}(z_1,X_1)\tilde{\xi}^{\alpha_2}(z_2,X_2),
    \end{aligned}
\end{equation*}
        where
\begin{equation*}
    \begin{aligned}
                \frac{1}{\sv}\Big(X_1\frac{\partial}{\partial X_1}+X_2\frac{\partial}{\partial X_2}\Big)F_{0,2}(0;X_1,X_2)= \sum_{\alpha\in\{1,2\}}\frac{1}{\Delta^\alpha}\tilde{\xi}^\alpha_{0}(X_1)\tilde{\xi}^\alpha_{0}(X_2).
            \end{aligned}
        \end{equation*}
        So we have
        \begin{equation*}
            \begin{aligned}
                & \ \ \ \ \ \frac{1}{\sv}\Big(X_1\frac{\partial}{\partial X_1}+X_2\frac{\partial}{\partial X_2}\Big)F_{0,2}(\bt,Q;X_1,X_2)
        \\
                & = [z_1^{0}z_2^{0}](z_1+z_2)\sum_{\alpha_1,\alpha_2\in\{1,2\}}V_{z_1,z_2}(\phi_{\alpha_1},\phi_{\alpha_2})
        \tilde{\xi}^{\alpha_1}(z_1,X_1)\tilde{\xi}^{\alpha_2}(z_2,X_2).
    \end{aligned}
\end{equation*}
        \item[(3)] For $2g-2+n>0$,
\begin{equation*}
    \begin{aligned}
                & \ \ \ \ \ F_{g,n}(\bt,Q;X_1,\dots, X_n) 
        \\
                &= \sum_{k_1,\dots,k_n\geq 0}\sum_{\alpha_1,\dots,\alpha_n \in \{1,2\}}
        \llangle\phi_{\alpha_1}\psi^{k_1},\dots,\phi_{\alpha_n}\psi^{k_n}\rrangle_{g,n}^{\bP^1,T}
        \prod_{j=1}^{n}\tilde{\xi}_{k_j}^{\alpha_j}(X_j)
        \\
                &= [z_1^{-1}\dots z_n^{-1}]\sum_{\alpha_1,\dots,\alpha_n\in\{1,2\}}
        \llangle\frac{\phi_{\alpha_1}}{z_1-\psi_1},\dots,\frac{\phi_{\alpha_n}}{z_n-\psi_n}\rrangle_{g,n}^{\bP^1,T}
        \prod_{j=1}^{n}\tilde{\xi}^{\alpha_j}(z_j,X_j).
    \end{aligned}
\end{equation*}
    \end{itemize}
\end{prop}

Let $\vec{\Ga}\in{\bf\Ga}_{g,n}(\bP^1)$ be a labeled graph defined in Section \ref{sec:A-model-graph-sum}. To each ordinary leaf $l_j\in L^o(\Ga)$ with
$\beta(l_j)=\beta\in\{1,2\}$ and $k(l_j)=k\in \bZ_{\geq 0}$, we assign the following weight (open leaf)
\begin{equation}\label{eqn:A-model-open-leaf}
    (\tilde{\cL}^O)^\beta_k(l_j) = [z^k](\sum_{\alpha,\gamma\in\{1,2\}}
    (\tilde{\xi}^\alpha(z,X_j)S^{\hat{\underline{\gamma}}}_{\ \alpha})_+
    R(-z)_\gamma^{\ \beta}).
\end{equation}

We define a new weight $\widetilde{\omega}^{O}_A(\vec{\Ga})$ of a labeled graph $\vec{\Ga}\in {\bf\Ga}_{g,n}(\bP^1)$ as 
\begin{equation*}
    \begin{aligned}
        \widetilde{\omega}^{O}_A(\vec{\Ga}) = &\prod_{v\in V(\Ga)}\Big(\sqrt{\Delta^{\beta(v)}(q)}\Big)^{2g(v)-2+\val(v)}\langle\prod_{h\in H(v)}\tau_{k(h)}\rangle_{g(v)}
        \prod_{e\in E(\Ga)}\cE^{\beta(v_1(e)),\beta(v_2(e))}_{k(h_1(e)),k(h_2(e))}
        \\
        & \cdot \prod_{l\in L^1(\Ga)}(\cL^1)^{\beta(l)}_{k(l)}(l)\prod_{j=1}^n(\tilde{\cL}^{O})^{\beta(l_j)}_{k(l_j)}(l_j).
    \end{aligned}
\end{equation*}
Then by Theorem \ref{thm:descendant-graph-sum} and Proposition \ref{prop:open-generating-function}, we have the following graph sum formula for $F_{g,n}$.
\begin{theorem}\label{thm:open-A-graph-sum}  \rm
    For $2g-2+n>0$, we have
    \[
        F_{g,n}(\bt,Q;X_1,\dots,X_n) = \sum_{\vec{\Ga}\in {\bf \Ga}_{g,n}(\bP^1)}\frac{\widetilde{\omega}_A^O(\vec{\Ga})}{|\Aut(\vec{\Ga})|}.
    \]
\end{theorem}

\begin{defn}\rm We define 
    \[
        F_{g,n}(\bt;X_1,\dots,X_n) := F_{g,n}(\bt,1;X_1,\dots,X_n).
    \]
    By Remark \ref{rmk:Q-is-1} and Proposition \ref{prop:open-generating-function}, the restriction of $F_{g,n}$ to $Q=1$ is well-defined.
\end{defn}

Theorem \ref{thm:open-A-graph-sum} implies
\begin{coro}\label{thm:open-A-graph-sum-Qone}\rm
    Let $\omega_A^O(\vec{\Ga}) = \widetilde{\omega}_A^O(\vec{\Ga})\big|_{Q=1}$. For $2g-2+n>0$, we have
    \[
        F_{g,n}(\bt;X_1,\dots,X_n) = \sum_{\vec{\Ga}\in {\bf \Ga}_{g,n}(\bP^1)}\frac{\omega_A^O(\vec{\Ga})}{|\Aut(\vec{\Ga})|}.
    \]
\end{coro}
\begin{convention}
    From this point onwards, we use the specialization $Q=1$, $q=e^{t^1}$ in the calculation.
\end{convention}

For $g,n,m\geq 0$, $a_i\in\bZ_{\geq 0}$ and $\gamma_i\in H^*_T(\bP^1)$, we also define the following double correlators as the generating function of open Gromov-Witten invariants with $m$ descendant insertions and $n$ boundary components:
\begin{align*}
    & \llangle \gamma_1\psi^{a_1},\dots,\gamma_m\psi^{a_m}\rrangle_{g,m,n}^{\OGW}(\bt,X_1,\dots,X_n)
    \\
    &:=  \sum_{\beta\in E(\CP)\atop {l\geq 0}}\sum_{\vec{\mu}=(\mu_1,\dots,\mu_n)\atop{\mu_i\in\bZ_{\neq 0}}}\frac{1}{l!}\langle\gamma_1\psi^{a_1},\dots,\gamma_m\psi^{a_m},\bt^l\rangle_{g,\beta',\vec{\mu}}^{\OGW} \prod_{i=1}^{n}X_i^{\mu_i}.
\end{align*}
We also have
\begin{equation*}\label{eqn:open-correlator}
    \begin{aligned}
    & \llangle \gamma_1\psi^{a_1},\dots,\gamma_m\psi^{a_m}\rrangle_{g,m,n}^{\OGW}(\bt,X_1,\dots,X_n)
    \\
    & = \sum_{k_1,\dots,k_n\geq 0\atop {\alpha_1,\dots,\alpha_n \in \{1,2\}}}
        \llangle\gamma_1\psi^{a_1},\dots,\gamma_m\psi^{a_m},\phi_{\alpha_1}\psi^{k_1},\dots,\phi_{\alpha_n}\psi^{k_n}\rrangle_{g,n}^{\bP^1,T}
        \prod_{j=1}^{n}\tilde{\xi}_{k_j}^{\alpha_j}(X_j).
\end{aligned}
\end{equation*}

For a labeled graph $\vec{\Ga}\in{\bf\Ga}_{g,m+n}(\bP^1)$ with $L^o(\Ga) = \{l_1,\dots,l_{m+n}\}$,
we introduce a new weight $\sw_A(\vec{\Ga})$ for the labeled graph $\vec{\Ga}\in{\bf \Ga}_{g,m+n}(\bP^1)$ as
\begin{equation*}
    \begin{aligned}
        \sw_A(\vec{\Ga}) = &\prod_{v\in V(\Ga)}\Big(\sqrt{\Delta^{\beta(v)}(q)}\Big)^{2g(v)-2+\val(v)}\langle\prod_{h\in H(v)}\tau_{k(h)}\rangle_{g(v)}
        \prod_{e\in E(\Ga)}\cE^{\beta(v_1(e)),\beta(v_2(e))}_{k(h_1(e)),k(h_2(e))}
        \\
        & \cdot \prod_{l\in L^1(\Ga)}(\cL^1)^{\beta(l)}_{k(l)}(l)\prod_{j=1}^{m}(\cL^{\bf u})_{k(l_j)}^{\beta(l_j)}(l_j)\prod_{j=1}^n(\tilde{\cL}^{O})^{\beta(l_{m+j})}_{k(l_{m+j})}(l_{m+j})\Big|_{Q=1}.
    \end{aligned}
\end{equation*}

\begin{theorem}\label{thm:A-SYZ-open}
    Suppose $g,m\geq 0$, $n\geq 1$ and $2g-2+m+n>0$. It holds that
    $$
        \llangle {\bf u}_1,\dots,{\bf u}_m\rrangle_{g,m,n}^{\OGW}(\bt,X_1,\dots,X_n) = \sum_{\vec{\Ga}\in{\bf \Ga}_{g,m+n}(\bP^1)}\frac{\sw_A(\vec{\Ga})}{|\Aut(\vec{\Ga})|}.
    $$
\end{theorem}

\subsection{Explicit formula of $F_{0,1}$}\label{sec:explicit-disk}
Following from Proposition \ref{prop:open-generating-function}, the generating function $F_{0,1}$ is given by
\[
    F_{0,1}(\bt;X) = \sum_{d>0}\Big(
        J_1(-\sv/d)D^1(d)X^{-d}+
        J_2(\sv/d)D^2(d)X^d
    \Big),
\]
where $J_\alpha(z) := (J(z),\phi_\alpha)_{\CP,T}$ are the components of the $J$-function in Section \ref{sec:J-function}.

By Equation (\ref{eq:J-function}), for $d>0$
\begin{equation*}
    \begin{aligned}
        J^1(-\sv/d)&= -\sv J_1(-\sv/d)
        = e^{-dt^0/\sv}(-\sv/d)^d\Ga(d+1)I_d(-2\sqrt{q}d/\sv),
    \\
        J^2(\sv/d) &= \sv J_2(\sv/d)
        = e^{dt^0/\sv}(\sv/d)^d\Ga(d+1)I_d(2\sqrt{q}d/\sv).
    \end{aligned}
\end{equation*}
We get
\[
    F_{0,1}(\bt;X) = \sum_{d>0}e^{-dt^0/\sv}\frac{\sv}{d^2}I_d(-2\sqrt{q}d/\sv)X^{-d}
    + \sum_{d>0}e^{dt^0/\sv}\frac{\sv}{d^2}I_d(2\sqrt{q}d/\sv)X^d.
\]
Let $q, \sv$ be positive real numbers. By the symmetry of the modified Bessel function $I_\alpha(x)$ (see Appendix \ref{sec:Bessel}), we have
\begin{equation}\label{eqn:explicit-disk}
        F_{0,1}(\bt;X) = \sum_{d\in\bZ_{\neq 0}}e^{dt^0/\sv}\frac{\sv}{d^2}I_d(2\sqrt{q}d/\sv)X^d.
\end{equation}

\section{Mirror curve and topological recursion}\label{sec:Bmodel}
\subsection{Mirror curve}\label{sec:curve}
Let $Y$ be a coordinate on $\bC^*$ and let $q=e^{t^1},\sv$ be positive real numbers. 
The $T$-equivariant superpotential $W:\bC^*\rightarrow \bC$ is defined as
\[
    W(Y) = t^0 + \frac{\sv}{2}\log q + Y + \frac{q}{Y} - \sv\log Y.
\]
Consider the spectral curve $C_q$ defined as follows:
\[
    C_q = \{(x,y)\in\bC^2: x = W(e^y)\}.
\]
Let $X= e^{-x/\sv}$, $Y=e^y$. The curve $C_q$ has a parameterization:
\begin{equation*}
    \left\{
        \begin{aligned}
            x(Y) &= t^0 + \frac{\sv}{2}\log q + Y + \frac{q}{Y} - \sv\log Y,
            \\
            y(Y) &= \log Y.
        \end{aligned}
    \right.
\end{equation*}
Let $\bP^1$ be the compactification of $\bC^*$ with $Y\in\bC^*\subset \bP^1$ as its coordinate.
Consider the differential $dx(Y)= (Y-q/Y-\sv)dy$. The branch points $P_\alpha$ of $dx$ are given by
\[
    P_1 = \frac{\sv - \sv\sqrt{1+4q/\sv^2}}{2},
    \quad 
    P_2 = \frac{\sv + \sv\sqrt{1+4q/\sv^2}}{2}.
\]

Near each branch point $Y=P_\alpha$, we choose a local coordinates $\zeta_\alpha$ such that 
\[
    x=\check{u}^\alpha+\zeta_\alpha^2,
\quad
    y = \check{v}^\alpha + \sum_{k=1}^\infty h^\alpha_k(q)\zeta^k_\alpha.
\]
Let $\lambda = xdy$ be the Liouville form on $\bC^2$, $\Phi := \lambda|_{C_q}$.
Let $B(Y_1,Y_2)$ be a holomorphic bidifferential on $\bP^1$ defined by
\[
    B(Y_1,Y_2) = \frac{dY_1dY_2}{(Y_1-Y_2)^2}.
\]
Around $(P_\alpha, P_\beta)$, we expand $B(Y_1,Y_2)$ as 
\[
    B(\zeta_\alpha,\zeta_\beta) = \Big(\frac{\delta_{\alpha\beta}}{(\zeta_\alpha - \zeta_\beta)^2}+\sum_{k,l\geq 0}B^{\alpha,\beta}_{k,l}\zeta^k_\alpha\zeta^l_\beta\Big)
    d\zeta_\alpha d\zeta_\beta.
\]
Given $\alpha=1,2$ and $d\in\bZ_{\geq 0}$, define
    \[
        d\xi_{\alpha,d}(p) := (2d-1)!!2^{-d}\Res_{p'\rightarrow P_\alpha}
        B(p,p')\zeta_\alpha^{-2d-1}.
    \]
    Then $d\xi_{\alpha,d}$ is a meromorphic 1-form on $\bP^1$ with a single pole of order $2d+2$ at $P_\alpha$.
    In the local coordinate $\zeta_\alpha$ near $P_\alpha$, $d\xi_{\alpha, d}$ has the following expansion:
    \[
        d\xi_{\alpha, d} = \left(
            \frac{(2d+1)!!}{2^d\zeta_\alpha^{2d+2}}+ \text{analytic part in $\zeta_\alpha$}
        \right)d\zeta_\alpha.
    \]    
    In particular, we have 
    \[
        d\xi_{\alpha,0} = \frac{\sqrt{2}}{\sqrt{\Delta^\alpha(q)}}d\left(\frac{P_\alpha}{Y-P_\alpha}\right).
    \]

\subsection{The $\fh_X$-operator}
The function $X=e^{-x/\sv}$ has an essential singularity at $Y=0$:
\[
    X= e^{-x/\sv} = \frac{e^{-t^0/\sv}}{\sqrt{q}}Ye^{-(Y+q/Y)/\sv}.
\]
We also notice that $dX/X$ is a meromorphic differential on $\bP^1$:
\[
    \frac{dX}{X} = -\frac{dx}{\sv} = \frac{Y^2-\sv Y - q}{-\sv Y^2}dY = \frac{(Y-P_1)(Y-P_2)}{-\sv Y^2}dY.
\]
Let $D_\epsilon$ be a punctured disk around $Y=0$: 
\[
    D_\epsilon := \{Y\in\bC^*: 0<|Y|<\ep\},
\]
where $\ep$ is a small positive real number such that $D_\epsilon$ does not contain
any ramification points $\{P_1,P_2\}$.
In the punctured disk $D_\epsilon$, $Y$ can be expanded as a power series of $X$ by the Lagrange inversion theorem.
At $Y=0$, we define the operator $\fh^\circ_X$ and $\fh^\bullet_X$ that expand holomorphic functions $f(Y)$ and holomorphic differential forms 
$\theta(Y)$ on $D_\epsilon$ into Laurent series in $X$, based on their local behavior at $Y=0$:
\begin{itemize}
    \item For a holomorphic function $f(Y)$ on $D_\epsilon$, define $\fh^\circ_X:\cO(D_\ep)\rightarrow \bC[\![X^{\pm 1}]\!]$ by
    \[
        \fh^\circ_X(f) := \sum_{\mu\in\bZ_{\neq 0}} \Big(\mathop{\Res}\limits_{Y\rightarrow 0}f(Y)X^{-\mu}\frac{dX}{X}\Big)X^{\mu}.
    \]
    \item For a holomorphic differential form $\theta(Y)$ on $D_\epsilon$, define $\fh^\bullet_X:\Omega^1(D_\ep)\rightarrow \bC[\![X^{\pm 1}]\!]$ by
    \[     
        \fh_X^\bullet(\theta) := \sum_{\mu\in\bZ_{\neq 0}} \Big(\mathop{\Res}\limits_{Y\rightarrow 0} \theta(Y) X^{-\mu}\Big)\frac{X^{\mu}}{\mu}.
    \]
    If $f$ is a single-valued holomorphic function on $D_\ep$ such that $\theta(Y)=df$, then $\fh_X^\bullet(\theta)=\fh_X^\circ(f)$ by the integrations by parts.
\end{itemize}
For multi-holomorphic functions (\emph{resp.} forms) on $(D_\epsilon)^{\times n}$, we define $\fh_{X_1,\dots,X_n}^{\circ}$ and $\fh_{X_1,\dots,X_n}^\bullet$ as follows:
\begin{itemize}
    \item Let $f(Y_1,\dots,Y_n)$ be a holomorphic function on $(D_\epsilon)^{\times n}$, we define
    \[
        \fh^\circ_{X_1,\dots,X_n}(f) := \sum_{\mu_1,\dots,\mu_n\in\bZ_{\neq 0}}
        \Big(
            \mathop{\Res}\limits_{Y_1\rightarrow 0}\cdots \mathop{\Res}\limits_{Y_n\rightarrow 0}f(Y_1,\dots,Y_n)\prod_{i=1}^{n}X^{-\mu_i}_i\frac{dX_i}{X_i}
        \Big)\prod_{i=1}^{n}X^{\mu_i}_i.
    \]
    \item Let $\theta(Y_1,\dots,Y_n)$
    be a holomorphic differential form from the set of sections $\Gamma((D_\ep)^{\times n}, \omega_{C_q}^{\boxtimes n})$, we define
    \[  
        \fh^\bullet_{X_1,\dots,X_n}(\theta) := \sum_{\mu_1,\dots,\mu_n\in\bZ_{\neq 0}}
        \Big(
            \mathop{\Res}\limits_{Y_1\rightarrow 0}\cdots \mathop{\Res}\limits_{Y_n\rightarrow 0}\theta(Y_1,\dots,Y_n)\prod_{i=1}^{n}X^{-\mu_i}_i
        \Big)\prod_{i=1}^{n}\frac{X^{\mu_i}_i}{\mu_i}.
    \] 
\end{itemize}

The $\fh_X$-operator plays an essential role in the definition of open leaves in Section \ref{sec:TR}.
Let
\[
    \xi_{\alpha,0}(Y) = \frac{\sqrt{2}}{\sqrt{\Delta^\alpha(q)}}\frac{P_\alpha}{Y-P_\alpha},
\]
which is a meromorphic function on $\bP^1$. It has a simple pole at $P_\alpha$ and is holomorphic elsewhere.
Let $$\eta := \xi_{2,0}-\sqrt{-1}\xi_{1,0},\quad \chi := \xi_{2,0}+\sqrt{-1}\xi_{1,0}.$$
Let $q,\sv>0$, $\Delta^2(q)=\sv\sqrt{1+4q/\sv^2}>0$.
Then we have
\begin{equation*}
    \begin{aligned}
        \eta &= \sqrt{\frac{2}{\Delta^2(q)}}\left(\frac{P_2}{Y-P_2}-\frac{P_1}{Y-P_1}\right)
        \\
        &= \sqrt{\frac{2}{\Delta^2(q)}}\frac{Y\Delta^2(q)}{(Y-P_1)(Y-P_2)}
        \\
        &= \frac{1}{-\sv}\sqrt{2\Delta^2(q)}\frac{dY}{dX/X}\frac{1}{Y},
    \end{aligned}
\end{equation*}
and
\begin{equation*}
    \begin{aligned}
        \chi &= \sqrt{\frac{2}{\Delta^2(q)}}\left(\frac{P_2}{Y-P_2}+\frac{P_1}{Y-P_1}\right)
        \\
        &= \sqrt{\frac{2}{\Delta^2(q)}}\frac{\sv Y+2q}{(Y-P_1)(Y-P_2)}
        \\
        &= -\sqrt{\frac{2}{\Delta^2(q)}}\left(
            \frac{dY}{dX/X}\frac{1}{Y} + \frac{2qdY}{\sv dX/X}\frac{1}{Y^2}
        \right).
    \end{aligned}
\end{equation*}

Apply the operator $\fh_X^\circ$ to the meromorphic functions $\eta$ and $\chi$, we obtain two Laurent series:
\[
    \fh_X^\circ(\eta) := \sum_{\mu \in\bZ_{\neq 0}}A_\mu X^\mu, \quad \fh^\circ_X(\chi) := \sum_{\mu\in\bZ_{\neq 0}}B_\mu X^\mu.
\]
By calculation,
\begin{equation}\label{eqn:eta}
    \begin{aligned}
        A_\mu &= \Res_{Y=0}\eta X^{-\mu}\frac{dX}{X}
        \\
        &= \frac{1}{-\sv}\sqrt{2\Delta^2(q)}e^{\mu t^0/\sv}\sqrt{q}^\mu\Res_{Y=0}e^{\mu(Y+q/Y)/\sv}\frac{dY}{Y^{\mu+1}}
        \\
        &= \frac{1}{-\sv}\sqrt{2\Delta^2(q)}e^{\mu t^0/\sv}I_\mu\Big(\frac{2\sqrt{q}\mu}{\sv}\Big),
    \end{aligned}
\end{equation}
and
\begin{equation}\label{eqn:chi}
    \begin{aligned}
        B_\mu &= \Res_{Y=0}\chi X^{-\mu}\frac{dX}{X}
        \\
        &= -\sqrt{\frac{2}{\Delta^2(q)}}e^{\mu t^0/\sv}\sqrt{q}^\mu\Res_{Y=0}e^{\mu(Y+q/Y)/\sv}\left(\frac{dY}{Y^{\mu+1}}
        +\frac{2q}{\sv}\frac{dY}{Y^{\mu+2}}\right)
        \\
        &= -\sqrt{\frac{2}{\Delta^2(q)}}e^{\mu t^0/\sv}\left(I_\mu\Big(\frac{2\sqrt{q}\mu}{\sv}\Big)+ \frac{2\sqrt{q}}{\sv}I_{\mu+1}\Big(\frac{2\sqrt{q}\mu}{\sv}\Big)\right)
        \\
        &= -\sqrt{\frac{2}{\Delta^2(q)}}e^{\mu t^0/\sv}\frac{2\sqrt{q}}{\sv}I'_{\mu}\Big(\frac{2\sqrt{q}\mu}{\sv}\Big).
\end{aligned}
\end{equation}
\begin{remark}\rm
    More details about the residue calculation can be found in Appendix \ref{sec:residue}.
\end{remark}

\subsection{Chekhov-Eynard-Orantin topological recursion}\label{sec:TR}
Let $\omega_{g,n}$ be the differential forms
defined recursively by the Chekhov-Eynard-Orantin topological recursion \cite{EO07}:
$$\omega_{0,1}=0,\quad \omega_{0,2}=B(Y_1,Y_2).$$
For $2g-2+n>0$,
\begin{equation*}
    \begin{aligned}
        \omega_{g,n}(Y_1,\dots,Y_n) = \sum_{\alpha\in\{1,2\}}
        &\Res_{Y\rightarrow P_\alpha}\frac{\int_{\xi=Y}^{\hat{Y}}B(Y_n,\xi)}{2(\log(Y)-\log(\hat{Y}))dx}\Big(
            \omega_{g-1,n+1}(Y,\hat{Y},Y_1,\dots,Y_{n-1})
        \\
        &+ 
            \sum_{g_1+g_2=g}\sum_{I\sqcup J = \{1,\dots,n-1\}}\omega_{g_1,|I|+1}(Y,Y_{I})\omega_{g_2,|J|+1}(\hat{Y},Y_J)
        \Big),
    \end{aligned}
\end{equation*}
where $Y\neq P_\alpha$ is in a neighbourhood of $P_\alpha$, and $\hat{Y}\neq Y$ is the conjugate point of $Y$ such that $x(\hat{Y})=x(Y)$.

We introduce the following notations:
\begin{enumerate}
    \item Given $\alpha = 1,2$ and $k\in \bZ_{\geq 0}$, define
        \[
            W^\alpha_k := d((-\frac{d}{dx})^k(\xi_{\alpha,0})),
        \]
        \[
            \theta_\alpha(z) := \sum_{k=0}^{\infty}d\xi_{\alpha,k}z^k,\quad
            \hat\theta_\alpha(z):= \sum_{k=0}^\infty W^\alpha_kz^k.
        \]
    \item The B-model $R$-matrix $\check{R}_\beta^{\ \alpha}(z)$ (which is a power series of $z$) is defined by asymptotic expansion
        \[
            \check{R}_\beta^{\ \alpha}(z) \sim \frac{\sqrt{-z}e^{-\check{u}^\alpha/z}}{2\sqrt{\pi}}\int_{\gamma_\alpha}e^{x/z}d\xi_{\beta,0}.
        \]
        Here $\gamma_\alpha$ is the Lefschetz thimble of the map $x$, i.e. $x(\gamma_\alpha)-\check{u}^\alpha \in \bR_{\geq 0}$.
    \item Given $\alpha,\beta\in\{1,2\}$, $k,l\geq 0$, define
        \begin{equation*}
            \begin{aligned}
                \check{B}^{\alpha,\beta}_{k,l} &:= \frac{(2k-1)!!(2l-1)!!}{2^{k+l+1}}B^{\alpha,\beta}_{2k,2l}
                \\
                &\ = 
                [z^kw^l]\Big(\frac{1}{z+w}\Big(\delta_{\alpha\beta}-\sum_{\gamma\in\{1,2\}}\check{R}_\gamma^{\ \alpha}(-z)\check{R}_\gamma^{\ \beta}(-w)\Big)\Big).
            \end{aligned}
        \end{equation*}
    \item Given $\alpha=1,2$ and $k\in \bZ_{\geq 1}$, define     
        \begin{equation*}
            \begin{aligned}
                \check{h}^\alpha_k &:= \frac{(2k-1)!!}{2^{k-1}}h^\alpha_{2k-1}
                \\
                & \ = [z^{k-1}]\Big(\sum_\beta h^\beta_1\check{R}_\beta^{\ \alpha}(-z)\Big).
            \end{aligned}
        \end{equation*}
\end{enumerate}

As in Section \ref{sec:A-model-graph-sum}, we introduce the B-model graph sum formula for $\omega_{g,n}$.
Let $\bp=(Y_1,\dots,Y_n)\in(\bP^1)^n$.
For a labeled graph $\vec{\Ga}\in{\bf\Ga}_{g,n}(\bP^1)$ with $L^o(\Ga) = \{l_1,\dots,l_n\}$, and $\bullet = \bp$ or $O$,
we assign the weight of $\vec{\Ga}$ as
\begin{equation*}
    \begin{aligned}
            \omega_B^\bullet(\vec{\Ga}) &= \prod_{v\in V(\Ga)}\Big(\frac{h^{\beta(v)}_1}{\sqrt{2}}\Big)^{2-2g(v)-\val(v)}
        \langle\prod_{h\in H(v)}\tau_{k(h)}\rangle_{g(v)}
        \\
        & \ \ \ \ \times \prod_{e\in E(\Ga)}\check{B}^{\beta(v_1(e)),\beta(v_2(e))}_{k(h_1(e)),k(h_2(e))}
        \prod_{l\in L^1(\Ga)}\frac{-1}{\sqrt{2}}\check{h}^{\beta(l)}_{k(l)}\prod_{j=1}^n 
        (\check{\cL}^\bullet)^{\beta(l_j)}_{k(l_j)}(l_j),
    \end{aligned}
\end{equation*}
where
\begin{itemize}
    \item (descendant leaf)
        \[
            (\check{\cL}^\bp)^{\beta(l_j)}_{k(l_j)}(l_j) = \frac{-1}{\sqrt{2}}d\xi_{\beta(l_j),k(l_j)}(Y_j),
        \]
    \item (open leaf)
        \[
            (\check{\cL}^O)^{\beta(l_j)}_{k(l_j)}(l_j) = \frac{-1}{\sqrt{2}} \fh_{X_j}^\bullet(d\xi_{\beta(l_j),k(l_j)}(Y_j)).
        \]
\end{itemize}

We have the B-model graph sum formula from \cite[Theorem 3.7]{DOSS}:
\begin{theorem}[Dunin-Barkowski-Orantin-Shadrin-Spitz \cite{DOSS}]\rm
    \label{thm:B-graph-sum}
    For $2g-2+n>0$,
    \begin{equation*}
        \begin{aligned}
            \omega_{g,n}(\bp) &= \sum_{\vec{\Ga}\in{\bf \Ga}_{g,n}(\bP^1)}\frac{\omega_B^\bp(\vec{\Ga})}{|\Aut(\vec{\Ga})|},
            \\
            \fh^\bullet_{X_1,\dots,X_n}(\omega_{g,n}) &= \sum_{\vec{\Ga}\in{\bf \Ga}_{g,n}(\bP^1)}\frac{\omega_B^O(\vec{\Ga})}{|\Aut(\vec{\Ga})|}.
        \end{aligned}
    \end{equation*}
\end{theorem}

We define the B-model open primary potentials as follows,
\begin{itemize}
    \item Define the B-model disk potential $W_{0,1}(t^0,q;X)$ as a Laurent series in $X$ with constant term zero such that
        \[
            \Big(\frac{1}{\sv}X\frac{d}{d X}\Big)W_{0,1}(t^0,q;X) = \fh_X^\bullet\Big(\frac{\partial \Phi}{\partial t^0}\Big). 
        \]
    \item Let $\tilde{\omega}_{0,2}(Y_1,Y_2)$ be the meromorphic 2-form on $\bC^*\times \bC^*$,
    $$\tilde{\omega}_{0,2}(Y_1,Y_2):=\omega_{0,2}(Y_1,Y_2)-\frac{dX_1dX_2}{(X_1-X_2)^2}.$$
    $\tilde{\omega}_{0,2}$ is holomorphic on $(D_\ep)^{\times 2}$.
    Define the B-model annulus invariants by
        \[
            W_{0,2}(t^0,q;X_1,X_2) = \fh_{X_1,X_2}^\bullet(\tilde{\omega}_{0,2}(Y_1,Y_2)).
        \]
    \item For $2g-2+n>0$, $\omega_{g,n}$ is holomorphic at $(D_\ep)^{\times n}$, we define
        \[
            W_{g,n}(t^0,q; X_1,\dots,X_n) = \fh_{X_1,\dots,X_n}^\bullet(\omega_{g,n}(\bp)).
        \]
\end{itemize}

\subsection{Strominger-Yau-Zaslow T-dual}\label{sec:SYZmirror}
\begin{definition}[equivariant K-theoretic framing]
    We define 
    $$
        \widetilde{\ch}_z: K_T(\CP)\rightarrow H^*_T(\CP;\bQ)\left[\!\left[\frac{\sv}{z}\right]\!\right]
    $$
    by the following two properties:
    \begin{itemize}
        \item [(a)] $\widetilde{\ch}_z$ is a homomorphism of additive groups:
        $$
            \widetilde{\ch}_z(\cE_1\oplus\cE_2) = \widetilde{\ch}_z(\cE_1)+\widetilde{\ch}_z(\cE_2).
        $$
        \item [(b)] If $\cE$ is a $T$-equivariant line bundle on $\CP$, then
        $$
            \widetilde{\ch}_z(\cE) = \exp\left(-\frac{2\pi\sqrt{-1}(c_1)_T(\cE)}{z}\right).
        $$
    \end{itemize}
    For any $\cE\in K_T(\CP)$, we define the \emph{$K$-theoretic framing} of $\cE$ by
$$
    \kappa_z(\cE) := (-z)^{1-(c_1)_T(T\CP)/z}\Gamma\left(1-\frac{(c_1)_T(T\CP)}{z}\right)\widetilde{\ch}_z(\cE),
$$
where $(c_1)_T(T\CP)=2H$.
\end{definition}

\begin{definition}[equivariant SYZ T-dual]\label{SYZ-dual}
Let $\cE=\cO_{\bP^1}(l_1 p_1 + l_2 p_2)$ be an equivariant ample line bundle on $\bP^1$, where $l_1, l_2$
are integers such that $l_1+l_2>0$. We define the equivariant Strominger-Yau-Zaslow (SYZ) T-dual \cite{SYZ96} $\SYZ(\cE)$ of $\cE$ to be the
oriented path in $\bC$ (see Figure \ref{fig:SYZmirror}).

We extend the definition $\bZ$-linearly to the equivariant K-theory group $K_T(\bP^1)$. If an element of $K_T(\bP^1)$ is written as a $\bZ$-linear combination of ample line bundle, 
then its SYZ T-dual is defined to be the corresponding formal $\bZ$-linear combination of the associated oriented paths. 
Integrals over such formal sum are defined by linearity, namely as the sum of the integrals over the individual oriented paths.
\end{definition}
\begin{figure}[h]
\begin{center}
\setlength{\unitlength}{1.6mm}
\begin{picture}(60,20)

\put(10,6){\vector(1,0){10}}
\put(20,6){\line(1,0){10}}
\put(30,6){\vector(0,1){5}}
\put(30,11){\line(0,1){5}}
\put(30,16){\vector(1,0){10}}
\put(40,16){\line(1,0){10}}

\put(3,2){$-\infty+(-2l_1-1)\pi i$}
\put(27,2){$(-2l_1-1)\pi i$}
\put(27,18){$(2l_2-1)\pi i$}
\put(42,18){$+\infty+(2l_2-1)\pi i$}

\end{picture}
   \caption{SYZ $T$-dual of $\cE$ in $\bC$}
  \label{fig:SYZmirror}
\end{center}
\end{figure}

Suppose $g,m,n\geq 0$ and $2g-2+m+n>0$, for any $\cE_1,\dots,\cE_m\in K_T(\CP)$, we define the B-model open descendant potential as 
\begin{equation}\label{eqn:SYZ-B-potential}
    \begin{aligned}
        &W_{g,m,n}(t^0,q;\cE_1,\dots,\cE_m;X_1,\dots, X_n) 
        \\
        &:= \int_{y_1\in \SYZ(\cE_1)}\cdots \int_{y_m\in \SYZ(\cE_m)}e^{\sum_{i=1}^m x(y_i)/z_i}\fh_{X_1,\dots,X_n}^\bullet(\omega_{g,m+n}),
    \end{aligned}
\end{equation}
where $\int_{y_i\in \SYZ(\cE_i)}$ acts on the $i$-th variables of $\omega_{g,m+n}$ and $\fh^\bullet_{X_1,\dots,X_n}$ acts on the last $n$ variables. 

Consider a labeled graph $\vec{\Ga}\in{\bf\Ga}_{g,m+n}(\bP^1)$ with $L^o(\Ga) = \{l_1,\dots,l_{m+n}\}$. For the ordinary leaf $l_j\in L^o(\Ga)$ with $1\leq j\leq m$, we assign the SYZ weight (SYZ leaf factor)
$$
    (\check{\cL}^\SYZ)^{\beta(l_j)}_{k(l_j)}(l_j) = \frac{-1}{\sqrt{2}}\int_{y_j\in \SYZ(\cE_j)}e^{x(y_j)/z_j}d\xi_{\beta(l_j),k(l_j)}(Y_j),
$$
and for the ordinary leaf $l_{m+j}\in L^o(\Ga)$ with $1\leq j\leq n$, we assign the open leaf factor
$$
    (\check{\cL}^O)^{\beta(l_{m+j})}_{k(l_{m+j})}(l_{m+j}) = \frac{-1}{\sqrt{2}}\fh^\bullet_{X_{j}}(d\xi_{\beta(l_{m+j}),k(l_{m+j})}(Y_{m+j})).
$$
We have the following B-model graph sum formula of $W_{g,m,n}$ with $2g-2+m+n>0$:
\begin{equation}\label{eqn:B-SYZ-open}
    W_{g,m,n}(t^0,q;\cE_1,\dots,\cE_m;X_1,\dots, X_n) = \sum_{\vec{\Ga}\in {\bf \Ga}_{g,m+n}(\bP^1)}\frac{{\sw}_B(\vec{\Ga})}{|\Aut(\vec{\Ga})|},
\end{equation}
where 
\begin{equation*}
    \begin{aligned}
        {\sw}_B(\vec{\Ga}) &= \prod_{v\in V(\Ga)}\Big(\frac{h^{\beta(v)}_1}{\sqrt{2}}\Big)^{2-2g(v)-\val(v)}
        \langle\prod_{h\in H(v)}\tau_{k(h)}\rangle_{g(v)}
        \\
        & \ \ \ \ \times \prod_{e\in E(\Ga)}\check{B}^{\beta(v_1(e)),\beta(v_2(e))}_{k(h_1(e)),k(h_2(e))}
        \prod_{l\in L^1(\Ga)}\frac{-1}{\sqrt{2}}\check{h}^{\beta(l)}_{k(l)}
        \\
        & \ \ \ \ \times \ \ \prod_{j=1}^m 
        (\check{\cL}^\SYZ)^{\beta(l_j)}_{k(l_j)}(l_j)\prod_{j=1}^n(\check{\cL}^O)^{\beta(l_{m+j})}_{k(l_{m+j})}(l_{m+j}).
    \end{aligned}
\end{equation*}

\section{Mirror Symmetry}\label{sec:MS}
\subsection{Mirror symmetry of disk invariants}
Let $\Phi=xdy|_{C_q}$ which is a 1-form on $C_q$. Let 
\[
    \Phi_0 := \frac{\partial \Phi}{\partial t^0} = \frac{dY}{Y}.
\]
Consider 
\[
    \fh_X^\bullet(\Phi_0) = \sum_{\mu\in\bZ_{\neq 0}} R_\mu X^\mu,
\]
where
\begin{equation*}
    \begin{aligned}
        R_\mu 
        &= \frac{1}{\mu}\Res_{Y=0}X^{-\mu}\frac{dY}{Y}
        \\
        &= \frac{1}{\mu}e^{\mu t^0/\sv}\sqrt{q}^\mu\Res_{Y=0}e^{\mu(Y+q/Y)/\sv}\frac{dY}{Y^{\mu+1}}
        \\
        &= \frac{1}{\mu}e^{\mu t^0/\sv}I_\mu\Big(\frac{2\sqrt{q}\mu}{\sv}\Big).
    \end{aligned}
\end{equation*}
Compare with Equation \eqref{eqn:explicit-disk}, we have
\[
    \frac{\partial}{\partial t^0}F_{0,1}(\bt; X)=\Big(\frac{1}{\sv}X\frac{d}{d X}\Big)F_{0,1}=\fh_X^\bullet(\Phi_0).
\]
Therefore, we obtain the following theorem:
\begin{theorem}
    [Mirror symmetry of disk invariants]\label{thm:disk}
    \[
        F_{0,1}(\bt; X) = W_{0,1}(t^0,q;X).
    \]
\end{theorem}

\subsection{Identification of open leaves}
We define 
\[
    U(z)(\bt,X) := \sum_{\alpha\in\{1,2\}} \tilde{\xi}^\alpha(z,X)S_z(1,\phi_\alpha)\Big|_{Q=1}.
\]
For $m\in\bZ_{\geq -2}$, 
\[
    \left(\frac{1}{\sv}X\frac{d}{dX}\right)[z^{m}](U(z)(\bt,X)) = [z^{m+1}](U(z)(\bt,X)).
\]

In Section \ref{sec:explicit-disk}, we have shown that
\[
    F_{0,1} = [z^{-2}]U(z)(\bt,X) = \sum_{d\in\bZ_{\neq0}}e^{dt^0/\sv}\frac{\sv}{d^2}I_d(2\sqrt{q}d/\sv)X^d,
\]
\[
    [z^0]U(z)(\bt,X) = \left(\frac{1}{\sv}X\frac{d}{dX}\right)^2F_{0,1}= \frac{1}{\sv}\sum_{d\in\bZ_{\neq0}}e^{dt^0/\sv}I_d(2\sqrt{q}d/\sv)X^d.
\]
By quantum differential equation, 
\[
    z\frac{\partial U}{\partial t^1} = \sum_{\alpha\in\{1,2\}}\tilde{\xi}^\alpha(z,X)S_z(H,\phi_\alpha)\Big|_{Q=1},
\]
and 
\begin{equation*}
    \begin{aligned}
            [z^{-1}]\frac{\partial U}{\partial t^1}
            &= \frac{\partial}{\partial t^1}\left(\frac{1}{\sv}X\frac{d}{dX}\right)[z^{-2}](U(z)(\bt,X))
            \\
            &= \frac{\partial}{\partial t^1}\sum_{d\in\bZ_{\neq 0}}e^{dt^0/\sv}\frac{1}{d}I_d(2\sqrt{q}d/\sv)X^d 
            \\
            &= \sum_{d\in\bZ_{\neq 0}}e^{dt^0/\sv}\frac{\sqrt{q}}{\sv}I'_d(2\sqrt{q}d/\sv)X^d.
    \end{aligned}
\end{equation*}
Comparing with Equation \eqref{eqn:eta} and Equation \eqref{eqn:chi}, we have
\begin{equation*}
    \begin{aligned}
        \fh_X^\circ(\eta) &= -\sqrt{2\Delta^2(q)}[z^0]U(z)(\bt,X),
        \\
        \fh_X^\circ(\chi) &= -2\sqrt{\frac{2}{\Delta^2(q)}}[z^0]\sum_{\alpha\in\{1,2\}}\tilde{\xi}^\alpha(z,X)S_z(H,\phi_\alpha)\Big|_{Q=1}.
    \end{aligned}
\end{equation*}
Recall that
\[
    \hat{\phi}_\alpha(q) = (\Psi^{-1})_\alpha^{\ 0} 1 + (\Psi^{-1})_\alpha^{\ 1} H,
\]
\[
    \sum_{\beta\in\{1,2\}}\tilde{\xi}^\beta(z,X)S(\hat{\phi}_\alpha(q),\phi_\beta)\Big|_{Q=1}
    = (\Psi^{-1})_\alpha^{\ 0} U + z(\Psi^{-1})_\alpha^{\ 1}\frac{\partial U}{\partial t^1}.
\]
We obtain
\begin{equation*}
    \begin{aligned}
        [z^0]\sum_{\beta\in\{1,2\}}\tilde{\xi}^\beta(z,X)S(\hat{\phi}_1(q),\phi_\beta)\Big|_{Q=1} &=
        [z^0]\frac{\sqrt{-\Delta^2(q)}}{2}U + [z^{-1}]\frac{1}{\sqrt{-\Delta^2(q)}}\frac{\partial U}{\partial t^1}
    = -\frac{\fh_X^\circ(\xi_{1,0})}{\sqrt{2}},    
    \\
       [z^0]\sum_{\beta\in\{1,2\}}\tilde{\xi}^\beta(z,X)S(\hat{\phi}_2(q),\phi_\beta)\Big|_{Q=1} &=
        [z^0]\frac{\sqrt{\Delta^2(q)}}{2}U + [z^{-1}]\frac{1}{\sqrt{\Delta^2(q)}}\frac{\partial U}{\partial t^1}
        = -\frac{\fh_X^\circ(\xi_{2,0})}{\sqrt{2}}.
    \end{aligned}
\end{equation*}

Therefore, for $\alpha\in\{1,2\}$, $k> 0$,
\begin{equation}\label{eqn:A-B-open}
    \begin{aligned}
        [z^k]\sum_{\beta\in\{1,2\}}\tilde{\xi}^\beta(z,X)S(\hat{\phi}_\alpha(q),\phi_\beta)\Big|_{Q=1} &= \fh_X^\circ\Big(
        (-\frac{d}{dx})^k\Big(\frac{-1}{\sqrt{2}}\xi_{\alpha,0}\Big)\Big),
        \\
        \Big(\sum_{\beta\in\{1,2\}}\tilde{\xi}^\beta(z,X)S(\hat{\phi}_\alpha(q),\phi_\beta)\Big|_{Q=1}\Big)_+ &= \fh_X^\bullet\Big(\sum_{k\geq 0} 
        \frac{-1}{\sqrt{2}}W^\alpha_k z^k\Big) 
        = -\fh_X^\bullet\Big(\frac{\hat{\theta}_\alpha(z)}{\sqrt{2}}\Big).
    \end{aligned}
\end{equation}

\begin{lma}\label{lma:theta-R-matrix}
    We have
    \[
        \theta_\alpha(z) = \sum_{\beta\in\{1,2\}}\check{R}_\beta^{\ \alpha}(-z)\hat{\theta}_\beta(z).
    \]
\end{lma}
\begin{proof}
    The lemma follows from
    \[
        d\xi_{\alpha,k} = \sum_{i=0}^{k}\sum_{\beta\in\{1,2\}}([z^{k-i}]\check{R}_\beta^{\ \alpha}(-z))W_i^\beta,
    \]
    which is shown in the proof of \cite[Theorem A]{FLZ17}.
\end{proof}
Then we have the identification of open leaves:
\begin{theorem}\label{thm:open-leaves}
    For each ordinary leaf $l_j$ with $\beta(l_j)=\beta\in\{1,2\}$ and $k(l_j)=k\in\bZ_{\geq 0}$, we have
    \[
        (\tilde{\cL}^O)^\beta_k(l_j)\big|_{Q=1} = (\check{\cL}^O)^\beta_k(l_j).
    \]
\end{theorem}
\begin{proof}
    By \cite[Proposition 3.10]{FLZ17}, the A- and B-model $R$-matrices are equal:
    $$R_\beta^{\ \alpha}(z)\big|_{Q=1}=\check{R}_\beta^{\ \alpha}(z).$$
    Then Theorem \ref{thm:open-leaves} follows from Equation \eqref{eqn:A-model-open-leaf}, Equation \eqref{eqn:A-B-open}, and Lemma \ref{lma:theta-R-matrix}.
\end{proof}

\subsection{Mirror symmetry of annulus invariants}
Define 
\begin{equation*}
    \begin{aligned}
        C(Y_1,Y_2) &:= (-\frac{\partial}{\partial x(Y_1)}-\frac{\partial}{\partial x(Y_2)})
        \Big(\frac{\omega_{0,2}}{dx(Y_1)dx(Y_2)}\Big)(Y_1,Y_2)dx(Y_1)dx(Y_2)
        \\
        & \ = \Big(-d_1\circ\frac{1}{dx(Y_1)}-d_2\circ\frac{1}{dx(Y_2)}\Big)(\tilde{\omega}_{0,2}(Y_1,Y_2)).
    \end{aligned}
\end{equation*}

The following proposition is proved in \cite[Lemma 6.9]{FLZ20b}.
\begin{prop} We have
\[
    C(Y_1,Y_2) = \frac{1}{2}\sum_{\alpha\in\{1,2\}}d\xi_{\alpha,0}(Y_1)d\xi_{\alpha,0}(Y_2).
\]
\end{prop}
\begin{theorem}
    [Mirror symmetry of annulus invariants]\label{thm:annulus}
    \[
        F_{0,2}(\bt;X_1,X_2) = W_{0,2}(t^0,q;X_1,X_2).
\]
\end{theorem}
\begin{proof}

\begin{equation*}
    \begin{aligned}
    & \ \ \ \ \ \mathfrak{h}^\bullet_{X_1,X_2}(C)= 
    \mathfrak{h}^\bullet_{X_1,X_2}(\frac{1}{2}\sum_{\alpha\in\{1,2\}}d\xi_{\alpha,0}(Y_1)d\xi_{\alpha,0}(Y_2))
    \\
    &= [z_1^0z_2^0]\sum_{\alpha,\beta,\gamma\in\{1,2\}}
    \tilde{\xi}^\beta(z_1,X_1)\tilde{\xi}^\gamma(z_2,X_2)
    S_{z_1}(\hat{\phi}_\alpha(q),\phi_\beta)S_{z_2}(\hat{\phi}_\alpha(q),\phi_\gamma)\big|_{Q=1}
    \\
    &= [z_1^0z_2^0](z_1+z_2)\sum_{\beta,\gamma\in\{1,2\}}V_{z_1,z_2}(\phi_\beta,\phi_\gamma)
    \tilde{\xi}^\beta(z_1,X_1)\tilde{\xi}^\gamma(z_2,X_2)\big|_{Q=1}
    \\
    &= \frac{1}{\sv}(X_1\frac{\partial}{\partial X_1}+X_2\frac{\partial}{\partial X_2})F_{0,2}(\bt;X_1,X_2).
    \end{aligned}
\end{equation*}
Here, the third equality is using Equation \eqref{eqn:V-to-S}.
By the integrations by parts,
\begin{equation*}
    \begin{aligned}
        \fh^\bullet_{X_1,X_2}(C) &= \fh^\bullet_{X_1,X_2}\Big(\Big(-d_1\circ\frac{1}{dx(Y_1)}-d_2\circ\frac{1}{dx(Y_2)}\Big)\tilde{\omega}_{0,2}\Big) 
        \\
        &= \Big(\frac{1}{\sv}X_1\frac{\partial}{\partial X_1}+\frac{1}{\sv}X_2\frac{\partial}{\partial X_2}\Big)W_{0,2}(t^0,q;X_1,X_2).
    \end{aligned}
\end{equation*}
It shows the equality for the terms $X_1^{\mu_1}X_2^{\mu_2}$ with $\mu_1+\mu_2\neq 0$. The terms $X_1^{\mu}X_2^{-\mu}$ follow from direct computation.
\end{proof}

For $(g,m,n)=(0,1,1)$, we define $W_{0,1,1}(t^0,q;\cE_1;X_2)$ as the unique Laurent series of $X_2$ determined by
\begin{equation*}
    \begin{aligned}
        \Big(\frac{1}{\sv}X_2\frac{d}{d X_2} + z_1^{-1}\Big)W_{0,1,1}(t^0,q;\cE_1;X_2) = \int_{y_1\in \SYZ(\cE_1)}e^{x(y_1)/z_1}\fh^\bullet_{X_2}(C).
    \end{aligned}
\end{equation*}
\begin{theorem}\label{thm:0-1-1mirror}
    \[    
        W_{0,1,1}(t^0,q;\cE_1;X_2) = \llangle\frac{\kappa_{z_1}(\cE_1)}{z_1-\psi_1}\rrangle_{0,1,1}^{\OGW}(\bt,X_2).
    \]
\end{theorem}
\begin{proof}By \cite[Equation (15)]{FLZ17}, we know 
    \[
        -z\int_{y\in\SYZ(\cE)}e^{x/z}\frac{d\xi_{\beta,0}}{\sqrt{2}} = \llangle\hat{\phi}_\beta(q),\frac{\kappa_{z}(\cE)}{z-\psi}\rrangle_{0,2}^{\CP,T}\big|_{Q=1}.
    \]
    Then
    \begin{equation*}
    \begin{aligned}
    & z_1\int_{y_1\in \SYZ(\cE_1)}e^{x(y_1)/z_1}\fh^\bullet_{X_2}(C)
    \\
    &= \sum_{\alpha\in\{1,2\}} z_1\int_{y_1\in \SYZ(\cE_1)}\frac{d\xi_{\alpha,0}(Y_1)}{\sqrt{2}}
    \cdot \mathfrak{h}_{X_2}^\bullet(\frac{d\xi_{\alpha,0}(Y_2)}{\sqrt{2}})
    \\
    &= [z_2^0]\sum_{\alpha,\gamma\in\{1,2\}}
    \tilde{\xi}^\gamma(z_2,X_2)
    S_{z_1}(\hat{\phi}_\alpha(q),\kappa_{z_1}(\cE_1))S_{z_2}(\hat{\phi}_\alpha(q),\phi_\gamma)\big|_{Q=1}
    \\
    &= [z_2^0]\sum_{\gamma=1,2}\tilde{\xi}^\gamma(z_2,X_2)\cdot (z_1+z_2)\llangle\frac{\kappa_{z_1}(\cE_1)}{z_1-\psi_1},\frac{\phi_{\gamma}}{z_2-\psi_2}\rrangle_{0,2}^{\CP,T}\big|_{Q=1}
    \\
    &= \Big(\frac{1}{\sv}X_2\frac{d}{d X_2}+z_1^{-1}\Big)z_1\sum_{\gamma\in\{1,2\}, k\geq 0} \tilde{\xi}^\gamma_k(X_2)\llangle \frac{\kappa_{z_1}(\cE_1)}{z_1-\psi_1},\phi_\gamma\psi_2^k\rrangle_{0,2}^{\CP,T}\big|_{Q=1}
    \\
    &= \Big(\frac{1}{\sv}X_2\frac{d}{d X_2}+z_1^{-1}\Big)z_1\llangle \frac{\kappa_{z_1}(\cE_1)}{z_1-\psi_1}\rrangle^{\OGW}_{0,1,1}.
    \end{aligned}
\end{equation*}
The third equality is using Equation \eqref{eqn:V-to-S}.
\end{proof}

\subsection{All genus mirror symmetry}
\begin{theorem}
    [All genus mirror theorem (a)]\label{thm:stableMS} For any $n>0$ and $g\geq 0$, we have
    \[
        F_{g,n}(\bt;X_1,\dots,X_n) = W_{g,n}(t^0,q;X_1,\dots,X_n).
    \]
\end{theorem}
\begin{proof}
    For the unstable cases $(g,n)=(0,1)$ and $(0,2)$, this theorem is Theorem \ref{thm:disk} and Theorem \ref{thm:annulus} respectively.

    For stable cases $2g-2+n>0$, it suffices to show the A- and B-model graph sum formula in Corollary \ref{thm:open-A-graph-sum-Qone} and Theorem \ref{thm:B-graph-sum} are equal.
    By \cite[Theorem A]{FLZ17}, we have
    \[
        R_\beta^{\ \alpha}(z)\big|_{Q=1} = \check{R}_\beta^{\ \alpha}(z),
    \]
    and 
    \[
        \frac{h^\alpha_1}{\sqrt{2}} = \frac{1}{\sqrt{\Delta^\alpha(q)}}.
    \]
    Hence, the weights in the graph sum match except for the open leaves. Since 
    Theorem \ref{thm:open-leaves} identifies the A- and B- open leaves, the theorem follows immediately.
\end{proof}

\begin{theorem}[All genus mirror theorem (b)]\label{thm:SYZ-open-mirror}
    Suppose $g,m,n\geq 0$, for any $\cE_1,\dots,\cE_m\in K_T(\CP)$, we have
    \begin{align*}
        &W_{g,m,n}(t^0,q;\cE_1,\dots,\cE_m;X_1,\dots,X_n)
        \\
        &=\llangle\frac{\kappa_{z_1}(\cE_1)}{z_1-\psi_1},\dots,\frac{\kappa_{z_m}(\cE_m)}{z_m-\psi_m}\rrangle_{g,m,n}^{\OGW}(\bt,X_1,\dots,X_n).
    \end{align*}
\end{theorem}
\begin{proof}
    For the case $n=0$, this theorem is Theorem \cite[Theorem B]{FLZ17}. For the case $m=0$, this theorem is Theorem \ref{thm:stableMS}.
    For the unstable case $(g,m,n)=(0,1,1)$, this theorem is Theorem \ref{thm:0-1-1mirror}.

    For the stable case, it suffices to show the A- and B-model graph sum formula in Theorem \ref{thm:A-SYZ-open} and \eqref{eqn:B-SYZ-open} are equal. Following the proof of Theorem \ref{thm:stableMS}, it remains to show the $\SYZ$ leaf matches. 
    By \cite[Equation (15)]{FLZ17}, we know 
    $$
        \int_{y\in \SYZ(\cE)}e^{x(y)/z}\frac{W^\beta_i}{\sqrt{2}} = -z^{-i-1}\llangle\hat{\phi}_\beta(q),\frac{\kappa_z(\cE)}{z-\psi}\rrangle_{0,2}^{\CP,T}\Big|_{Q=1}.
    $$
    By Lemma \ref{lma:theta-R-matrix}, we have
    \begin{equation*}
        (\check{\cL}^\SYZ)^{\beta(l_j)}_{k(l_j)}(l_j) = \sum_{i=0}^k\sum_{\gamma\in\{1,2\}}([z^{k-i}]\check{R}_\gamma^{\ \beta}(-z))z_j^{-i-1}\llangle\hat{\phi}_\gamma(q),\frac{\kappa_{z_j}(\cE_j)}{z_j-\psi}\rrangle_{0,2}^{\CP,T}\Big|_{Q=1}.
    \end{equation*}
    In A-model, ${\bf u}_j(z) = \sum_{\alpha\in\{1,2\}}{\bf u}_j^\alpha(z)\phi_\alpha(q)$ is replaced by
    \begin{equation}\label{eqn:A-u}
        {\bf u}_j^\alpha(z) = \sum_{a\geq 0}\big(\phi^\alpha(q),\kappa_{z_j}(\cE_j)\big)_{\CP,T}z_j^{-a-1}z^a,
    \end{equation}
    so 
    \begin{equation*}
        \begin{aligned}
            &[z^i]\left(\frac{{\bf u}_j^\alpha(z)}{\sqrt{\Delta^\alpha(q)}}S^{\hat{\underline{\gamma}}}_{\ \hat{\underline{\alpha}}}(z)\right)_+ 
            \\
            &= \sum_{b\geq 0}z_j^{-i-b-2}\llangle\hat{\phi}_\gamma(q),\tau_b(\hat{\phi}_\alpha(q))\rrangle_{0,2}^{\CP,T}\big(\hat{\phi}_\alpha(q),\kappa_{z_j}(\cE_j)\big)_{\CP,T}
            \\
            &= z_j^{-i-1}\llangle\hat{\phi}_\gamma(q),\frac{\kappa_{z_j}(\cE_j)}{z_j-\psi}\rrangle_{0,2}^{\CP,T}.
        \end{aligned}
    \end{equation*}
    In other words, for the ${\bf u}_j(z)$ corresponding to $\frac{\kappa_{z_j}(\cE_{j})}{z_j-\psi}$, we have
    $$
       (\cL^{\bf u})_{k(l_j)}^{\beta(l_j)}(l_j)\Big|_{Q=1} = (\check{\cL}^\SYZ)^{\beta(l_j)}_{k(l_j)}(l_j).
    $$
    The theorem follows immediately. 
\end{proof}

\section{Equivariant open Gromov-Witten theory of $(\CP,\RP)$ via coherent boundary conditions}\label{sec:geo-open}

\subsection{Geometric counting via coherent boundary conditions}\label{sec:geo-counting}
In previous sections, we established the counting of stable maps from bordered Riemann surfaces to $(\CP,\RP)$ by means of integration over the fixed locus of the moduli space $\overline{\cM}_{(g,h),n}(\CP,\RP|\beta',\vec{\mu})$.
We further demonstrated that these invariants are systematically encoded
by the topological recursion, via the local expansion of the multidifferential forms $\omega_{g,n}$ on the boundary of mirror curve $C_q$. 

In symplectic geometry, the enumeration of stable maps by means of integration over the global moduli space $\overline{\cM}_{(0,h),n}(\CP,\RP|\beta',\vec{\mu})$  
has also garnered significant interest \cite{BNPT22}. We refer to this global enumeration as \emph{geometric counting}. 
One approach defines it via multisections intersection \cite{PJT24}, while another employs the coherent boundary conditions and the localization formula for the orbifold with corners \cite{BNPT22}. 
In the non-equivariant limit, the construction using the coherent boundary conditions reduces to the non-equivariant counting defined by sections of bundles.

The geometric counting contains two contributions: one arises from the integration over the fixed locus $F$, which is the invariants defined in Section \ref{sec:virOGW} (Definition \ref{prop:localization}); the other one is the corner contribution.

The work \cite[Theorem 2.18, Definition 5.9]{BNPT22} proposes a combinatorial formula for the geometric counting, which is expressed solely in terms of fixed-point graphs.
In the genus 0 case, by carefully adjusting the coefficients preceding each fixed-point graph, 
they incorporate the corner contribution into the formula that comprises only fixed-point loci. 
Meanwhile, the authors conjecture that there exists a geometric construction via coherent boundary conditions whereby this combinatorial formula holds for geometric counting in all genera. 
If this conjecture is proven, the counting of fixed-point graphs, (i.e., the invariants we define in Section \ref{sec:virOGW}), 
will serve as the fundamental building block for calculating geometric intersection numbers. 
Our Mirror Theorem \ref{thm:stableMS} and \ref{thm:SYZ-open-mirror} provide a recursive algorithm for geometric counting.

\subsection{Combinatorial formula}\label{sec:comb-OGW}
A moduli specification $S$ is a decorated graph $S = (V_b \sqcup V_w, E, g, {\sf{I}}, \beta',\mu)$:
\begin{itemize}
    \item [(i)] $V_b$ is the set of black vertices, $V_w$ is the set of white vertices, and $E$ is the set of edges. 
    For each connected component $S'$, $S'$ contains exactly one black vertex, denoted by $v$,
    and $|V_w(v)|$ white vertices, where $V_w(v) \subseteq V_w$ denotes the white vertices adjacent to $v$.
    There is exactly a single edge connecting the black vertex and each white vertex in $V_w(v)$.
    \item [(ii)] $g, {\sf I}$, and $\beta'$ are functions on the black vertices, where $g:V_b\rightarrow \bZ_{\geq 0}$, ${\sf I}: V_b \rightarrow 2^{\mathbb{N}}$, 
    $$\beta'=(d_-,d_+):V_b\rightarrow H_2(\CP,\RP;\bZ)\cong \bZ\oplus\bZ.$$ 
    \item [(iii)] $\mu$ is a function on white vertices, i.e. $\mu:V_w\rightarrow\bZ$.
\end{itemize}
We require that 
\begin{itemize}
    \item [(a)] The label sets ${\sf I}(v)$ are pairwise disjoint for different black vertices.
    \item [(b)] The stability condition holds for each connected component $S'$ (with black vertex $v$):
    $$
        3(d_+(v)+d_-(v)) + 2g(v) + 2|{\sf I}(v)| + |V_w(v)| > 2.
    $$
    \item [(c)] For each connected component $S'$ (with black vertex $v$), it holds that
    $$
        d_+(v) - \sum_{w\in V_w(v) : \mu(w)>0} \mu(w) = d_-(v) + \sum_{w\in V_w(v) : \mu(w)<0} \mu(w) \geq 0.
    $$
\end{itemize}
An isomorphism between two moduli specifications is a graph isomorphism which respects the decorations.

Given the moduli space $\overline{\cM}_{(g,h),n}(\CP,\RP|\beta',\vec{\mu})$, one can associate a connected moduli specification $S = (V_b \sqcup V_w, E, g, {\sf{I}}, \beta',\mu)$ as follows
\begin{itemize}
    \item [(i)] There is exactly one black vertex in $V_b$ and $l(\vec{\mu})$ white vertices. 
    \item [(ii)] The black vertex is decorated by $g,{\sf I},\beta'$, corresponding to the genus, marked points and the degree of the moduli space.
    \item [(iii)] Each white vertex is decorated by a unique entry $\mu_i$ in $\vec{\mu}$. The white vertices are unordered, so the decoration is unique.
\end{itemize}

Given $\vec{a}=(a_i)_{i\in{\sf I}}$ and $\vec{\gamma} = (\gamma_i)_{i\in {\sf I}}$, where $a_i\in\bZ_{\geq 0}$ and $\gamma_i\in H^*_T(\bP^1)$. If the moduli specification $S$ is connected, we define 
$$
    I(S,\vec{a},\vec{\gamma}) = \langle\gamma_1\psi^{a_1},\dots,\gamma_n\psi^{a_n}\rangle^{\OGW}_{g,\beta',\vec{\mu}}.
$$
For disconnected moduli specification $S$ with finitely many connected components $S_i$, we define
$$
    I(S,\vec{a},\vec{\gamma}) = \prod_i I(S_i,\vec{a}|_{S_i}, \vec{\gamma}|_{S_i}).
$$
\par
A moduli specification is \emph{pure} if $\mu(w)\neq 0$ for each white vertex $w$. 
\begin{defn}\label{def:mor-decoration}\rm
A morphism from a moduli specification $S^{new}$ to a moduli specification $S^{old}$ is a decorated graph 
$M=(V_b\sqcup V_w, E, H^{CB},\widetilde{E},g, {\sf I},\beta',\mu,\sigma, \mu^{old})$ such that
\begin{itemize}
    \item [(a)] $(V_b\sqcup V_w, E,g, {\sf I},\beta',\mu)$ is the moduli specification $S^{new}$.
    \item [(b)] The wavy edges $\widetilde{E}$ connect white vertices.
    \item [(c)] The contracted boundary half-edges $H^{CB}$ are leaf-type flags which emanate from black vertices. 
    \item [(d)] A wavy flag is a pair of a wavy edge and an incident white vertex: $$\{(\tilde{e},v)\in \widetilde{E}\times V_w: v\in\tilde{e}\}.$$
     For any white vertex $v$, let $\si_v$ be a cyclic order of the wavy flags emanating from $v$.
    \item [(e)] An involution $\tau$ on wavy flags: if $v,v'$ are two incident vertices of $\tilde{e}\in\tilde{E}$, then $\tau(\tilde{e},v) := (\tilde{e},v')$.
    \item [(f)] The cyclic order $\si$ and the involution $\tau$ form the set of cycles of $\si^{-1}\circ\tau$: $$\quad \{(v_1, \tilde{e}_1,v_2,\dots,\tilde{e}_{m-1},v_m)~|~ v_i\in V_w, \tilde{e}_i\in\widetilde{E}, v_i, v_{i+1}\in \tilde{e}_i, v_1 = v_m\}/\sim,$$ 
    where the equivalence relation $\sim$ identifies tuples that differ by a cyclic permutation. Here for the white vertex without incident wavy edge, we assign the cycle to be its own. The function $\mu^{old}$ maps the set of cycles of $\si^{-1}\circ\tau$ to $\bZ$. 
    \item [(g)] $S^{old}$ is the desingularization of $M$, where the desingularization is the moduli specification as follows.
    \begin{itemize}
        \item [$\bullet$] For every connected component of the graph $G=(V_b\sqcup V_w, E\sqcup \widetilde{E})$, we associate a black vertex.
        \item [$\bullet$] For such a black vertex $v'$, we define 
        \begin{equation*}
            \quad \quad \beta'(v') = \sum_v \beta'(v), \quad {\sf I}(v') = \bigsqcup_v {\sf I}(v),\quad g(v') = h^1(G)+\sum_v g(v),
        \end{equation*}
        where we sum or union over the black vertices $v$ in $G$ associated to the connected component represented by $v'$.
        \item [$\bullet$] For a cycle $c$ from the set of $\si^{-1}\circ\tau$ in the connected component corresponding to $v'$, we assign a white vertex $w$ connecting $v'$, with $\mu(w)= \mu^{old}(c)$.
        \item [$\bullet$] For each contracted boundary half-edge $h$ in the connected component corresponding to $v'$, we assign a white vertex $w$ connecting $v'$, with $\mu(w)=0$.
    \end{itemize}
\end{itemize} 
\end{defn}
Let $\Hom(S^{new},S^{old})$ be the set of isomorphism types of morphisms from $S^{new}$ to $S^{old}$.

For a contracted boundary half-edge $h$ of $M$, which is attached to a black vertex $v$, we can assign a degree $d$ to it:
$$
    d(h) := d_+(v) - \sum_{w\in V_w(v) : \mu(w)>0} \mu(w),
$$
where $V_w(v)$ is the set of white vertices attached to $v$. 

In order to state the combinatorial formula proposed in \cite{BNPT22}, we also introduce the boundary contribution graphs.
\begin{defn}
    \rm
    A boundary contribution (BC) graph $G=(V,E,\vec{L},\mu)$ consists of the following data:
    \begin{itemize}
        \item [(a)] The set of vertices $V$. 
        \item [(b)] The set of edges $E=\vec{E}\sqcup\widetilde{E}$, where $\vec{E}$ is the set of directed edges and $\widetilde{E}$ is the set of wavy edges. Every vertex has exactly one incoming edge and one outgoing edge. 
         Each vertex is incident to exactly one wavy edge, and no wavy edge is a loop. 
        \item [(c)] The set of the oriented loops $\vec{L}$. 
        \item [(d)] The set of faces $F^{new}\cup F^{old}$. 
            \begin{itemize}
                \item [(i)] The new faces $F^{new}$ consist of oriented loops $\vec{L}$ and the oriented closed paths of $(V,\vec{E})$ (up to cyclic order). 
                Here, the oriented closed paths are understood to be primitive: we do not count paths obtained by repeating a shorter closed path several times.
                \item [(ii)] The old faces $F^{old}$ are formed by oriented loops $\vec{L}$ and all oriented sequences $s$ formed by $(V,\vec{E}\sqcup \widetilde{E})$, 
                where $s = (v_1, \vec{e}_1, v_2, \tilde{e}_2, v_3,\vec e_3,\dots,\tilde{e}_{2m})$ (up to cyclic order), $\vec{e}_{2i+1} \in \vec{E}$ connects $v_{2i+1}$ to $v_{2i+2}$, and $\tilde{e}_{2i}\in\widetilde{E}$ is the wavy edge connecting $v_{2i}$ and $v_{2i+1\pmod {2m}}$.
                All directed edges $\vec{e}_{2i+1}$ in an oriented sequence $s$ are different. 
                \item [(iii)] The intersection of the set of new faces and old faces is the set of loops, i.e. $F^{new}\cap F^{old} = \vec{L}$.
            \end{itemize}        
        \item [(e)] The perimeter function $\mu: F^{new}\cup F^{old}\rightarrow \bZ$.
    \end{itemize}
\end{defn}

\begin{defn}\rm
    A metric on a boundary contribution graph $G$ is an assignment of length $x_e\in\bR$ for every $e\in\vec{E}$ such that
    \begin{itemize}
        \item [(a)] The sum of lengths of the oriented edges of any circuit of the graph is an integer.
        \item [(b)] The sum of lengths of the oriented edges of a new or old face is the perimeter of that face.
        \item [(c)] If $e\in\vec{E}$ is an edge of a new face $f$, then $x_e \mu(f)>0$.
    \end{itemize}
\end{defn}
Let $W_G\subset\bR^{\vec{E}}$ be the space of metrics on $G$ and $\overline{W}_G$ be its closure in $\bR^{\vec{E}}$.
Let $\cE\subset\vec{E}$ be a set of directed edges such that $\cE\sqcup\widetilde{E}$ forms a spanning forest of $G$. We define 
\[
    \Omega_G = \bigwedge_{e\in\cE} dx_e,
\]
and define the volume of a BC graph $G$ as 
\[
    \text{Vol}_G = (-1)^{alt}\left|\Bigg(\prod_{f\in F^{new}\backslash \vec{L}}\mu(f)\Bigg)\int_{\overline{W}_G}\Omega_G\right|,
\]
where $alt$ is the number of edges of $\widetilde{E}$ which touch two new faces $f_1,f_2$ with $\mu(f_1)\mu(f_2) < 0$.

We can assign a boundary contribution graph for morphism $M\in \Hom(S^{new},S^{old})$ in the following way (See Figure \ref{fig:boundary}).
\begin{defn}\rm 
    Given a morphism
    $$
        M=(V_b\sqcup V_w, E, H^{CB},\widetilde{E},g, {\sf I},\beta',\mu,\sigma, \mu^{old})\in \Hom(S^{new}, S^{old}),
    $$
    one can assign a BC graph, denoted by $BC(M) = (V,E,\vec{L},\mu)$ as follows:
    \begin{itemize}
        \item For any wavy flag $(\tilde{e},v)$, assign a vertex.
        \item For the wavy flags $(\tilde{e},v), (\tilde{e},v')$, where $\tilde{e}$ connects two vertices $v,v'$, one assigns a wavy edge connecting the corresponding vertices.
        \item Draw the directed edges $\vec{e}$ with respect to the cycles $\sigma_v$.
        \item Assign a loop for the isolated white vertex.
        \item So far, we get a graph $G$. The new faces $f^{new}$ of $G$ correspond to the white vertices $w$ of $M$, and we assign $\mu(f^{new})=\mu(w)$.
        The old faces $f^{old}$ correspond to the cycles $c$ of type $\si^{-1}\circ\tau$, and we assign $\mu(f^{old})=\mu^{old}(c)$.
    \end{itemize}
\end{defn}
\begin{figure}
\begin{center}
\tikzset{every picture/.style={line width=0.65pt}}

\begin{tikzpicture}[x=0.65pt,y=0.65pt,yscale=-1,xscale=1]

\draw [line width=0.75] (60,171.5) .. controls (61.67,169.83) and (63.33,169.83) .. (65,171.5) .. controls (66.67,173.17) and (68.33,173.17) .. (70,171.5) .. controls (71.67,169.83) and (73.33,169.83) .. (75,171.5) .. controls (76.67,173.17) and (78.33,173.17) .. (80,171.5) .. controls (81.67,169.83) and (83.33,169.83) .. (85,171.5) .. controls (86.67,173.17) and (88.33,173.17) .. (90,171.5) .. controls (91.67,169.83) and (93.33,169.83) .. (95,171.5) .. controls (96.67,173.17) and (98.33,173.17) .. (100,171.5) .. controls (101.67,169.83) and (103.33,169.83) .. (105,171.5) .. controls (106.67,173.17) and (108.33,173.17) .. (110,171.5) .. controls (111.67,169.83) and (113.33,169.83) .. (115,171.5) .. controls (116.67,173.17) and (118.33,173.17) .. (120,171.5) ;

\draw [line width=0.75] (120,171.5) .. controls (121.67,169.83) and (123.33,169.83) .. (125,171.5) .. controls (126.67,173.17) and (128.33,173.17) .. (130,171.5) .. controls (131.67,169.83) and (133.33,169.83) .. (135,171.5) .. controls (136.67,173.17) and (138.33,173.17) .. (140,171.5) .. controls (141.67,169.83) and (143.33,169.83) .. (145,171.5) .. controls (146.67,173.17) and (148.33,173.17) .. (150,171.5) .. controls (151.67,169.83) and (153.33,169.83) .. (155,171.5) .. controls (156.67,173.17) and (158.33,173.17) .. (160,171.5) .. controls (161.67,169.83) and (163.33,169.83) .. (165,171.5) .. controls (166.67,173.17) and (168.33,173.17) .. (170,171.5) .. controls (171.67,169.83) and (173.33,169.83) .. (175,171.5) .. controls (176.67,173.17) and (178.33,173.17) .. (180,171.5) ;

\draw [line width=0.75] (120,171.5) .. controls (119.65,169.17) and (120.65,167.83) .. (122.98,167.49) .. controls (125.31,167.14) and (126.31,165.8) .. (125.97,163.47) .. controls (125.62,161.14) and (126.62,159.8) .. (128.95,159.46) .. controls (131.28,159.12) and (132.28,157.78) .. (131.93,155.45) .. controls (131.59,153.12) and (132.59,151.78) .. (134.92,151.44) .. controls (137.25,151.09) and (138.25,149.75) .. (137.9,147.42) .. controls (137.55,145.09) and (138.55,143.75) .. (140.88,143.41) .. controls (143.21,143.07) and (144.21,141.73) .. (143.86,139.4) .. controls (143.52,137.07) and (144.52,135.73) .. (146.85,135.39) -- (149.7,131.55) ;

\draw [line width=0.75] (149.7,131.55) .. controls (152.03,131.87) and (153.04,133.2) .. (152.72,135.53) .. controls (152.4,137.86) and (153.41,139.19) .. (155.74,139.52) .. controls (158.07,139.84) and (159.08,141.17) .. (158.76,143.5) .. controls (158.44,145.84) and (159.45,147.17) .. (161.79,147.49) .. controls (164.12,147.81) and (165.13,149.14) .. (164.81,151.47) .. controls (164.49,153.8) and (165.5,155.13) .. (167.83,155.45) .. controls (170.16,155.78) and (171.17,157.11) .. (170.85,159.44) .. controls (170.53,161.77) and (171.54,163.1) .. (173.87,163.42) .. controls (176.2,163.74) and (177.21,165.07) .. (176.89,167.4) .. controls (176.57,169.73) and (177.58,171.06) .. (179.91,171.39) -- (180,171.5) ;

\draw (145,126.9) node [anchor=north west][inner sep=0.75pt] {$\circ$};
\draw (85,126.9) node [anchor=north west][inner sep=0.75pt] {$\circ$};
\draw (174,166.1) node [anchor=north west][inner sep=0.75pt] {$\circ$};
\draw (115,165.9) node [anchor=north west][inner sep=0.75pt] {$\circ$};
\draw (55,165.9) node [anchor=north west][inner sep=0.75pt] {$\circ$};
\draw (47.6,174.54) node [anchor=north west][inner sep=0.75pt] {$1$};
\draw (115.2,174.54) node [anchor=north west][inner sep=0.75pt] {$2$};
\draw (144,113.34) node [anchor=north west][inner sep=0.75pt] {$4$};
\draw (84,113.34) node [anchor=north west][inner sep=0.75pt] {$5$};
\draw (175.6,173.98) node [anchor=north west][inner sep=0.75pt] {$3$};
\draw (114,228.4) node [anchor=north west][inner sep=0.75pt] {$M$};

\draw (209.8,166.88) -- (221.92,166.88) -- (221.92,164.5) -- (230,169.25) -- (221.92,174) -- (221.92,171.63) -- (209.8,171.63) -- cycle ;


  \draw [line width=0.75] (257.87,166.47) .. controls (258.52,164.2) and (259.97,163.39) .. (262.24,164.04) .. controls (264.51,164.69) and (265.96,163.89) .. (266.61,161.62) .. controls (267.26,159.35) and (268.71,158.54) .. (270.98,159.19) .. controls (273.25,159.84) and (274.71,159.04) .. (275.36,156.77) .. controls (276.01,154.5) and (277.46,153.69) .. (279.73,154.34) .. controls (282,154.99) and (283.45,154.18) .. (284.1,151.91) .. controls (284.75,149.64) and (286.2,148.84) .. (288.47,149.49) .. controls (290.74,150.14) and (292.2,149.33) .. (292.85,147.06) .. controls (293.5,144.79) and (294.95,143.99) .. (297.22,144.64) .. controls (299.49,145.29) and (300.94,144.48) .. (301.59,142.21) .. controls (302.24,139.94) and (303.69,139.14) .. (305.96,139.79) .. controls (308.23,140.44) and (309.68,139.63) .. (310.33,137.36) .. controls (310.98,135.09) and (312.44,134.29) .. (314.71,134.94) -- (315.13,134.7) ;

  \draw [line width=0.75] (358.28,149.36) .. controls (360.61,149.7) and (361.61,151.04) .. (361.28,153.37) .. controls (360.95,155.7) and (361.95,157.04) .. (364.28,157.37) .. controls (366.61,157.7) and (367.61,159.04) .. (367.27,161.37) .. controls (366.94,163.7) and (367.94,165.04) .. (370.27,165.38) .. controls (372.6,165.71) and (373.6,167.05) .. (373.26,169.38) .. controls (372.93,171.71) and (373.93,173.05) .. (376.26,173.38) .. controls (378.59,173.72) and (379.59,175.06) .. (379.25,177.39) .. controls (378.92,179.72) and (379.92,181.06) .. (382.25,181.39) .. controls (384.58,181.72) and (385.58,183.06) .. (385.24,185.39) .. controls (384.91,187.72) and (385.91,189.06) .. (388.24,189.4) -- (390,191.75) ;

  \draw [line width=0.75] (509.2,111.45) .. controls (509.78,113.74) and (508.93,115.17) .. (506.64,115.75) .. controls (504.36,116.33) and (503.51,117.76) .. (504.09,120.04) .. controls (504.67,122.33) and (503.82,123.76) .. (501.53,124.34) .. controls (499.24,124.92) and (498.39,126.35) .. (498.97,128.64) .. controls (499.55,130.92) and (498.7,132.35) .. (496.42,132.93) .. controls (494.13,133.51) and (493.28,134.94) .. (493.86,137.23) .. controls (494.44,139.52) and (493.59,140.95) .. (491.3,141.53) .. controls (489.01,142.11) and (488.16,143.54) .. (488.75,145.83) .. controls (489.33,148.12) and (488.48,149.55) .. (486.19,150.12) .. controls (483.9,150.7) and (483.05,152.13) .. (483.63,154.42) .. controls (484.22,156.71) and (483.37,158.14) .. (481.08,158.72) .. controls (478.79,159.29) and (477.94,160.72) .. (478.52,163.01) .. controls (479.1,165.3) and (478.25,166.73) .. (475.96,167.31) .. controls (473.67,167.89) and (472.82,169.32) .. (473.41,171.61) .. controls (473.99,173.9) and (473.14,175.33) .. (470.85,175.9) .. controls (468.56,176.48) and (467.71,177.91) .. (468.29,180.2) .. controls (468.88,182.49) and (468.03,183.92) .. (465.74,184.5) .. controls (463.45,185.08) and (462.6,186.51) .. (463.18,188.8) -- (463,189.1) ;

  \draw (390,191.75) .. controls (427.43,208.59) and (453.88,195.93) .. (461.5,191.97) ;
  \draw [shift={(463.26,191.05)}, rotate = 154.28] [color={rgb, 255:red, 0; green, 0; blue, 0 }][line width=0.65] (10.93,-3.29) .. controls (6.95,-1.4) and (3.31,-0.3) .. (0,0) .. controls (3.31,0.3) and (6.95,1.4) .. (10.93,3.29) ;

  \draw (463.26,191.05) .. controls (432.96,176.71) and (403.65,187.09) .. (391.73,191.16) ;
  \draw [shift={(390,191.75)}, rotate = 341.67] [color={rgb, 255:red, 0; green, 0; blue, 0 }][line width=0.65] (10.93,-3.29) .. controls (6.95,-1.4) and (3.31,-0.3) .. (0,0) .. controls (3.31,0.3) and (6.95,1.4) .. (10.93,3.29) ;

  \draw [line width=0.75] (373.93,99.19) .. controls (375.88,97.86) and (377.51,98.16) .. (378.84,100.11) .. controls (380.17,102.06) and (381.81,102.36) .. (383.76,101.03) .. controls (385.7,99.7) and (387.34,100) .. (388.67,101.94) .. controls (390,103.89) and (391.64,104.19) .. (393.59,102.86) .. controls (395.54,101.53) and (397.17,101.83) .. (398.5,103.78) .. controls (399.83,105.73) and (401.47,106.03) .. (403.42,104.7) .. controls (405.37,103.37) and (407,103.67) .. (408.33,105.62) .. controls (409.66,107.57) and (411.3,107.87) .. (413.25,106.54) .. controls (415.2,105.21) and (416.83,105.51) .. (418.16,107.46) .. controls (419.5,109.4) and (421.14,109.7) .. (423.08,108.37) .. controls (425.03,107.04) and (426.66,107.34) .. (427.99,109.29) .. controls (429.32,111.24) and (430.96,111.54) .. (432.91,110.21) .. controls (434.86,108.88) and (436.49,109.18) .. (437.82,111.13) .. controls (439.15,113.08) and (440.79,113.38) .. (442.74,112.05) -- (446.5,112.75) ;

  \draw (446.5,112.75) .. controls (476.21,129.93) and (500.72,116.84) .. (507.55,112.54) ;
  \draw [shift={(509.2,111.45)}, rotate = 149.84] [color={rgb, 255:red, 0; green, 0; blue, 0 }][line width=0.65] (10.93,-3.29) .. controls (6.95,-1.4) and (3.31,-0.3) .. (0,0) .. controls (3.31,0.3) and (6.95,1.4) .. (10.93,3.29) ;

  \draw (509.2,111.45) .. controls (478.96,97.27) and (458.1,106.77) .. (448.13,111.9) ;
  \draw [shift={(446.5,112.75)}, rotate = 332.45] [color={rgb, 255:red, 0; green, 0; blue, 0 }][line width=0.65] (10.93,-3.29) .. controls (6.95,-1.4) and (3.31,-0.3) .. (0,0) .. controls (3.31,0.3) and (6.95,1.4) .. (10.93,3.29) ;

  \draw (373.93,99.19) .. controls (346.36,97.34) and (321.16,106.65) .. (315.46,133.06) ;
  \draw [shift={(315.13,134.7)}, rotate = 280.24] [color={rgb, 255:red, 0; green, 0; blue, 0 }][line width=0.65] (10.93,-3.29) .. controls (6.95,-1.4) and (3.31,-0.3) .. (0,0) .. controls (3.31,0.3) and (6.95,1.4) .. (10.93,3.29) ;

  \draw (358.28,149.36) .. controls (376.59,138.34) and (375.22,111.88) .. (374.13,101.06) ;
  \draw [shift={(373.93,99.19)}, rotate = 83.64] [color={rgb, 255:red, 0; green, 0; blue, 0 }][line width=0.65] (10.93,-3.29) .. controls (6.95,-1.4) and (3.31,-0.3) .. (0,0) .. controls (3.31,0.3) and (6.95,1.4) .. (10.93,3.29) ;

  \draw (315.13,134.7) .. controls (324.54,141.08) and (347.98,147.35) .. (356.44,149.03) ;
  \draw [shift={(358.28,149.36)}, rotate = 188.58] [color={rgb, 255:red, 0; green, 0; blue, 0 }][line width=0.65] (10.93,-3.29) .. controls (6.95,-1.4) and (3.31,-0.3) .. (0,0) .. controls (3.31,0.3) and (6.95,1.4) .. (10.93,3.29) ;

  \draw (257.87,166.47) .. controls (293.16,196.78) and (328.1,170.44) .. (259.36,165.78) ;
  \draw [shift={(258.32,165.71)}, rotate = 3.64] [color={rgb, 255:red, 0; green, 0; blue, 0 }][line width=0.65] (10.93,-3.29) .. controls (6.95,-1.4) and (3.31,-0.3) .. (0,0) .. controls (3.31,0.3) and (6.95,1.4) .. (10.93,3.29) ;

  \draw (269.5,117.5) circle (11);
  
  \draw [{Stealth[length=5.5pt, width=3.5pt, inset=1.5pt, bend]}-] (269.5,117.5) ++(300:11) arc (300:350:11);

  \draw (252.4,159.8) node [anchor=north west][inner sep=0.65pt] {$\bullet$};
  \draw (251.4,171.8) node [anchor=north west][inner sep=0.65pt] {$\tilde{e}_{12}$};
  \draw (310.97,127.95) node [anchor=north west][inner sep=0.65pt] [rotate=-19.5] {$\bullet$};
  \draw (295.73,106.38) node [anchor=north west][inner sep=0.65pt] {$\tilde{e}_{21}$};
  \draw (355.73,143.6) node [anchor=north west][inner sep=0.65pt] [rotate=-19.5] {$\bullet$};
  \draw (338.77,152.18) node [anchor=north west][inner sep=0.65pt] {$\tilde{e}_{23}$};
  \draw (371.48,93.13) node [anchor=north west][inner sep=0.65pt] [rotate=-19.5] {$\bullet$};
  \draw (369.79,75.1) node [anchor=north west][inner sep=0.65pt] {$\tilde{e}_{24}$};
  \draw (384.9,186.6) node [anchor=north west][inner sep=0.65pt] {$\bullet$};
  \draw (376.9,196.3) node [anchor=north west][inner sep=0.65pt] {$\tilde{e}_{32}$};
  \draw (456.9,186.6) node [anchor=north west][inner sep=0.65pt] {$\bullet$};
  \draw (465.4,197.1) node [anchor=north west][inner sep=0.65pt] {$\tilde{e}_{34}$};
  \draw (441.4,107.3) node [anchor=north west][inner sep=0.65pt] {$\bullet$};
  \draw (441.9,118.3) node [anchor=north west][inner sep=0.65pt] {$\tilde{e}_{42}$};
  \draw (503.4,107.3) node [anchor=north west][inner sep=0.65pt] {$\bullet$};
  \draw (511.2,114.85) node [anchor=north west][inner sep=0.65pt] {$\tilde{e}_{43}$};
  \draw (368,228.4) node [anchor=north west][inner sep=0.65pt] {$BC(M)$};
  \draw (262,88.0) node [anchor=north west][inner sep=0.65pt] {$l_{5}$};

\end{tikzpicture}
\end{center}
\caption{boundary contribution graph}\label{fig:boundary}
\begin{minipage}{\textwidth}
    \raggedright
    The left hand side is an example of morphism between moduli specification (just show white vertices and wavy edges); the right hand side is the associated boundary contribution graph. Here $\tilde{e}_{ij}$ is the assigned vertex of the wavy flag from $i$ to $j$.
    $l_5$ is the loop corresponding to the isolated vertex 5.
\end{minipage}
\end{figure}

Now we state the combinatorial formula in \cite{BNPT22}.
\begin{defn}\rm\label{def:comb-OGW}
    Let $S$ be a moduli specification, $\vec{a},\vec{\gamma}$ be vectors of nonnegative integers and cohomology classes in $H^*_T(\CP)$. Define the \emph{combinatorial} open Gromov-Witten invariants of $(S,\vec{a},\vec{\gamma})$ as:
    $$
        \text{OGW}(S,\vec{a},\vec{\gamma}) = \sum_{S'\text{ is pure}}\sum_{M\in \Hom(S',S)} \frac{1}{|\Aut(M)|}\frac{\text{Vol}_{BC(M)}}{(-\sv)^{|\widetilde{E}|}}\left(\prod_{h\in H^{CB}(M)}\frac{d(h)}{\sv}\right)I(S',\vec{a},\vec{\gamma}).
    $$
\end{defn} 
The main conjecture in \cite{BNPT22} states that
\begin{conjecture}
    There exists a geometric definition for $\text{OGW}(S,\vec{a},\vec{\gamma})$ for all moduli specifications $S$.
\end{conjecture}
\begin{remark}
   It is proved in \cite{BNPT22} that the coherent boundary condition serves as the geometric definition of $\text{OGW}(S,\vec{a},\vec{\gamma})$ in genus 0.
\end{remark}

\subsection{Generating function of geometric counting}\label{sec:gen-geo}
Let $S=(V_b \sqcup V_w, E, g, {\sf{I}}, \beta',\mu)$ be a moduli specification, $\vec{a}=(a_i)_{i\in{\sf I}}$ and $\vec{\gamma}=(\gamma_i)_{i\in{\sf I}}$.
We could change the decoration of $S$ in the following two ways:
\begin{itemize}
    \item [(1)] (add marked points) For every $l\in\bZ_{\geq 1}$, we define $S^{[l]}$ by extending the set of labels on the black vertices to ${\sf I}^{[l]}= {\sf I}\sqcup \{1,\dots,l\}$, while keeping all other decorations unchanged.
    
    The vector $\vec{a}^{[l]}$ is extended accordingly by setting $\vec{a}^{[l]}_i = a_i$ for $i\in {\sf I}$ and $\vec{a}^{[l]}_i=0$ for $i\in\{1,\dots,l\}$.
    The vector $\vec{\gamma}^{[l]}$ is extended by setting $\vec{\gamma}^{[l]}=\gamma_i$ for $i\in {\sf I}$ and $\vec{\gamma}^{[l]}_i=\bt$ for $i\in\{1,\dots,l\}$, where $\bt\in H^*_T(\CP)$.

    \item [(2)] (forget degree) For a black vertex $v\in V_b$, its decoration $d_v = d(v)$ corresponding to the curve degree $\beta \in H_2(\bP^1)$ is given by
    $$
        d(v) = d_+(v) - \sum_{w\in V_w(v) : \mu(w)>0} \mu(w),
    $$
    where $V_w(v)$ is the set of white vertices attached to $v$.

    By forgetting the decoration $\beta'$, we obtain from $S$ a decorated graph $$\Ga_S = (V_b \sqcup V_w, E, g, {\sf{I}}, \mu).$$ 
    We say $\Ga_S$ is \emph{pure} if $\mu(w)\neq 0$ for each white vertex $w$.

    Conversely, given a vector $(d_v)_{v\in V_b}$, we define 
    $\Ga_S[d_v]$ as the moduli specification by assigning $\beta'$ as follows
    \begin{equation}\label{eqn:vec-degree}
        \qquad d_+(v) = d_v+\sum_{w\in V_w(v):\mu(w)>0} \mu(w),\quad d_-(v) = d_v - \sum_{w\in V_w(v):\mu(w)<0} \mu(w).
    \end{equation}
 
    Different vectors $(d_v)$ may yield the same moduli specification (i.e. indistinguishable at the level of the decorated graph $S$). We therefore define 
    $[d_v]$ as the equivalence class of all such $(d_v)$ that determine the same moduli specification.
\end{itemize}

Assume that $S$ is a \emph{connected} moduli specification, 
we would like to define the generating function of geometric counting as:
$$
    \textOGW(\Ga_S,\vec{a},\vec{\gamma},\bt) := \sum_{d\geq 0}\sum_{l\geq 0} \frac{1}{l!}\textOGW(\Ga_S[d]^{[l]},\vec{a}^{[l]},\vec{\gamma}^{[l]}),
$$
where $\Ga_S[d]^{[l]}$ is the moduli specification by adding the degree $d$ and extending $l$ marked labels on the unique black vertex. The dependence of $\textOGW(\Ga_S,\vec{a},\vec{\gamma},\bt)$ on $\bt$
is reflected in $\vec{\gamma}^{[l]}$, where $\bt$ is assigned to $\{1,\dots,l\}$.

Furthermore, assume $S$ is \emph{pure} ($\mu$ is non-zero on white vertices), we would like to use the combinatorial formula in Definition \ref{def:comb-OGW} to express $\textOGW(\Ga_S,\vec{a},\vec{\gamma},\bt)$. 

We introduce the following notations:
\begin{itemize}
    \item For two moduli specifications $S', S$, we can construct a morphism $\Ga_M \in \Hom(\Ga_{S'},\Ga_S)$ by dropping the condition on the equality of $\beta'(v)$, following the pattern of Definition \ref{def:mor-decoration}. 
    \item For a morphism $M\in \Hom(S',S)$, the association of the boundary contribution graph $BC(M)$ is independent of the decoration $d_v$ on the black vertices of $M$, so we can also define the boundary contribution graph for $\Ga_{M}\in \Hom(\Ga_{S'},\Ga_S)$, denoted by $BC(\Ga_M)$.
    \item Let $v$ be a black vertex of $\Ga_M$, and let $g_v, {\sf I}_v, V_w(v)$ be the genus, marked labels, the set of white vertices attached to $v$, respectively. We define
    $$I(\vec{a}|_v,\vec{\gamma}|_v,\bt) = [\prod_{w_j\in V_w(v)}X_j^{\mu(w_j)}]\llangle \gamma_i\psi^{a_i}|_{i\in {\sf I}_v} \rrangle_{g_v,|{\sf I}_v|, |V_w(v)|}^{\OGW}.$$
\end{itemize}

Then we have the following theorem.
\begin{theorem}\label{thm:A-geo-count}\rm 
    Let $S$ be a pure connected moduli specification. Then we have
    \begin{align*}
        \textOGW(\Ga_S,\vec{a},\vec{\gamma},\bt)= \sum_{\Ga_{S'}\text{ is pure}}\sum_{\Ga_M \in \Hom(\Ga_{S'},\Ga_S)}\frac{1}{|\Aut(\Gamma_M)|}\frac{\text{Vol}_{BC(\Gamma_M)}}{(-\sv)^{|\widetilde{E}|}}
        \prod_{v\in V_b(M)} I(\vec{a}|_v,\vec{\gamma}|_v,\bt).
    \end{align*}
\end{theorem}

\begin{proof}
The proof follows from the enumeration of $M\in \Hom(S',S)$ and the Orbit-Stabilizer theorem.

In \emph{Step 1}, we will sum over the extended $l$ marked labels.

When $S$ is pure, $M\in \Hom(S',S)$ does not have the contracted boundary half-edges $H^{CB}$. The combinatorial formula states that 
\begin{align*}
    \textOGW(S^{[l]},\vec{a}^{[l]},\vec{\gamma}^{[l]}) = \sum_{S'_{l}\text{ is pure}} \sum_{M_{l}\in \Hom(S'_{l},S^{[l]})} \frac{\text{Vol}_{BC(M_{l})}}{(-\sv)^{|\widetilde{E}|}} \frac{I(S'_{l},\vec{a}^{[l]},\vec{\gamma}^{[l]})}{|\Aut(M_{l})|}.
\end{align*}

For $M_l \in \Hom(S'_l,S^{[l]})$, there exists a forgetful map $$\forget_l: \Hom(S'_l,S^{[l]})\rightarrow \Hom(S',S)$$ defined by forgetting $\{1,\dots,l\}$ in ${\sf I}^{[l]}$.
Let $M = \forget_l(M_l)$. 

Let $V_b(M)$ be the set of black vertices of $M$, and let $v^{[l]}=(v_1,\dots, v_l)$ be the vector of black vertices decorated by $i\in\{1,\dots ,l\}$.
There is a canonical $\Aut(M)$-action on $V_b(M)^{l}$ (\emph{length-$l$ vectors of $V_b(M)$}) defined as
$$\si\cdot (v_1,\dots,v_l) := (\si(v_1),\dots,\si(v_l)), \quad \si\in\Aut(M).$$
 
By definition, the automorphism group $\Aut(M_l)$ is a subgroup of $\Aut(M)$ that fixes black vertices labeled by $\{1,\dots,l\}$.
In other words, $\Aut(M_l)$ is the stabilizer of $v^{[l]}$. By the orbit-stabilizer theorem, we have
$$
    \frac{|\Aut(M)|}{|\Aut(M_l)|} = |\text{Orb}(v^{[l]})|.
$$
Two assignments $v_1^{[l]}, v_2^{[l]}$ on $M$ gives the same decoration $M_l\in \forget^{-1}_l(M)$ if and only if $v_1^{[l]}, v_2^{[l]}$ are in the same orbit of $\Aut(M)$.
In other words, there is a one-to-one correspondence between the set $\forget^{-1}_l(M)$ and the set of orbits $V_b(M)^l/\Aut(M)$. 

Moreover, $\text{Vol}_{BC(M_l)}$, $|\widetilde{E}|$ are independent of the decoration on black vertices.
Hence,
\begin{align*}
    \textOGW(S^{[l]},\vec{a}^{[l]},\vec{\gamma}^{[l]}) = \sum_{S'\text{ is pure} \atop {M\in \Hom(S',S)}} \frac{\text{Vol}_{BC(M)}}{(-\sv)^{|\widetilde{E}|}}\sum_{M_l\in \forget_l^{-1}(M)}\frac{I(S'_{l}(M_l),\vec{a}^{[l]},\vec{\gamma}^{[l]})}{|\Aut(M_l)|}&
    \\
    = \sum_{S'\text{ is pure}\atop {M\in \Hom(S',S)}} \frac{1}{|\Aut(M)|}\frac{\text{Vol}_{BC(M)}}{(-\sv)^{|\widetilde{E}|}}
    \sum_{v^{[l]}\in V_b(M)^{l}}I(S'_{v^{[l]}},\vec{a}^{[l]},\vec{\gamma}^{[l]})&,
\end{align*}
where $S'_{v^{[l]}}$ is the morphism obtained by labelling $\{1,\dots,l\}$ to $v_i$'s in $S'$. 
Moreover,
we have
\begin{align*}
    &\sum_{l\geq 0}\frac{1}{l!}\textOGW(S^{[l]},\vec{a}^{[l]},\vec{\gamma}^{[l]}) 
    \\
    &= \sum_{S' \text{ is pure}\atop {M\in \Hom(S',S)}}
    \frac{1}{|\Aut(M)|}\frac{\text{Vol}_{BC(M)}}{(-\sv)^{|\widetilde{E}|}}
    \prod_{v\in V_b(M)} I(\vec{a}|_v,\vec{\gamma}|_v,d_v,\bt),
\end{align*}
where $$
    I(\vec{a}|_v,\vec{\gamma}|_v,d_v,\bt) := \sum_{l\geq 0}\frac{1}{l!}\langle \gamma_i\psi^{a_i}|_{i\in {\sf I}_v},\bt^l\rangle_{g_v,\beta'(v),\vec{\mu}(v)}^{\OGW},
$$
$\beta'(v)$ is given by \eqref{eqn:vec-degree}, and $\vec{\mu}(v)$ is the vector of the values of $\mu$ on the white vertices attached to $v$.

~

In \emph{Step 2}, we sum over all curve degree $d$. 

Let $\Ga_M \in \Hom(\Ga_{S'},\Ga_S)$ and $(d_v)$ be the vector of decoration of curve degree on the black vertex $V_b(M)$.
We write $\Ga_M[d_v] \in \Hom(\Ga_{S'}[d_v],\Ga_S[d])$, where $d$ is determined by $\beta'=\sum_{v\in V_b(\Ga_M)}\beta'(v)$.

Let $m=|V_b(\Ga_M)|$, and let $(v_1,\dots,v_m)$ be the vector of black vertices of $\Ga_M$.
There is an $\Aut(\Ga_M)$ action on $(d_{v_1},\dots,d_{v_m})$: 
$$\si\cdot (d_{v_1},\dots,d_{v_m}):= (d_{\si^{-1}(v_1)},\dots,d_{\si^{-1}(v_m)}), \quad \si\in\Aut(\Ga_M).$$
Then $\Aut(\Ga_M[d_v]) = \text{Stab}_{\Aut(\Ga_M)}(d_v)$. Then the Orbit-Stabilizer theorem tells that
$$
    \frac{|\Aut(\Ga_M)|}{|\Aut(\Ga_M[d_v])|} = |\text{Orb}((d_v))|.
$$
Let $[d_v]$ be the equivalence classes of $(d_v)$ under the action $\Aut(\Ga_M)$, which are in 1-1 correspondence to the morphisms $\Ga_M[d_v] \in \Hom(\Ga_{S'}[d_v],\Ga_S[d])$.
We have
\begin{align*}
    &\textOGW(\Gamma_S,\vec{a},\vec{\gamma},\bt) 
    \\
    &= \sum_{\Ga_{S'}\text{ is pure}}\sum_{\Ga_M \in \Hom(\Ga_{S'},\Ga_S)}\frac{\text{Vol}_{BC(\Gamma_M)}}{(-\sv)^{|\widetilde{E}|}}\sum_{[d_v]}
    \frac{1}{|\Aut(\Gamma_M[d_v])|}
    \prod_{v\in V_b(M)} I(\vec{a}|_v,\vec{\gamma}|_v,d_v,\bt)
    \\
    &= \sum_{\Ga_{S'}\text{ is pure}}\sum_{\Ga_M \in \Hom(\Ga_{S'},\Ga_S)}\frac{1}{|\Aut(\Gamma_M)|}\frac{\text{Vol}_{BC(\Gamma_M)}}{(-\sv)^{|\widetilde{E}|}}\sum_{d_v}
    \prod_{v\in V_b(M)} I(\vec{a}|_v,\vec{\gamma}|_v,d_v,\bt)
    \\
    &= \sum_{\Ga_{S'}\text{ is pure}}\sum_{\Ga_M \in \Hom(\Ga_{S'},\Ga_S)}\frac{1}{|\Aut(\Gamma_M)|}\frac{\text{Vol}_{BC(\Gamma_M)}}{(-\sv)^{|\widetilde{E}|}}
    \prod_{v\in V_b(M)} I(\vec{a}|_v,\vec{\gamma}|_v,\bt),
\end{align*}
which completes the proof.
\end{proof}

\subsection{B-model geometric counting}\label{sec:B-geo-count}
Let $W_{g,m,\vec{\mu}}(\bt,\cE_i)$ be the coefficients of the Laurent series $W_{g,m,n}(t^0,q;\cE_i;X_j)$ (see \eqref{eqn:SYZ-B-potential}):
$$
    W_{g,m,n}(t^0, q;\cE_1,\dots,\cE_m;{X}_1,\dots,{X}_n) = \sum_{\vec{\mu}=(\mu_1,\dots,\mu_n)\atop {\mu_i\in\bZ_{\neq 0}}} W_{g,m,\vec{\mu}}(\bt,\cE_i)X_1^{\mu_1}\dots X_n^{\mu_n}.
$$
Let $\Ga_S = (V_b \sqcup V_w, E, g, {\sf{I}}, \mu)$ be a connected decorated graph with one black vertex, $n$ white vertices and $m$ marked labels ${\sf I}$.
We define a graph sum based on the summation over $\Ga_M\in \Hom(\Ga_{S'},\Ga_S)$. On the marked label ${\sf I}$, we assign the classes $\frac{\kappa_{z_i}(\cE_i)}{z_i-\psi_i}$, $i\in {\sf I}$.
In A-model, we define $\textOGW_{g,m,n}(\Ga_S, \frac{\kappa_{z_i}(\cE_i)}{z_i-\psi_i},\bt)$ as a generating function of open geometric counting with $\vec{a},\vec{\gamma}$ corresponding to the class $\frac{\kappa_{z_i}(\cE_i)}{z_i-\psi_i}$:
\begin{align*}
    \textOGW_{g,m,n}(\Ga_S, \frac{\kappa_{z_i}(\cE_i)}{z_i-\psi_i},\bt) := \sum_{\vec{a}=(a_1,\dots,a_m)\atop{a_i\in\bZ_{\geq 0}}} \textOGW(\Ga_S,\vec{a},(\kappa_{z_i}(\cE_i))_{i\in {\sf I}},\bt)\prod_{i=1}^m z_i^{-a_i-1}.
\end{align*}
In B-model, we define
\begin{align*}
    G_{g,m,n}(\Ga_S,\bt)= \sum_{\Ga_{S'}\text{ is pure}}&\sum_{\Ga_M \in \Hom(\Ga_{S'},\Ga_S)}\frac{1}{|\Aut(\Gamma_M)|}\frac{\text{Vol}_{BC(\Gamma_M)}}{(-\sv)^{|\widetilde{E}|}}
    \\
    &\times \prod_{v\in V_b(M)} W_{g_v, |{\sf I}_v|,\vec{\mu}(v)}(\bt, \cE_i|i\in {\sf I}_v),
\end{align*}

By Theorem \ref{thm:SYZ-open-mirror} and Theorem \ref{thm:A-geo-count}, we have
\begin{theorem}\label{thm:geo-MS}
    \rm Let $S$ be a pure connected moduli specification with one black vertex, $n$ white vertices and $m$ marked labels ${\sf I}$, it holds that
    $$
        \textOGW_{g,m,n}(\Ga_S, \frac{\kappa_{z_i}(\cE_i)}{z_i-\psi_i},\bt) = G_{g,m,n}(\Ga_S,\bt).
    $$
\end{theorem}

\appendix
\section{Bessel functions}\label{sec:Bessel}
The special function $I_\alpha(x)$ in $J$-function is the modified Bessel function of the first kind. It is defined as
\[
    I_\alpha(x) = \sum_{m=0}^{\infty} \frac{1}{m!\Gamma(m+\alpha+1)}(\frac{x}{2})^{2m+\alpha}.
\]
The derivative of $I_\alpha(x)$ is given by
\[
    I'_\alpha(x) = I_{\alpha+1}(x) + \frac{\alpha}{x}I_\alpha(x).
\]
Let $d\in\bZ_{\neq 0}$ and $q,\sv>0$, we get 
\[
    I'_d\Big(\frac{2\sqrt{q}d}{\sv}\Big) = I_{d+1}\Big(\frac{2\sqrt{q}d}{\sv}\Big) + 
    \frac{\sv}{2\sqrt{q}}I_d\Big(\frac{2\sqrt{q}d}{\sv}\Big).
\]
For $n\in\mathbb{N}$, 
\[
    I_n(x) = I_{-n}(x).
\]
We have $$I_0(0)= 1, \quad I_\alpha(0)= 0 \quad(\alpha > 0).$$

\section{Residue Calculation}\label{sec:residue}
Let $\mu\in\bZ$, $q,\sv > 0$, consider
\[
    C_\mu := \Res_{Y=0} e^{\mu(Y+\frac{q}{Y})/\sv}\frac{dY}{Y^{\mu+1}},
\]
\[
    D_\mu := \Res_{Y=0} e^{\mu(Y+\frac{q}{Y})/\sv}\frac{dY}{Y^{\mu+2}}.
\]
To calculate $C_\mu$ and $D_\mu$, we expand
\[
    e^{\mu(Y+q/Y)/\sv} = \sum_{m,n = 0}^\infty \frac{(\frac{\mu}{\sv})^n(\frac{\mu q}{\sv})^m Y^{n-m}}{n!m!}.
\]
For $C_\mu,\ \mu\geq 0$, let $n=m+\mu$,
\begin{equation*}
    \begin{aligned}
       C_\mu &= \sum_{m=0}^{\infty} \frac{(\frac{\mu}{\sv})^{m+\mu}(\frac{\mu q}{\sv})^m}{(m+\mu)!m!}
    \\
    &= 
    \frac{1}{\sqrt{q}^{\mu}}\sum_{m=0}^{\infty}\frac{1}{m!(m+\mu)!}\Big(\frac{\mu\sqrt{q}}{\sv}\Big)^{2m+\mu}
    \\
    &= 
    \frac{1}{\sqrt{q}^\mu}I_\mu\Big(\frac{2\sqrt{q}\mu}{\sv}\Big). 
    \end{aligned}
\end{equation*}
For $C_\mu,\ \mu <0$, let $m=n-\mu$,
\begin{equation*}
    \begin{aligned}
        C_\mu &= \sum_{n=0}^{\infty} \frac{(\frac{\mu}{\sv})^n(\frac{\mu q}{\sv})^{n-\mu}}{n!(n-\mu)!}
        \\
        &= \frac{1}{\sqrt{q}^\mu}\sum_{n=0}^{\infty}\frac{1}{n!(n-\mu)!}\Big(\frac{\mu\sqrt{q}}{\sv}\Big)^{2n-\mu} 
        \\
        &= \frac{1}{\sqrt{q}^\mu}I_{-\mu}\Big(\frac{2\sqrt{q}\mu}{\sv}\Big) = 
        \frac{1}{\sqrt{q}^\mu}I_{\mu}\Big(\frac{2\sqrt{q}\mu}{\sv}\Big).
    \end{aligned}
\end{equation*}
For $D_\mu,\ \mu\geq -1$, let $n = m+\mu+1$,
\begin{equation*}
    \begin{aligned}
        D_\mu &= \sum_{m=0}^{\infty}\frac{(\frac{\mu}{\sv})^{m+\mu+1}(\frac{\mu q}{\sv})^m}{(m+\mu+1)!m!}
        \\
        &= \frac{1}{\sqrt{q}^{\mu+1}}\sum_{m=0}^{\infty}\frac{1}{(m+\mu+1)!m!}\Big(\frac{\mu\sqrt{q}}{\sv}\Big)^{2m+\mu+1}
        \\
        &= \frac{1}{\sqrt{q}^{\mu+1}}I_{\mu+1}\Big(\frac{2\sqrt{q}\mu}{\sv}\Big).
    \end{aligned}
\end{equation*}
For $D_\mu,\ \mu<-1$, let $m = n-\mu-1$,
\begin{equation*}
    \begin{aligned}
        D_\mu &= \sum_{n=0}^{\infty}\frac{(\frac{\mu}{\sv})^n(\frac{\mu q}{\sv})^{n-\mu-1}}{n!(n-\mu-1)!}
        \\
        &= \frac{1}{\sqrt{q}^{\mu+1}}\sum_{n=0}^{\infty}\frac{1}{n!(n-\mu-1)!}\Big(\frac{\mu\sqrt{q}}{\sv}\Big)^{2n-\mu-1}
        \\
        &= \frac{1}{\sqrt{q}^{\mu+1}}I_{-\mu-1}\Big(\frac{2\sqrt{q}\mu}{\sv}\Big) 
        = \frac{1}{\sqrt{q}^{\mu+1}}I_{\mu+1}\Big(\frac{2\sqrt{q}\mu}{\sv}\Big).
    \end{aligned}
\end{equation*}

\end{document}